\documentclass[11pt,letterpaper]{article}

\usepackage[T1]{fontenc}
\usepackage[utf8]{inputenc}
\usepackage[margin=1in]{geometry}
\usepackage{amsmath,amsthm,amsfonts,mathtools}
\usepackage{newtxtext,newtxmath,bm}
\usepackage{tgcursor}
\usepackage{hoang-or}
\usepackage{microtype}
\usepackage{xcolor}
\usepackage{graphicx}
\usepackage{caption,subcaption}
\usepackage{booktabs,multirow,pbox}
\usepackage{enumitem}
\usepackage{algorithm,algpseudocode}
\usepackage{multicol,lscape}
\usepackage{thmtools,thm-restate}
\usepackage{mathalfa}
\usepackage[round,authoryear]{natbib}
\usepackage[hidelinks]{hyperref}
\usepackage{cleveref}

\hypersetup{
  pdftitle={Finite-Time Behavior of the Erlang-C Model: Mixing Time, Mean Queue Length and Tail Bounds},
  pdfauthor={Hoang Huy Nguyen, Sushil Mahavir Varma, Siva Theja Maguluri},
  pdfkeywords={Queuing Theory, Birth-and-Death Processes, Markov Chain Mixing}
}
\allowdisplaybreaks[1]
\newtheorem{theorem}{Theorem}
\newtheorem{lemma}{Lemma}
\newtheorem{corollary}{Corollary}
\newtheorem{claim}{Claim}
\newtheorem{assumption}{Assumption}
\newcommand{\rate}{\mathfrak{Rate}}
\newcommand{\pwait}{\ensuremath{p_{\mathrm{no-idle}}}}

\renewcommand{\proof}[1]{%
  \par\addvspace{\topsep}\noindent\textit{#1}\enspace\ignorespaces}
\renewcommand{\endproof}{\par\addvspace{\topsep}}

\newcommand{\TABLE}[3]{%
  \caption{#1}%
  \begingroup\small\setlength{\tabcolsep}{3pt}%
  \resizebox{\linewidth}{!}{#2}\par#3\endgroup}

\title{Finite-Time Behavior of the Erlang-C Model:\\
Mixing Time, Mean Queue Length and Tail Bounds}
\author{%
  Hoang Huy Nguyen\\
  \small Georgia Institute of Technology\\
  \small\href{mailto:hnguyen455@gatech.edu}{\texttt{hnguyen455@gatech.edu}}
  \and
  Sushil Mahavir Varma\\
  \small University of Michigan\\
  \small\href{mailto:sushilv@umich.edu}{\texttt{sushilv@umich.edu}}
  \and
  Siva Theja Maguluri\\
  \small Georgia Institute of Technology\\
  \small\href{mailto:siva.theja@gatech.edu}{\texttt{siva.theja@gatech.edu}}
}
\date{}

\begin{document}
\maketitle
\begin{abstract}
Resource allocation problems in service systems like data centers and ride-hailing are usually studied using queueing models. Such systems are primarily studied in the steady-state and in asymptotic regimes such as under heavy traffic due to their analytical tractability. However, almost all applications in real life do not operate in asymptotic regimes, and so, there is a clear discrepancy in translating theoretical queuing results to practical applications. In this work, we bridge this gap by presenting nonasymptotic and finite-time bounds for Erlang-C systems (also known as $M/M/n$ queues), providing a stepping stone towards understanding the transient behavior of more general queuing systems. We bound the Chi-square distance between the finite-time queue length distribution and the stationary distribution, show that it decays exponentially fast, and characterize the rate of decay. We observe that the Erlang-C system exhibits a phase transition, depending on a parameter that measures the load relative to the size of the system. We then use these results to obtain bounds on the mean queue length and tails of the queue lengths in finite time for the nonasymptotic system. We also establish that the rate we obtain is tight up to universal constants in appropriate heavy-traffic asymptotic regimes.

We obtain these results using the Lyapunov-Poincar\'e approach, where we first carefully design a Lyapunov function to obtain a negative drift outside a finite set. Within the finite set, we develop different strategies depending on the properties of the finite set to get a handle on the mixing behavior via a local Poincar\'e inequality. A key aspect of our methodological contribution is obtaining tight guarantees in these two regions, which when combined, give us tight mixing time bounds. We believe that this approach is of independent interest for studying mixing in reversible countable-state Markov chains more generally.
\end{abstract}

\section{Introduction}

{

Many service systems in practice, such as data centers \citep{Stolyar2013InfiniteServerPacking,MendelsonXu2026SparseCommunication}, communication networks \citep{prakirt-zubeldia2025quantum-switch}, and ride-hailing systems \cite{BesbesCastroLobel2022SpatialCapacity,varma2023dynamic}, can typically be modeled as queuing systems. Previously, these systems were mostly studied in the steady state as an approximation of the transient system since the analyses are more tractable in the equilibrium. However, practical systems operate in finite time and with a finite number of servers, and so, there is an apparent gap between theoretical analysis and practical real-time performance guarantees. Moreover, for most scenarios, practitioners need to respond to real-time changes to the systems, such as sudden demand surges \citep{HuangSongTong2016,chen2025optimalRoutingSurges}. Additionally, an emerging line of research is to develop data-driven learning algorithms for queuing systems \citep{ZhongBirgeWard2024LearningSchedule,vankempen2024learningpayoffsroutingskillbased, LiuWardZhang2026DataDrivenMatching}, which requires an understanding of the transient behavior of these queuing systems. Therefore, getting a finite-time handle on the wait time and tail latency, which can be estimated using moment bounds (e.g., mean queue length) and the distributional properties of the queue length distribution, is vital for instant decision-making.

On the other hand, oftentimes it is very difficult or even impossible to obtain the explicit, closed-form steady-state distribution of many queuing systems. Therefore, one resorts to asymptotic analysis, such as heavy traffic analysis \citep{eryilmaz2013asymptotically,Transform_Lange_2020} or diffusion approximation \citep{Ward2003-diffusion-abandonment,braverman2017steins}. Intuitively, the closer we are to the asymptotic regime, the better these asymptotic approximations become. On the other hand, prior works suggest that the queuing systems mix more slowly as one gets closer to the asymptotic regimes, such as the Heavy Traffic regime \cite{philippe-robert-stochastic-networks-and-queues,Pang_2007_martingale_proof_survey}. Thus, this poses a dilemma where a system has a larger steady-state approximation error if it is far from the asymptotic regime, while it takes a long time to get close to the steady-state if it is close to the asymptotic regime. Hence, it is unclear whether the asymptotic regimes are sufficiently good approximations of the transient system, and this raises a need for tight quantitative bounds between the finite-time queue length distribution and the steady-state distribution. A competitive quantitative handle on the finite-time queue length distribution would yield a complete picture of the finite-time statistics, providing us not only moment bounds but also tail bounds, which are particularly important when the operational objective is to control rare but severe congestion events rather than only average performance \citep{dean-barroso-tail-at-scale,ScullyVanKreveld2025GittinsTail}. This motivates a finite-time distributional analysis that can simultaneously quantify convergence of the full queue length distribution, transient moments, and transient tail probabilities.

}

\subsection{Contributions}
\label{ssec:contributions}
{

In this work, we consider the Erlang-C system with $n$ servers whose service times are exponentially distributed with mean $1$ and the arrival rate $\lambda \in (0,n)$. We provide a quantitative bound on the distance between the time $t$ queue length distribution (denoted as $\pi_{n,t}$) and the steady-state queue length distribution (denoted as $\nu_n$) in the form of
\begin{align}
    \label{eqn: mixing-form}
    \chi\br{\pi_{n,t}, \nu_n} \leq e^{-\rate \, \times \, t} \chi\br{\pi_{n,0}, \nu_n}
\end{align}
where $\rate$ is the largest value such that this equation is true and denotes the mixing rate of the system. Here $\chi$ is the square root of the Chi-square distance, which is defined in Subsection \ref{ssec: chi-square-distance}. A non-asymptotic lower bound of $\rate$ is described in Theorem \ref{thm: unifying-theorem}. As a special case, we also recover the rate of $M/M/1$ (Proposition \ref{thm: mm1-mixing}) and $M/M/\infty$ systems (Proposition \ref{thm: mminf-mixing}).

It turns out that the mixing rate of the $M/M/n$ queue exhibits a phase transition-like behavior depending on the load, which is encapsulated in the traffic parameter $\alpha$ that is defined in \eqref{eqn: alpha-definition}. When $0<\alpha <1/2$ or $\lambda < n-\sqrt{n}$, corresponding to a relatively low load, the system behaves like an $M/M/\infty$ system and mixes quickly as the rate is approximately $1$. This is because most arriving jobs see an idle server. On the other hand, when $\alpha > 1/2$ or $\lambda > n-\sqrt{n}$, which corresponds to a larger load and leads to many jobs having to wait, the system behaves similarly to an $M/M/1$ queue and mixes slowly. For a non-asymptotic system with the arrival rate close to the phase transition point $n-\sqrt{n}$, it may exhibit both behaviors, and so its mixing rate can be somewhere in between. We visualize this mixing rate behavior in Figure \ref{fig: erlang-c-rate-and-mean-vs-lambda} and summarize the results in Table \ref{tab:mixing_time}.

\begin{figure}[!hbt]
    \centering
    \includegraphics[width=0.8\linewidth]{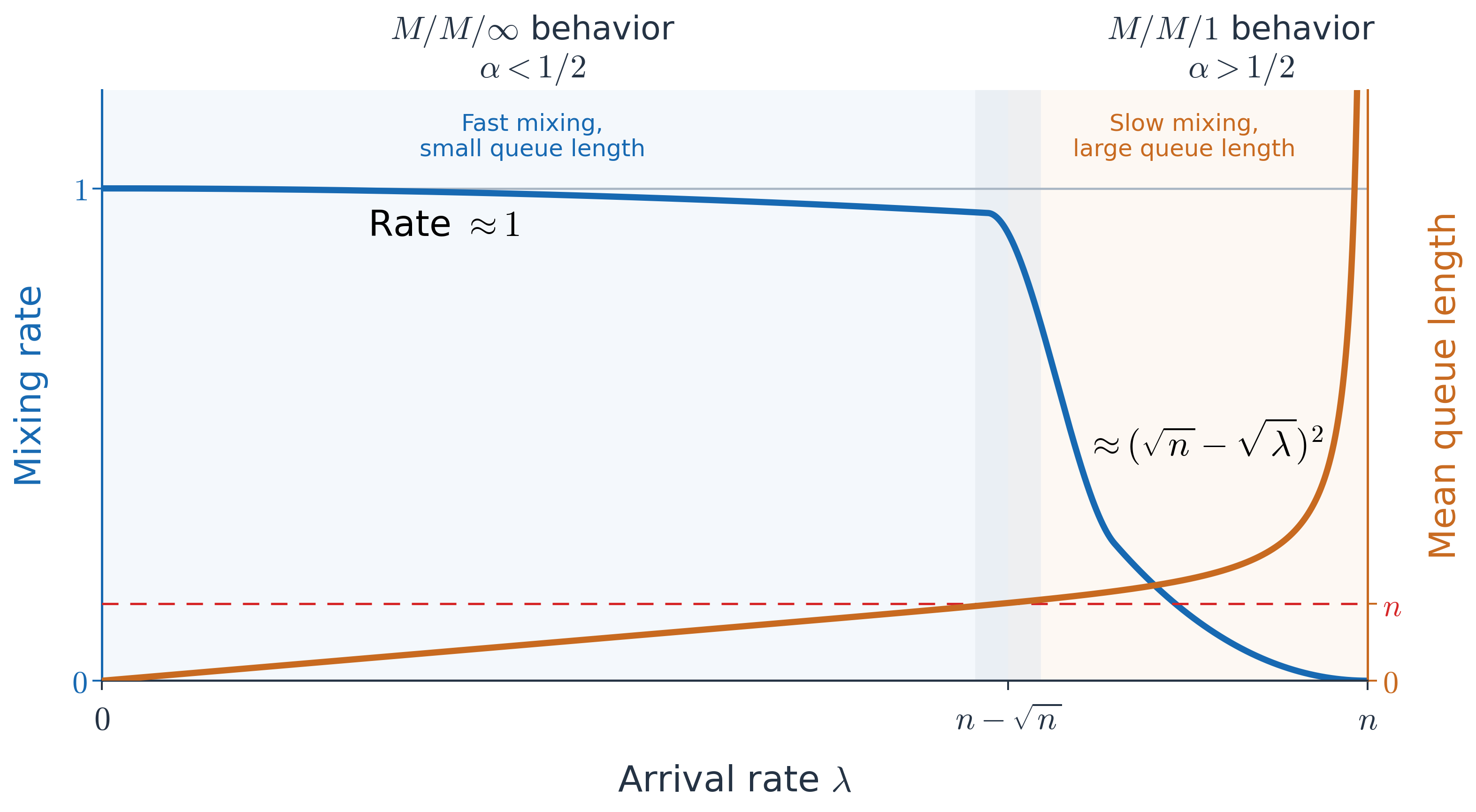}
    \caption{Illustration of the mixing rate (blue, left axis) and stationary mean queue length (orange, right axis) across traffic regimes. When $\alpha < 1/2$, the system mixes quickly while the mean queue length is small. When $\alpha > 1/2$, the mixing is slow, and the queue length is large. The phase transition point is at $\alpha = 1/2$ where $\lambda \approx n-\sqrt{n}$.}
    \label{fig: erlang-c-rate-and-mean-vs-lambda}
\end{figure}

To the best of our knowledge, our work is the first to establish mixing time bounds for $M/M/n$ systems in Chi-square, which is a strong notion of distributional distance that implies all of the following explicit finite-time results. %
\begin{itemize}
    \item \textbf{Finite-time mean queue length bound}: From the Chi-square convergence bound, it is known that one can obtain an upper bound on the moments between the two target distributions \citep{Polyanskiy_Wu_2024}. As a result, we obtain an upper bound on the distance between the finite-time mean queue length and the steady-state mean queue length. The rate at which the finite-time mean queue length converges to the steady-state mean queue length is dictated by the mixing rate of the system described in Theorem \ref{thm: unifying-theorem}. We formally state this result in Corollary \ref{corollary: mean-queue-length-mmn}.
    \item \textbf{Finite-time tail bound}: We also obtain an explicit bound on the upper tail of the queue length distribution at time $t$, which quantifies the transient probability of seeing a large queue above the server pool before the system has fully mixed. This result is formally stated in Corollary \ref{corollary: tail-bounds}. Our finite-time characterizations match prior steady-state results in the limit.
    \item \textbf{Finite-time idle-server probability}: Finally, we provide a finite-time bound on the probability that at least one server is idle at time $t$. This complements the tail bound by quantifying the transient probability of under-utilization and is stated in Corollary \ref{corollary: finite-time-idle-server-probability}.
    \item \textbf{Implications on other metrics}: As a consequence of our Chi-square bounds, we also obtain upper bounds in Total Variation distance and Wasserstein distance between the time-$t$ queue length distribution and the steady-state distribution, which are the metrics more commonly used in prior finite-time analyses of queuing systems.
\end{itemize}

\textbf{Tightness of the bound}: To show that our characterization of the mixing rate in \eqref{eqn: mixing-form} is tight, we look at the large-system asymptotics of $M/M/n$ queue as $n \to \infty$. In particular, we fix the traffic parameter $\alpha \in (0, \infty)$ and consider a sequence of systems indexed by $n$ with $\lambda_n = n-n^{1-\alpha}$. The asymptotic regime when $\alpha = 1/2$ is famously known as the Halfin-Whitt regime. Similarly, the settings of $\alpha >1/2$ and $\alpha <1/2$ are called the super-Halfin-Whitt and sub-Halfin-Whitt regime respectively. For the super-Halfin-Whitt regime, we show that $\lim_{n \to \infty} \rate / (\sqrt{n} - \sqrt{\lambda_n})^2 = 1$, which matches the result for the $M/M/1$ queue. On the other hand, for the sub-Halfin-Whitt regime, we show that $\lim_{n \to \infty} \rate = 1$, which matches the result for the $M/M/\infty$. The equality in the limit is established by showing an upper bound on the $\rate$ that matches the lower bound from Theorem \ref{thm: unifying-theorem} in the asymptotic limit. To further empirically validate our findings, we also provide numerical experiments on finite-time mean queue length bounds in various settings in Section \ref{sec: experiments}.

\begin{table}[!ht]
\centering
\TABLE{Overview of our mixing rate results and proof techniques for the $M/M/n$ queue with arrival rate $\lambda$, service rate $\mu=1$, and traffic parameter
$\alpha=-\log(1-\rho)/\log n$, where $\rho = \lambda/n$ is the load.
\label{tab:mixing_time}}
{
\begin{tabular}{|c||c||c|c|c|c|c|c|c||c|}
\hline
Queuing systems
& $M/M/\infty$
& \multicolumn{7}{c||}{$M/M/n$}
& $M/M/1$
\\ \hline

Traffic parameter
& --
& \multicolumn{3}{c|}{$\alpha \in (0,1/2)$}
& \multicolumn{2}{c|}{$\alpha = 1/2$}
& $\alpha \in (1/2,1)$
& $\alpha \in [1,\infty)$
& --
\\ \hline\hline

\begin{tabular}{@{}c@{}}
$\rate$ lower bound \\
Theorem~\ref{thm: unifying-theorem}
\end{tabular}
& $1$
& \multicolumn{5}{c|}{$\Omega(1)$}
& $(1-o(1))(\sqrt{n}-\sqrt{\lambda})^2$
& $(\sqrt{n}-\sqrt{\lambda})^2$
& $(1-\sqrt{\lambda})^2$
\\ \hline

Lyapunov function
& Linear
& \multicolumn{3}{c|}{Constant + Exp}
& \multicolumn{4}{c||}{Two-sided Exp}
& Exp
\\ \hline

Finite set
& $\{\lambda\}$
& \multicolumn{3}{c|}{$\{0,1,\ldots,n\}$}
& \multicolumn{3}{c|}{$\{\sqbr{\lfloor2\lambda\rfloor-n+1}^+,\ldots,n-1\}$}
& \multicolumn{1}{c||}{$\{n-1\}$}
& $\{0\}$
\\ \hline

Local Poincar\'e method
& N/A
& \multicolumn{3}{c|}{Truncation method}
& \multicolumn{3}{c|}{Canonical path method}
& N/A
& N/A
\\ \hline\hline

Asymptotic regime
& --
& \multicolumn{3}{c|}{Sub-HW}
& \multicolumn{2}{c|}{Halfin-Whitt}
& \multicolumn{2}{c||}{Super-HW}
& --
\\ \hline

\begin{tabular}{@{}c@{}}
Asymptotic $\rate$ \\
Theorem~\ref{thm: limiting-mixing-rates}
\end{tabular}
& $1$
& \multicolumn{3}{c|}{$1$}
& \multicolumn{2}{c|}{$\Omega(1)$}
& \multicolumn{2}{c||}{$(\sqrt{n}-\sqrt{\lambda})^2$}
& $(1-\sqrt{\lambda})^2$
\\ \hline

Upper bound test function
& Linear
& \multicolumn{3}{c|}{Linear}
& \multicolumn{4}{c||}{Constant + Exp}
& Exp
\\ \hline
\end{tabular}
}
{}
\end{table}

}

\subsection{Our approach: The Lyapunov-Poincar\'e method}
\label{ssec: our-approach}

{The Erlang-C system that is the focus of this paper is a countable state Markov chain.} While there is a large body of literature to analyze mixing behavior of continuous state space Markov chains and Markov diffusions \citep{Bakry2013AnalysisAGMarkovDiffusionbook, Durmus2014QuantitativeBoundsLangevin, Levin2017-mixing-book, qin2021geometric-wasserstein, qu-glynn2023wasserstein-contractive-drift}, comparatively, there is limited work on countable state space Markov chains. This is due to a lack of a chain rule in contrast to the continuous-space setting, and many methods in the countable state space are no longer applicable due to the discreteness of the state space. On the other hand, the literature on finite state Markov chains \citep{Levin2017-mixing-book, UBCmixingbook} is also not directly applicable, as the infiniteness of the state space in our setting introduces additional technical challenges. In this paper, we adopt the Lyapunov drift approach to study mixing of the Erlang-C system, since such an approach has already been successfully used for stationary analysis of countable state chains \citep{eryilmaz2013asymptotically, Transform_Lange_2020}.

Previously, it is well-known that one can obtain exponential ergodicity for a stochastic system if the system admits some Lyapunov drift outside of some finite set \citep{meyn1994computable, meyn-tweedie-exponential-ergodicity-1995, Lund1996-computable-rate-stochastically-ordered, douc-moulines2018markovchains, Meyn-markov-chain-book}. However, in order to characterize the mixing rate, one also needs a good handle on the behavior inside the finite set, which is a lot harder to obtain. The special case when the finite set is a singleton (which is the case in an $M/M/1$ queue) is well understood \citep{Lund1996-computable-rate-stochastically-ordered}. When the finite set is not a singleton, one typically uses a ``stitching theorem'' to combine knowledge of the drift outside the finite set and local mixing behavior inside the finite set to obtain a mixing time bound for the entire system. Prior work has explored the use of minorization \cite{meyn-tweedie-exponential-ergodicity-1995,Baxendale_convergence_2005} to handle the finite set and get mixing in a certain weighted norm. However, it is unclear if the resultant mixing rate is tight. A more recent result \cite{taghvaei-lyapunov-poincare} used the drift-plus-minorization framework to obtain Poincar\'e constants, and thus a chi-square mixing rate. In particular, the minorization condition on the finite set was used to first obtain a local Poincar\'e inequality, which in turn is deployed in the ``stitching theorem''. In the $M/M/n$ queue, it is unclear how one can obtain a minorization condition that leads to a tight mixing rate, and so, we instead directly work with the local Poincar\'e inequality within the finite set, and use it in conjunction with an appropriate stitching theorem. More precisely, we propose a \textit{non-uniform} version of the local Poincar\'e inequality, which is necessary to obtain asymptotically tight rates all the way to the phase transition. We then extend the stitching theorem to accommodate the non-uniform variant. These extensions of the framework enable us to obtain asymptotically tight mixing rates of Erlang-C.

\textbf{Handling of the finite set}:
Our work takes novel approaches to obtain tight mixing bounds inside the finite set. To achieve this, we propose two novel finite set techniques: 1) the canonical path method and 2) the truncation method. These two methods correspond to two different regimes, which allow us to obtain mixing time bounds that match the asymptotic behaviors. These methods can be summarized as follows:

\begin{itemize}
    \item \textbf{Local canonical path method}: The canonical path method, previously developed in \citep{Jerrum-permanent}, is a powerful technique used to analyze the mixing time of finite discrete Markov chains by bounding how "congested" the state space is with respect to moving from one state to another. Specifically, this method bounds the mixing time of a finite discrete Markov chain with the congestion ratio, which scales with the stationary distribution and the sum of the lengths of paths passing through a particular edge. As such, this method can be viewed as a conductance-type approach to obtain mixing. Notably, the canonical path method obtains a tight mixing bound when the Markov chain is a symmetric random walk on a path graph \citep{Levin2017-mixing-book, UBCmixingbook}, which is also the behavior of the Erlang-C system in the finite set. By applying the local canonical path method, we obtain a tight weighted local Poincar\'e inequality of the system inside the finite set, which allows us to use our stitching theorems to obtain the mixing rate bound for the entire system.  We formally state our local canonical path lemma in Lemma \ref{lemma: canonical-path}.

    \item  \textbf{Truncation}: {In some scenarios, the finite set behavior might be unfavorable for us to directly apply the aforementioned local canonical path method. For instance, the finite set grows too quickly in the Sub-Halfin-Whitt regime for the aforementioned canonical path method to obtain a vanishing weighted local Poincar\'e constant. Hence, to obtain a tight bound for this regime, we propose the truncation method to obtain local mixing results by using another system with a known Poincar\'e constant to approximate the local Poincar\'e constant of the finite set. With this method, we obtain a tight mixing rate bound that converges to $1$ as $n \rightarrow \infty$ for the entire Sub-Halfin-Whitt regime. We formally state our truncation lemma in Lemma \ref{lemma: truncated poincare}.}
\end{itemize}
A summary of our approach in various cases is given in Table \ref{tab:mixing_time}.

\subsection{Related works}


\textbf{$M/M/1$ queue:} \citet{Morse1955-waiting-lines} is the first to characterize the transient behavior of an $M/M/1$ queue by solving the forward Kolmogorov's equation explicitly. Later, \citep{philippe-robert-stochastic-networks-and-queues} provides a simple coupling argument to establish the distance to stationarity for the $M/M/1$ queue, where the author upper bounds the distance to stationarity with the probability that the two processes will not coalesce within time $T$, which can be further upper bounded by the probability of not hitting $0$ within time $T$. Unfortunately, both of these techniques are difficult to generalize to the $M/M/n$ queue given its more complicated dynamics.

\textbf{$M/M/\infty$ queue:} The explicit characterization of the transient distribution of the queue length is known for the $M/M/\infty$ queue. Such a result is obtained by considering a two-dimensional Poisson point process representing the arrival time and the service time. A different approach is to bound the Entropy of the system \citep{Chafa_2006_entropic_inequalities_infty_queue}, where the author uses the binomial-Poisson entropic inequalities to establish the rate of convergence.

\textbf{$M/M/n$ queue:} In \citep{zeifman_lognorm}, the author uses the log-norm to characterize the convergence rate of the birth-and-death process, which applies to the $M/M/n$ setting. This method obtains the tight mixing rate of $(\sqrt{n\mu} - \sqrt{\lambda})^2$ for $\alpha \geq 1$ and a sub-optimal mixing rate of  $\mu - \frac{\lambda}{n-1} \ll (\sqrt{n\mu} - \sqrt{\lambda})^2$ for $\alpha \in (1/2, 1)$. This suggests that the log-norm approach may not be strong enough to obtain sharp convergence rates for all values of $\alpha \in (0, \infty)$.

In the Halfin-Whitt regime $(\alpha = 1/2)$, \citep{Halfin-Whitt-Gamarnik_2013} uses the spectral representation of birth-and-death processes, previously established by \citep{Karlin1957-KM-representation}, to obtain the rate of convergence of the $M/M/n$ queue in the Halfin-Whitt regime. Most notably, the authors were able to characterize the phase transition of the mixing rate. In particular, by letting the arrival rate $\lambda = n-B\sqrt{n}$ and $\mu = 1$, the mixing rate transitions from $(\sqrt{n\mu} - \sqrt{\lambda})^2$ to $\approx 1$ as $B$ increases from $0 \rightarrow \infty$, establishing a phase transition from $M/M/1$-like behavior to $M/M/\infty$-like behavior in the Halfin-Whitt regime. In their work, \citep{Halfin-Whitt-Gamarnik_2013} establishes a bound on $|P_{ij}(t)-\pi_j(\infty)|$ for $|i-j| = O(\sqrt{n})$ and for a sufficiently large $n$, which only provides a guarantee for sufficiently close states, but does not state finite-time moment bounds and tail bounds. These limitations motivate us to consider finite-time bounds in the Chi-square metric in our work.

The spectral approach is applicable beyond the $M/M/n$ queue to general birth-and-death processes, and has been developed in a series of works \citep{Karlin1957-KM-representation, Van_Doorn1985-decay-bounds, dominguez-book-orthogonal-polynomial}. These systems admit a system of orthogonal polynomials with respect to some spectral measure \citep{dominguez-book-orthogonal-polynomial}, which can be used to study exponential ergodicity \citep{Karlin1957-KM-representation, Halfin-Whitt-Gamarnik_2013}, or hitting time properties \citep{Van_Doorn2017-hitting-time-orthogonal-polynomials}.

Moreover, these prior results using the spectral representation can be used to obtain finite-time Chi-square bounds \citep{Mu-Fa-Chen1991-l2-convergence}. However, no prior work has explicitly stated such results, and these results can only be inferred by putting together several works in the literature. We obtain this and state a formal result in Appendix \ref{ssec: other-discussions-with-prev-works}. We also present a discussion contrasting this work with our approach, as well as an in-depth comparison of our work with previous works on the $M/M/n$ systems. One immediate difference is that our approach is applicable to any reversible Markov chain, whereas the orthogonal polynomial approach is limited to birth-and-death chains.

\textbf{Other queuing systems:} Beyond $M/M/n$, mixing results have been previously established for some load-balancing algorithms and queuing networks. For example, \citep{rutten2023meanfieldspatial} obtained sub-geometric convergence for the spatial load-balancing system and \citep{vijay_saghar_bayesian2023,yuanzhe2025jsqmixing} studied the convergence behavior of the join-the-shortest-queue system with two servers. On the other hand, \citep{Luczak_supermarket_2006, luczak-supermarket-equilibrium} obtained exponential convergence to a ball in the sub-Halfin-Whitt regime for the join-the-shortest-queue and power-of-$d$ policies. For Jackson networks, \citep{Lorek_Szekli_2015} shows exponential ergodicity, and the spectral gap of the open Jackson network is established in \citep{Mao2015-open-Jackson-spectral-gap} using the Markov chain decomposition method.

\textbf{Lyapunov drift method:} As mentioned in Subsection \ref{ssec: our-approach}, we adopt the Lyapunov drift framework and a stitching theorem approach to handle the finite set. Aside from the minorization approach to handle the finite set, alternative approaches include the use of contraction arguments within the finite set \citep{durmus2015subgeometricratesconvergencewasserstein, qin2021geometric-wasserstein}. More recently, \citep{qu-glynn2023wasserstein-contractive-drift} developed a variant that more naturally combines the contraction and the drift arguments. However, all these works establish convergence in a weighted Wasserstein distance. In contrast, our work builds on the Lyapunov drift framework to obtain Chi-square convergence, which is able to obtain Wasserstein-$1$ convergence (see Proposition \ref{prop: chi-square-implies-wasserstein-1}) as well.

\subsection{Notations}
\label{sssec:notations}

To help the readers better understand our work, we define some notations used in our work as follows. For a countable set $\cS$, denote the space of functions $f: \cS \rightarrow \R$ that are square-integrable by $\ell_{2,\nu}$. Now, define the inner product
\begin{equation*}
\lr{f}{g}_{\nu} := \sum_{x \in \cS} \nu(x)f(x)g(x)
\end{equation*}
for $f, g \in \ell_{2,\nu}$.
From the definition of the inner product, we have $\|f\|^2_{2,\nu} := \lr{f}{f}_{\nu}$. Now, denote
\begin{align*}
    \cD(\nu) = \left\{f\in\ell_{2,\nu}: \lim_{t\downarrow0}\frac{P_tf-f}{t} \text{ exists in the }\ell_{2,\nu}\text{ norm}\right\}.
\end{align*}
The distribution $\nu$ is said to be reversible for the operator $\cL: \cD(\nu) \rightarrow \ell_{2, \nu}$ if $\cL$ is self-adjoint on $\ell_{2,\nu}$; in particular,
\begin{equation*}
    \lr{f}{\cL g}_{\nu}= \lr{\cL f}{g}_{\nu},\quad \forall f,g \in \cD(\nu).
\end{equation*}
In addition, for a well-behaved function $f$, we denote $\nu(f) = \E_\nu(f), \Var_\nu(f) = \E_\nu(f^2) - \E_\nu(f)^2$, $\mathbf{1}$ as the vector of $1$s. For a set $K$, we denote $\nu(K) = \int_K d\nu$ and $1_K$ as the indicator function for the set $K$ where $1_K(x) = 1$ if $x \in K$ and $1_K(x) = 0$ otherwise. For $n \in \Z_{\geq 0}$, we denote by $\delta_n$ the point-mass distribution concentrated at $n$. All logarithms in this paper are natural logarithms. Finally, we denote $\sqbr{x}^+ = \max\{0,x\}$ and we interpret $\frac{0}{0}$ as $0$.

\section{Preliminaries}
\label{ssec: chi-square-distance}

{To measure the distance between probability distributions, we first introduce the general framework of $f$-divergences.
\begin{defi}
\label{def: f-divergence}
Given a convex function $f: [0, \infty) \to \mathbb{R}$ such that $f(1) = 0$, and two probability distributions $P$ and $Q$ over a countable support $\cS$ where $P$ is absolutely continuous with respect to $Q$ (denoted $P \ll Q$), the $f$-divergence of $P$ with respect to $Q$ is defined as:
\begin{align}
    D_f(P\|Q) = \sum_{x \in \cS} Q(x) f\br{\frac{P(x)}{Q(x)}}.
\end{align}
\end{defi}}
Under the $f$-divergence perspective, we define the Chi-square distance as an $f$-divergence where $f(x) = (x-1)^2$ as follows.
\begin{defi}
\label{def: chi-square distance}
Given two distributions $P, Q$ whose supports are subsets of the countable set $\cS$ and $P$ is absolutely continuous w.r.t $Q$, the Chi-square distance of $P$ with respect to $Q$ is given by
\begin{align}
    \chi^2 \br{P,Q} = \sum_{x \in \cS} \frac{(P(x)-Q(x))^2}{Q(x)}.
\end{align}
\end{defi}
As $\E_Q \sqbr{\frac{P}{Q}} = 1$, the Chi-squared distance $\chi^2(P,Q)$ can also be rewritten as
\begin{align}
    \chi^2(P,Q) = \sum_{x \in \cS} \frac{(P(x)-Q(x))^2}{Q(x)} = \sum_{x \in \cS} Q(x) \br{\frac{P(x)^2}{Q(x)^2}-1} = \Var_Q\br{\frac{P}{Q}}.
\end{align}
{Notably, it is known that every $f$-divergence is locally Chi-square \citep{Polyanskiy_Wu_2024} under suitable regularity conditions. When we choose $f(x) = \frac{|x-1|}{2}$, we obtain the Total Variation distance, which is formally defined as follows.
\begin{defi}
\label{def: tv distance}
Given two distributions $P, Q$ over a countable support $\cS$, the Total Variation (TV) distance between $P$ and $Q$ is given by
\begin{align}
    d_{TV}(P,Q) = \frac{1}{2} \sum_{x \in \cS} |P(x) - Q(x)|.
\end{align}
\end{defi}}

{In most previous works on the finite-time analysis of queuing systems, the distance to stationarity is often measured in TV distance \citep{philippe-robert-stochastic-networks-and-queues, Luczak_supermarket_2006, luczak-supermarket-equilibrium} or other similar metrics \citep{zeifman_lognorm, Halfin-Whitt-Gamarnik_2013}. Such a bound on the TV distance can be obtained using the Chi-square distance bound via the following inequality:
\begin{align}
    \label{eqn: tv-chi-square-bound}
    d_{TV}(P,Q) \leq \frac{\sqrt{\chi^2(P, Q)}}{2}.
\end{align}
The readers can refer to (7.37) in \cite{Polyanskiy_Wu_2024}. Additionally, Chi-square bounds can also be used to provide an upper bound on the Wasserstein-$1$ distance, which is defined as follows.
\begin{defi}
\label{def: w1-distance}
Given two distributions $P,Q$ over a countable support $\cS\subseteq\R$ with finite first moments and let $\Pi(P,Q)$ be the set of all couplings of $P$ and $Q$, the Wasserstein-$1$ distance between $P$ and $Q$ is defined as
\begin{align}
    W_1(P,Q) = \inf_{\Gamma\in\Pi(P,Q)}
    \sum_{x,y\in\cS}|x-y|\Gamma(x,y).
\end{align}
\end{defi}
From here, we have the following upper bound on the Wasserstein-$1$ distance using the Chi-square distance.
\begin{prop}
    \label{prop: chi-square-implies-wasserstein-1}
    Let $P,Q$ be distributions over a countable support $\cS\subseteq\R$ with finite first moments. Assume $P \ll Q$ and $Q$ has finite variance. Then
    \begin{align}
        W_1(P,Q) \leq \sqrt{\Var_Q(X)} \sqrt{\chi^2(P,Q)}.
    \end{align}
\end{prop}
The proof of this proposition is written in Appendix \ref{ssec:chi-square-to-w1-proof}. Hence, to provide an apples-to-apples comparison with the TV and the Wasserstein-$1$ distance, we shall define $\chi(P,Q) = \sqrt{\chi^2(P,Q)}$ as the square root of the Chi-square distance of $P$ with respect to $Q$.}

Another reason that we are interested in the Chi-square distance is that the transient mean queue length and tail bounds can be derived from the Chi-square distance to stationarity. Indeed, we have the following Theorem on the variational form of the Chi-square distance.

\begin{prop}
    \label{prop: variational-representation-chi-square}
    Given two distributions $P, Q$ with countable supports such that $P$ is absolutely continuous w.r.t. $Q$ and $Q$ is not a Dirac delta. Moreover, let $\cG_Q = \br{g: \cS \rightarrow \R: 0<\Var_Q(g(X))<\infty, \E_P[g] <  \infty}$, we have
    \begin{align}
        \label{eqn: variational-representation}
        \chi^2(P,Q) = \sup_{g \in \cG_Q} \frac{\br{\E_P[g(X)]-\E_Q[g(X)]}^2}{\Var_Q[g(X)]}
    \end{align}
\end{prop}

The proof of Proposition \ref{prop: variational-representation-chi-square} is well-known and is presented in Appendix \ref{ssec: variational-representation-proof}. From Proposition \ref{prop: variational-representation-chi-square}, we can substitute the function $g$ of our interest to obtain specific bounds. In particular, we obtain the following bound on the distance of the moments of the distributions by substituting $g(x) = x^k$.

\begin{corollary}
    \label{corollary: moment-bound}
    Given two distributions $P, Q$ with countable supports such that $P$ is absolutely continuous w.r.t. $Q$. In addition, let $k \in \Z^+$ and assume that $\E_Q[X^{2k}] < \infty$, then
    \begin{align}
        \left| \E_P \sqbr{X^k} - \E_Q \sqbr{X^k} \right| \leq \chi(P,Q) \sqrt{\E_Q\sqbr{X^{2k}}-\E_Q\sqbr{X^k}^2}.
    \end{align}
\end{corollary}

We can substitute $g(X) = e^{\theta X}$ for an appropriately chosen $\theta \in \mathbb{R}$ to establish a bound on the difference of the Moment Generating Functions (MGFs) as stated in the following corollary.

\begin{corollary}
\label{corollary: mgf-bound}
Given two distributions $P, Q$ with countable supports such that $P$ is absolutely continuous w.r.t. $Q$ and let $\theta \in \R^+$ such that $\E_Q[e^{2\theta X}] < \infty$, then
\begin{align}
    \left| \E_P \sqbr{e^{\theta X}} - \E_Q \sqbr{e^{\theta X}} \right| \leq \chi\br{P, Q} \sqrt{\E_Q\sqbr{e^{2\theta X}}-\E_Q\sqbr{e^{\theta X}}^2}.
\end{align}
\end{corollary}

The MGF bounds are particularly useful for obtaining exponential tail bounds. With the proper scaling, many classical queuing systems such as $M/M/n$ and JSQ possess an MGF \citep{Transform_Lange_2020, prakirt-confluence-large-deviation}, and thus one can derive exponential tail bounds for such systems.

These results will be the foundation for obtaining the finite-time statistics of our queuing systems of interest. Indeed, if we can obtain a Chi-square bound between the time $t$ and the steady-state distribution, then we can apply these results to obtain finite-time statistics such as mean queue length bounds and tail bounds. To this end, we develop machinery to establish the distance to stationarity bounds for $\pi_{n,t}$, in turn, resulting in finite-time statistics.

\section{Model and Main Results}
\label{sec:model-main-results}

\subsection{Model}
\label{ssec:model}

Consider the $M/M/n$ queuing system, also known as the Erlang-C system, described as follows. Customers arrive at a central queue following a Poisson process with rate $\lambda > 0$. In addition, there are $n$ homogeneous servers that serve the waiting customers in a first-come-first-served (FCFS) manner. The service time of each customer is independently and exponentially distributed with mean $1/\mu$. Without loss of generality, we assume $\mu = 1$ unless specified otherwise. Moreover, to ensure stability, we further restrict $\lambda \in (0,n)$. The number of customers $\{q(t): t \geq 0\}$ in the system evolves as a continuous-time Markov chain (CTMC) with state space $\Z_{\geq 0}$. This CTMC is a birth-and-death process with the birth rate $\lambda$ and the death rate $\min\{n, q\}$ in the state $q(t)=q$. We denote $\cL$ as the infinitesimal generator of the CTMC and $P_t$ as the corresponding Markov semigroup, where $P_t = e^{t \cL}$. For a given $n$, we denote the queue length distribution at time $t \geq 0$ as $\pi_{n,t}$ and the stationary distribution of the $M/M/n$ queue by $\nu_n$. We also denote $\rho = \frac{\lambda}{n}$ as the load of the system and $\varepsilon = 1-\rho = 1-\frac{\lambda}{n}$ as the heavy traffic parameter.

Moreover, we define $M/M/\infty$ as the birth-and-death process with the birth rate $\lambda$ and the death rate $q$ in state $q(t) = q$. This system behaves as if there are an infinite number of servers, and whenever a customer arrives, it is immediately allocated an idle server to get served. For this system, we denote the queue length distribution at time $t$ as $\pi_{\infty,t}$ and the stationary distribution by $\nu_\infty$.

\subsection{Main Results}
\label{ssec: main-results}
For a stable $M/M/n$ queue with load $\rho=\lambda/n \in (0,1)$ and $n \geq 2$, we define the traffic parameter
\begin{align}
    \label{eqn: alpha-definition}
    \alpha = 1-\frac{\log\br{n-\lambda}}{\log n} =    \frac{\log\br{\frac{1}{1-\rho}}}{\log n}.
\end{align}
Equivalently, this definition gives the arrival-rate parametrization $\lambda = n-n^{1-\alpha}$. At the high level, $\alpha$ quantifies the spare capacity $n-\lambda$ relative to the system size $n$, and the larger the value of $\alpha$, the closer the system is to the heavy-traffic limit. It turns out that this parameter determines the qualitative mixing behavior of the system. In particular, when $\alpha > 1/2$, the system behaves similarly to an $M/M/1$ queue. On the other hand, when $0 < \alpha < 1/2$, the system behaves more like an $M/M/\infty$ system. Thus, we seek to obtain a lower bound on $\rate_{n,\alpha}$, where $\rate_{n,\alpha}$ denotes the spectral gap $\operatorname{Gap}(\cL)$ of the $M/M/n$ system with the arrival rate parameter $\alpha$. As a consequence, we can obtain a bound on $\chi(\pi_{n, t}, \nu_n)$ of the form
\begin{align}
    \label{eqn: mixing-rate-form}
    \chi(\pi_{n,t},\nu_n) \leq e^{-\rate_{n,\alpha} \, \times \, t} \chi(\pi_{n,0},\nu_n) \, \forall t \geq 0.
\end{align}
With a slight abuse of notation, we will drop $n,\alpha$ and simply refer to the mixing rate of the system as $\rate$. In the following theorem, we provide explicit characterizations of $\rate$ depending on the system regime as follows.

\begin{theorem}
\label{thm: unifying-theorem}
    {Consider the continuous-time $M/M/n$ system as in Subsection \ref{ssec:model}. Let $\alpha$ be defined as in \eqref{eqn: alpha-definition}, and assume $n \geq 110$ and $\lambda \geq 3$. Then, \eqref{eqn: mixing-rate-form} holds with
    \begin{align*}
        \rate \geq
        \begin{cases}
            (\sqrt{n}-\sqrt{\lambda})^2 &\text{if } \alpha \in [1, \infty), \\
            \frac{(\sqrt{n}-\sqrt{\lambda})^2}{1 + 1.04 \times 10^5 n^{2-4\alpha} + 1.6 \times 10^5 n^{3-6\alpha}} &\text{if } \alpha \in (1/2, 1), \\
            \frac{1}{4789.38} &\text{if } \alpha = 1/2, \\
            \frac{1}{33.2 + \frac{58.4}{\sqrt{2\pi e}} \br{\frac{4\alpha}{e(1-2\alpha)}}^{\frac{2\alpha}{1-2\alpha}}} &\text{if } \alpha \in (0, 1/2).
        \end{cases}
    \end{align*}}
\end{theorem}
The full proof of this theorem is deferred to Appendix \ref{sec:unifying-theorem-proof}, and its proof sketch is summarized in Section \ref{sec: proof-sketch}. Note that Theorem \ref{thm: unifying-theorem} assumes $n \geq 110$ and $\lambda \geq 3$ to obtain reasonable explicit constants. Improving the mixing rates as well as the lower bound assumption on $n$ and $\lambda$ is possible, and we leave the task of optimizing these constants as future work. {Qualitatively speaking, Theorem \ref{thm: unifying-theorem} suggests that the finite-time convergence behavior of the $M/M/n$ queue exhibits a phase transition at $\alpha = 1/2$, where $\rate \approx (\sqrt{n}-\sqrt{\lambda})^2$ for $\alpha > 1/2$, which is similar to that of an $M/M/1$ with service rate $n$. In contrast, for a fixed $0 < \alpha < 1/2$, the mixing rate is bounded below by a positive constant, reflecting a behavior closer to that of an $M/M/\infty$ system. To make this connection explicit, we prove the following mixing result for the $M/M/1$ queue using a technique similar to the proof of Theorem \ref{thm: unifying-theorem} when $\alpha \in [1,\infty)$:
\begin{prop}
\label{thm: mm1-mixing}
Let $\pi_{1,t}$ and $\nu_1$ be the queue length distribution at time $t \geq 0$ and the steady-state distribution of a continuous-time $M/M/1$ system with arrival rate $\lambda \in (0,1)$ and unit service rate, respectively. We have
\begin{align*}
    \rate = (1-\sqrt{\lambda})^2.
\end{align*}
\end{prop}
The transient behavior of $M/M/1$ is previously known in \cite{Morse1955-waiting-lines,Leguesdron1993-mm1-transient-anaysis,philippe-robert-stochastic-networks-and-queues}, and the exact Poincar\'e constant of the geometric distribution is known in \cite{Joulin2004-geometric-Poincare}. Here, we provide a chi-square convergence result for $M/M/1$ for completeness. In the regime $\alpha\in(0,1/2)$, our result shows that the mixing rate remains $\Theta(1)$ for a fixed $\alpha$. For comparison, we have the following well-known mixing result for $M/M/\infty$ \citep{Chafa_2006_entropic_inequalities_infty_queue}.
\begin{prop}
    \label{thm: mminf-mixing}
    Let $\pi_{\infty,t}$ and $\nu_\infty$ be the queue length distribution at time $t \geq 0$ and the steady-state distribution of a continuous-time $M/M/\infty$ system with arrival rate $\lambda \in \Z^+$ and unit service rate, respectively. We have
    \begin{align*}
        \rate = 1.
    \end{align*}
\end{prop}
The proof of Proposition \ref{thm: mminf-mixing} is given in Appendix \ref{ssec: mminf-proof}. Our result in the $\alpha \in (0,1/2)$ regime also improves substantially over bounds that vanish with $n$. For example, the result of \cite{zeifman_lognorm} implies a mixing rate of
\begin{align*}
    1-\frac{\lambda_n}{n-1} = O\br{n^{-\alpha}} \rightarrow 0
\end{align*}
as $n\rightarrow\infty$. Moreover, at the phase transition point $\alpha = 1/2$, we have $(\sqrt{n}-\sqrt{\lambda})^2 \approx 1/4$, which means that the mixing rate is bounded below by a universal constant in this regime. Hence, Theorem \ref{thm: unifying-theorem} matches the behavior of the corresponding asymptotic regimes, although it is not fine-grained enough to characterize the phase transition at $\alpha = 1/2$ as in \cite{Halfin-Whitt-Gamarnik_2013}. We refer the reader to Appendix \ref{ssec: other-discussions-with-prev-works} for a more detailed comparison with previous work.}

\subsubsection{Finite-time Corollaries}
To see how our convergence bound in Theorem \ref{thm: unifying-theorem} can be useful, we shall use it to derive finite-time statistics of the $M/M/n$ system. First, denote $\tau_{mix}(\delta) = \inf\left\{t \geq 0: \chi(\pi_{n,t},\nu_n) \leq \delta \right\}$ as the $\delta$-mixing time, we obtain the following corollary on the mixing time of the system from the bound on the Chi-square distance.
\begin{corollary}
    \label{corollary: mixing-time-bound}
    (Mixing time bounds for $M/M/n$ systems) Consider the continuous-time $M/M/n$ system as in Subsection \ref{ssec:model}. For any $0<\delta\leq \chi(\pi_{n,0},\nu_n)$, we have that:
    \begin{align}
        \label{eqn: mixing-time bound}
        \tau_{mix}(\delta) \leq \frac{\log\br{\frac{\chi(\pi_{n,0}, \nu_n)}{\delta}}}{\rate},
    \end{align}
    where $\rate$ is described in Theorem \ref{thm: unifying-theorem}.
\end{corollary}

Corollary \ref{corollary: mixing-time-bound} tells us that the $\delta$-mixing time $\tau_{mix}(\delta)$ has logarithmic dependence on the initial distance and $\delta^{-1}$. Additionally, when the traffic is light (i.e., for a fixed $\alpha \in (0,1/2)$), it takes $O(1)$ time to mix as $\rate = \Omega(1)$. However, when the system has heavy traffic (i.e. $\alpha > 1/2$), we have the mixing time is approximately $4 n^{2\alpha-1}$, which would grow to infinity as $n$ grows. Therefore, it would take a long time to mix in this regime. We defer the proof of Corollary \ref{corollary: mixing-time-bound} to Appendix \ref{sec: proof-of-tv-corollary}.

From the established mixing bounds in Theorem \ref{thm: unifying-theorem}, we can obtain finite-time statistics like the mean and tail. We start with the mean queue length result below.

\begin{corollary}
\label{corollary: mean-queue-length-mmn}
Under the assumptions of Theorem \ref{thm: unifying-theorem}, consider the continuous-time $M/M/n$ system as in Subsection \ref{ssec:model}. We have
\begin{align}
    \left| \E_{\pi_{n,t}}[q] - \E_{\nu_n}[q] \right| \leq e^{- \rate \, \times \, t} \sqrt{2}\br{\sqrt{n}+n^\alpha} \chi(\pi_{n,0}, \nu_n)
\end{align}
where $\rate$ is defined as in Theorem \ref{thm: unifying-theorem}.
\end{corollary}

One can prove Corollary \ref{corollary: mean-queue-length-mmn} by combining Theorem \ref{thm: unifying-theorem} and Corollary \ref{corollary: moment-bound} for $k = 1$. We defer the details of the proof to Appendix \ref{ssec: mean-queue-length-mmn-proof}. Similar to the above corollary, we can obtain bounds on any moment $k \in \Z^+$ by Corollary \ref{corollary: moment-bound} combined with a $2k$-moment bound on the steady-state distribution $\nu_n$. Moreover, in addition to the mean queue length bound, we can establish some tail-bound results to gain some understanding of the concentration of the finite-time behavior of the queue length distribution as follows.

\begin{corollary}
\label{corollary: tail-bounds}
Under the assumptions of Theorem \ref{thm: unifying-theorem}, consider the continuous-time $M/M/n$ system as in Subsection \ref{ssec:model}, and let $\varepsilon= 1-\frac{\lambda}{n}$ and $q$ be the random variable denoting the queue length. For $n \geq n_0 = \max\left\lbrace 110, \left\lceil 2^{\frac{1}{\alpha}} \right\rceil\right\rbrace$ and $x > 0$, we have
\begin{align}
    \P_{\pi_{n,t}}\sqbr{{\varepsilon} \br{q-n} > x} &\leq \operatorname{prefactor} \times \br{1 + e^{-\rate \, \times \, t} \chi(\pi_{n,0}, \nu_n)}  e^{-\frac{x}{2}}
\end{align}
{where $\rate$ is described as in Theorem \ref{thm: unifying-theorem} and
\begin{align}
    \operatorname{prefactor} =
    \begin{cases}
        \sqrt{e(x+1)} &\text{ when } \alpha \in [1/2,\infty) \\
        \sqrt{e^{1-\frac{x n^{1-2\alpha}}{2(x+1)}} + 10e(x+1)n^{2\alpha-1}} &\text{ when } \alpha \in (0,1/2).
    \end{cases}
\end{align}}
\end{corollary}

The proof of Corollary \ref{corollary: tail-bounds} is deferred to Appendix \ref{sssec: tail-bounds-proof}. Note that if $q(t) \geq n$, then $q(t) - n$ is the number of waiting customers at time $t$. The results in Corollary \ref{corollary: tail-bounds} tell us the tail distribution of the number of waiting customers at time $t$. {While we have the tail $e^{-\frac{x}{2}}$ in all regimes, the pre-exponent term is different for each regime. In particular, we have the pre-exponent factor vanishes to $0$ as $n \rightarrow \infty$ and for every fixed $x > 0$ in the Sub-Halfin-Whitt regime, whereas the pre-exponent term in the Super-Halfin-Whitt regime has a constant order. At stationarity, Corollary \ref{corollary: tail-bounds} provides an exponential upper-tail bound that is qualitatively consistent with \cite{prakirt-confluence-large-deviation}. However, our bound does not recover the sharp stationary large-deviation exponent obtained therein. Thus, the two results are complementary: \cite{prakirt-confluence-large-deviation} characterizes the stationary tail sharply, whereas our mixing bound quantifies the additional transient error influenced by the initial distribution.}
Additionally, we can also estimate the probability of having an idle server at time $t$ as follows.

\begin{corollary}
    \label{corollary: finite-time-idle-server-probability}
    Under the assumptions of Theorem \ref{thm: unifying-theorem}, consider the continuous-time $M/M/n$ system as in Subsection \ref{ssec:model}. Let $r_n$ be a random variable denoting the number of idle servers, that is $r_n = \sqbr{n-q}^+$, where $q$ is the queue length random variable. {For $\alpha \in [1/2, \infty)$, we have
    \begin{align*}
        \P_{\pi_{n,t}}\sqbr{r_n > 0} \leq O\br{n^{\frac{1}{2}-\alpha}} + 2\sqrt{e \pi} n^{\frac{1}{4}-\frac{\alpha}{2}} e^{-\rate \, \times \, t} \chi(\pi_{n,0}, \nu_n).
    \end{align*}
    \noindent
    For $\alpha \in (0,1/2)$, we have
    \begin{align*}
        \P_{\pi_{n,t}}\sqbr{r_n > 0} \geq 1- O\br{n^{\alpha - \frac{1}{2}} e^{-n^{\frac{1}{2}-\alpha}}} - e^{-\rate \, \times \, t} \chi\br{\pi_{n,0}, \nu_n}.
    \end{align*}
    Here $\rate$ is described as in Theorem \ref{thm: unifying-theorem}.}
\end{corollary}

The proof of Corollary \ref{corollary: finite-time-idle-server-probability} can be found in Appendix \ref{ssec: idle-server-finite-time}. These results are the finite-time version of the steady-state results established in \cite{prakirt-confluence-large-deviation} and thus can be seen as a finite-time generalization of such results. {In particular, for a sufficiently large time $t_n \geq \tau_{mix}(\delta)$ for some constant $\delta > 0$ as in Corollary \ref{corollary: mixing-time-bound} (i.e., the system is well-mixed), the probability of having an idle server in the Super-Halfin-Whitt regime is small and vanishes to $0$ as $n \to \infty$. On the other hand, as suggested in Corollary \ref{corollary: tail-bounds}, the probability of having a waiting customer is small in the Sub-Halfin-Whitt regime, and so, we have the probability of having an idle server is lower bounded in this regime.}

\subsection{Asymptotic tightness of the rates}
\label{ssec: limiting-results}
Recall that our Theorem \ref{thm: unifying-theorem} is stated in a non-asymptotic form: for any $n \geq 110$ and $\lambda \in [3,n)$, one can compute the corresponding value of $\alpha = 1-\frac{\log(n-\lambda)}{\log n}$ and obtain an explicit mixing-rate lower bound, potentially with some constant factor loss. To further investigate the tightness of this mixing rate bound, we shall study the mixing behavior of the $M/M/n$ system at the asymptotic. Throughout this subsection, unless otherwise stated, we consider a sequence of arrivals indexed by $n$ such that $\lambda_n = n-n^{1-\alpha}$ for a fixed $\alpha>0$ and let $n\rightarrow \infty$. As an extension of Theorem \ref{thm: unifying-theorem}, we obtain the following theorem.
\begin{theorem}
    \label{thm: limiting-mixing-rates}
    Consider a sequence of continuous-time $M/M/n$ systems as in Subsection \ref{ssec:model}, with arrival rates $\lambda_n=n-n^{1-\alpha}$ for a fixed $\alpha>0$, the following statements hold:
    \begin{enumerate}
        \item If $\alpha \in (1/2,\infty)$, then we have
        \begin{align}
            \label{eqn: limiting-super-hw-Cn-limit}
            \lim_{n\rightarrow\infty} \frac{\rate}{(\sqrt{n}-\sqrt{\lambda_n})^2} = 1.
        \end{align}
        \item If $\alpha \in (0,1/2)$, then we have
        \begin{align}
            \label{eqn: limiting-sub-hw-Dn-limit}
            \lim_{n\rightarrow\infty} \rate = 1.
        \end{align}
        \item If $\alpha = 1/2$, then we have
        \begin{align}
            \label{eqn: limiting-halfin-whitt-Dn-limit}
            \liminf_{n\rightarrow\infty} \rate \geq \frac{1}{359.06}.
        \end{align}
    \end{enumerate}
\end{theorem}
The proof of Theorem \ref{thm: limiting-mixing-rates} is deferred to Appendix \ref{sec: limiting-mixing-rates-proof}. Here, Theorem \ref{thm: limiting-mixing-rates} tells us that as we are closer to the asymptotic regime (e.g., $n \to \infty$), the mixing rate bound becomes tighter and closer to that of the corresponding limiting systems. In particular, for $\alpha \in (1/2,\infty)$, the system behaves like an $M/M/1$ system with service rate $n$ whose mixing rate is $(\sqrt{n}-\sqrt{\lambda_n})^2$. Hence, the first part of Theorem \ref{thm: limiting-mixing-rates} shows that our mixing-rate bound is asymptotically tight in the Super-Halfin-Whitt regime. This behavior is consistent with the intuition that, in the Super-Halfin-Whitt regime, the $M/M/n$ queue behaves like an $M/M/1$ queue with arrival rate $\lambda$ and service rate $n$. This is also consistent with the limiting results of \cite{Halfin-Whitt-Gamarnik_2013} and with prior mixing bounds for birth-and-death queues \citep{Van_Doorn1985-decay-bounds, zeifman_lognorm}. Moreover, for $\alpha \in (0,1/2)$, we also show that in the limit, we can match the mixing rate $1$ of the $M/M/\infty$ system, as previously established in \cite{Chafa_2006_entropic_inequalities_infty_queue}. To the best of our knowledge, we are the first to establish an asymptotically tight bound for the $M/M/n$ system in this regime. Finally, in the Halfin-Whitt regime $\alpha=1/2$, Theorem \ref{thm: limiting-mixing-rates} shows that our method still yields a constant-order mixing rate at the phase transition where we have $(\sqrt{n}-\sqrt{\lambda_n})^2 \approx \frac{1}{4}$ in this regime. However, the limiting constant in \eqref{eqn: limiting-halfin-whitt-Dn-limit} is worse than the constant obtained in the Sub-Halfin-Whitt regime. This is perhaps unsurprising, since the Halfin-Whitt regime lies precisely at the boundary between the Super-Halfin-Whitt and Sub-Halfin-Whitt regimes, and its behavior reflects features of both. We do not attempt here to recover the sharper spectral-gap phase transition characterized in \citep{Halfin-Whitt-Gamarnik_2013}; rather, our goal is to show that the mixing rate in this boundary regime is uniformly bounded below by a positive constant. We present the proof of this regime in Appendix \ref{ssec: unified-halfin-whitt-proof}.

\section{The Lyapunov-Poincar\'e method}
\label{ssec: lyapunov-poincare-method}
To establish the mixing results stated in Subsection \ref{ssec: main-results}, we shall introduce the so-called Lyapunov-Poincar\'e machinery. All the definitions and results in this section are valid for any CTMC with a countable state space $\cS$, generator $\cL$, and Markov semigroup $P_t$.

\begin{assumption}
    \label{assumption: ctmc}
    The given CTMC is irreducible, aperiodic, positive recurrent, and reversible.
\end{assumption}
Assumption \ref{assumption: ctmc} implies that there exists a unique stationary distribution $\nu$ such that $\nu(x) \cL (x,y) = \nu(y) \cL(y,x)$ for all $x,y \in \cS$.

\begin{defi} 
\label{def: poincare_inequality}
A CTMC admits a Poincar\'e constant $C_P>0$ if for every test function $f \in \ell_{2,\nu}$, we have
\begin{align}
    \label{eqn: poincare-inequality}
    \Var_\nu(f) \leq C_P \cE(f,f).
\end{align}
\end{defi}
Here, the Dirichlet form $\cE(f,f)$ is defined for $f\in\ell_{2,\nu}$ by the nonnegative sum
\begin{align}
    \label{eqn: dirichlet-form}
    \cE(f,f) = \langle f,-\cL f \rangle_\nu = \frac12\sum_{x\in\cS}\sum_{y\ne x}\nu(x)\cL(x,y)\br{f(y)-f(x)}^2.
\end{align}
If $\cE(f,f)=\infty$, \eqref{eqn: poincare-inequality} holds trivially. As a consequence, Poincar\'e inequalities imply exponential ergodicity of the CTMC \citep{Bakry2013AnalysisAGMarkovDiffusionbook}. We formally state it in the following proposition for completeness.

\begin{prop}
    \label{prop: poincare-equals-mixing}
    Under Assumption \ref{assumption: ctmc}, if the system admits the Poincar\'e inequality \eqref{eqn: poincare-inequality} with constant $C_P$, then for an initial queue length distribution satisfying $\chi(\pi_0,\nu) < \infty$, we have
    \begin{align}
        \label{eqn:chi-square-convergence}
        \chi\br{\pi_t,\nu} \leq e^{-\frac{t}{C_P}} \chi(\pi_0,\nu)
    \end{align}
    where $\pi_t$ is distribution at time $t \geq 0$ and $\nu$ is the stationary distribution. Consequently, we have $\rate \geq \frac{1}{C_P}$.
\end{prop}

The proof of Proposition \ref{prop: poincare-equals-mixing} is well-known and can be found in \cite{Bakry2013AnalysisAGMarkovDiffusionbook} and in Appendix \ref{ssec:mixing-prop-proof}. Proposition \ref{prop: poincare-equals-mixing} shows that Poincar\'e inequality with constant $C_P$ implies exponential convergence in the square root of the Chi-square distance with mixing rate $1/C_P$. Thus, for the rest of the section, our focus is on establishing the Poincar\'e constant.

Building on the recent developments on Markov chain mixing \citep{taghvaei-lyapunov-poincare}, we now present the Lyapunov-Poincar\'e methodology below, which has two main ingredients. The first one is a Foster-Lyapunov \emph{like} assumption:
\begin{assumption}
\label{assumption: foster-lyapunov}
There exists a Lyapunov function $V: \cS \rightarrow [1,\infty)$ along with $b: \cS \rightarrow [0,\infty)$ and $\gamma > 0$ such that
\begin{align}
    \label{eqn: foster-lyapunov}
    \cL V(q) \leq - \gamma V(q) + b(q) \, \forall q \in \cS.
\end{align}
\end{assumption}
For the special case of $b = B 1_K$ for some finite $K \subseteq \cS$ and $B>0$, the above is the same as the classical Foster-Lyapunov condition. More generally, let $K = \{q: b(q) > 0\}$, then in words, the above assumption ensures a negative drift proportional to $V(q)$ outside $K$. On the other hand, inside $K$, our drift might be positive, but it is controlled by $b(q)$ for $q \in K$. If $K$ is a singleton, then we have a slightly modified assumption as follows.

\begin{assumption}
\label{assumption: foster-lyapunov-singleton}
There exists a single point $x^* \in \cS$, a Lyapunov function $V: \cS \rightarrow [0,\infty)$ satisfying $V(q) > 0 \, \forall q \neq x^*$ and constants $\gamma, B > 0$ such that
\begin{align}
    \label{eqn: foster-lyapunov-singleton}
    \cL V(q) \leq - \gamma V(q) + B 1_{\{x^*\}} \, \forall q \in \cS.
\end{align}
\end{assumption}

Unlike Assumption \ref{assumption: foster-lyapunov}, Assumption \ref{assumption: foster-lyapunov-singleton} allows us to have $V(x^*) < 1$. In many settings to be discussed below, this modification is necessary to obtain the tight mixing rate in the case where we have a negative drift outside of a singleton. From here, we have the following proposition.

\begin{prop}
    \label{prop: lyapunov-poincare-singleton}
    Under Assumptions \ref{assumption: ctmc} and \ref{assumption: foster-lyapunov-singleton}
    , the CTMC admits the Poincar\'e constant $C_P =  \frac{1}{\gamma}$.
\end{prop}

While similar convergence results for TV distance were previously established \citep{Lund1996-computable-rate-stochastically-ordered, Baxendale_convergence_2005}, this result allows us to establish Chi-square convergence, and we present the proof in Appendix \ref{sssec: singleton-poincare-constant-proof}. More generally, the finite set $K$ may not be a singleton for most of the queuing systems (including the $M/M/n$ queue), and thus, the Foster-Lyapunov assumption alone fails to provide a Poincar\'e inequality since we do not know how the system would mix inside $K$. So we need to establish a local mixing result inside $K$, which can be done using a local Poincar\'e inequality. Such a result is formalized in the following assumption.

\begin{assumption}
\label{assumption: weighted-poincare}
(Weighted Poincar\'e) Given the stationary distribution $\nu$, a function $b: \cS \rightarrow [0,\infty)$ such that $0 < \sum_{q \in \cS} b(q)\nu(q) < \infty$ and let {$\vartheta(q) = \frac{b(q) \nu(q)}{\sum_{q \in \cS} b(q)\nu(q)} \forall q \in \cS$} be a $(b,\nu)$-weighted distribution. Denote $K = \{q \in \cS: b(q) > 0\}$, a CTMC is said to admit a $b$-weighted Poincar\'e inequality if there exists a non-negative constant $C_b$ such that
\begin{align}
    \label{eqn: weighted-poincare}
    {\Var_\vartheta(f) = \norm{f - \E_\vartheta[f]}_{2,\vartheta}^2 \leq \frac{C_b}{\nu(K)} \langle f, -\cL f \rangle_\nu \, \forall f \in \ell_{2, \nu}.}
\end{align}
\end{assumption}
Now, we combine Assumption \ref{assumption: weighted-poincare} along with the drift assumption (Assumption \ref{assumption: foster-lyapunov}) to obtain the Poincar\'e constant for the entire system, which is stated in the following Stitching Theorem.
\begin{theorem}
    (Stitching Theorem)
    \label{theorem: lyapunov-poincare}
    Denote $K = \{q \in \cS: b(q) > 0\}$ and $\nu_K(q) = \frac{\nu(q)}{\nu(K)} 1_K(q) \, \forall q \in \cS$.
    Under Assumptions \ref{assumption: ctmc}, \ref{assumption: foster-lyapunov} and \ref{assumption: weighted-poincare}, the following inequality holds
    \begin{align}
        \Var_\nu(f) \leq \frac{1 + \br{\sum_{q \in \cS} b(q) \nu_K(q)} C_b}{\gamma} \langle f, -\cL f \rangle_\nu \, \forall f \in \ell_{2,\nu},
    \end{align}
    where $C_b$ is the weighted local Poincar\'e constant in Assumption \ref{assumption: weighted-poincare}, $\nu$ is the stationary distribution of the CTMC and $b: \cS \rightarrow [0,\infty)$ is the positive drift term in Assumption \ref{assumption: foster-lyapunov}.
\end{theorem}

The Stitching Theorem (Theorem \ref{theorem: lyapunov-poincare}) is applicable for irreducible, aperiodic, positive recurrent, and reversible CTMCs. Our Stitching Theorem is a generalization of Theorem 1 in \citep{taghvaei-lyapunov-poincare} for countable state space CTMCs, and this generalization is crucial in getting the tight mixing rate in the Super-Halfin-Whitt regime all the way to $\alpha > 1/2$ instead of $\alpha > 2/3$. We defer this discussion to \eqref{eqn: stitching-error-super-hw} and the proof of Theorem \ref{theorem: lyapunov-poincare} to Appendix \ref{sssec: stitching-theorem-proof}. For the special case of $b = B 1_K$ for some finite $K \subseteq \cS$ and $B>0$, Assumption \ref{assumption: weighted-poincare} gives the local Poincar\'e inequality
\begin{align}
    \label{eqn: local-poincare}
    \Var_{\nu_K}(f) \leq \frac{C_L}{\nu(K)} \langle f,-\cL f\rangle_{\nu},
\end{align}
where $\nu_K(x) = \frac{\nu(x)}{\sum_{x \in K} \nu(x)} \forall x \in K$. And so, this allows us to recover the result of \citep{taghvaei-lyapunov-poincare} as stated below.

\begin{corollary}
    \label{corollary: lyapunov-poincare-constant-b}
    Under Assumptions \ref{assumption: ctmc} and \ref{assumption: foster-lyapunov}, suppose that the Lyapunov drift condition holds with $b(q)=B 1_{K}(q)$, where \(B>0\) and \(K\subseteq\cS\) is finite. Furthermore, suppose that Assumption \ref{assumption: weighted-poincare} holds for this choice of \(b\) with constant \(C_L>0\).
    Then, the CTMC admits the Poincar\'e constant $\frac{1+B C_L}{\gamma}$.
\end{corollary}

A version of Corollary \ref{corollary: lyapunov-poincare-constant-b} is known in \citep{taghvaei-lyapunov-poincare} for reversible Markov chains and \citep{Bakry-simple-proof-lyapunov-poincare-2008} for Markov processes. Readers can find the proof of Corollary \ref{corollary: lyapunov-poincare-constant-b} in Appendix \ref{sssec: proof-of-constant-b-corollary}, which is by applying the result of Theorem \ref{theorem: lyapunov-poincare} for $b = B 1_K$. The constant $C_L$ is called the local Poincar\'e constant w.r.t. the finite set $K$ of the CTMC. Intuitively, the CTMC mixes fast outside the finite set $K$ due to the negative drift, in addition to the rapid mixing within $K$ ensured by the ``local'' Poincar\'e inequality. In many applications, establishing a tight characterization of the local mixing behavior is the key to obtaining the tight mixing rate, and previously, the local Poincar\'e constant is predominantly obtained using the minorization condition \citep{rosenthal1995minorization, meyn-tweedie-exponential-ergodicity-1995, taghvaei-lyapunov-poincare}. However, it is known that the approach using the minorization condition often yields suboptimal bounds, and the minorization condition itself is usually hard to obtain as well \citep{qianqin2020minorizationlimitations}. Moreover, whenever the bounded set is not a singleton, there will be a "stitching" error when combining the mixing behavior inside the finite set with the mixing behavior outside the finite set. Therefore, our Stitching Theorem (Theorem \ref{theorem: lyapunov-poincare}) is a necessary generalization to achieve a tight characterization of the local mixing behavior inside the finite set, whereas a naive application of Corollary \ref{corollary: lyapunov-poincare-constant-b} would not yield such a result. For a more detailed explanation, we refer the readers to the discussion around Equation \eqref{eqn: stitching-error-super-hw} as an example.

Comparing Corollary \ref{corollary: lyapunov-poincare-constant-b} with Proposition \ref{prop: lyapunov-poincare-singleton}, one can see that when we try to "stitch" negative drift outside of the finite set $K$ with the local mixing condition behavior inside the finite set, we incur a "stitching error" term $B C_L$ which hurts our Poincar\'e constant. Ideally, to get a tight mixing bound, we would like to make this stitching error term as small as possible or make it vanish to $0$ as $n$ goes to infinity. And so, it is crucial to choose the right Lyapunov function so that we get that right negative drift rate $\gamma$ while minimizing $B$, as well as developing methods to get a tight local mixing constant $C_L$. We shall elaborate more on these steps in the following Section \ref{sec: proof-sketch}.

\section{Proof sketch of main results}
\label{sec: proof-sketch}
Let $K = \{q \in \cS: b(q) > 0\}$, the Lyapunov-Poincar\'e machinery requires the following:
\begin{itemize}
    \item \textbf{Step 1}: Obtain the correct Lyapunov drift outside of some finite set $K$, which ensures rapid mixing outside $K$.%
    \item \textbf{Step 2}: Obtain local mixing guarantees inside the finite set $K$ by exploiting the properties of $K$ and $b$.
\end{itemize}
Recall that the $M/M/n$ queue experiences a phase transition at the Halfin-Whitt regime, i.e. $\alpha = \frac{1}{2}$ \citep{HalfinWhitt1981HeavyTrafficLF}. And so, to prove our main results, we will take different approaches and construct different Lyapunov functions for each of these regimes. The intuitions behind our approaches and the reasons why these regimes require two different approaches will become apparent as we go through the subsections below.

\begin{figure}[ht!]
    \centering
    \begin{subfigure}[t]{0.45\textwidth}
        \centering
        \includegraphics[width=\linewidth,height=1.8in,keepaspectratio]{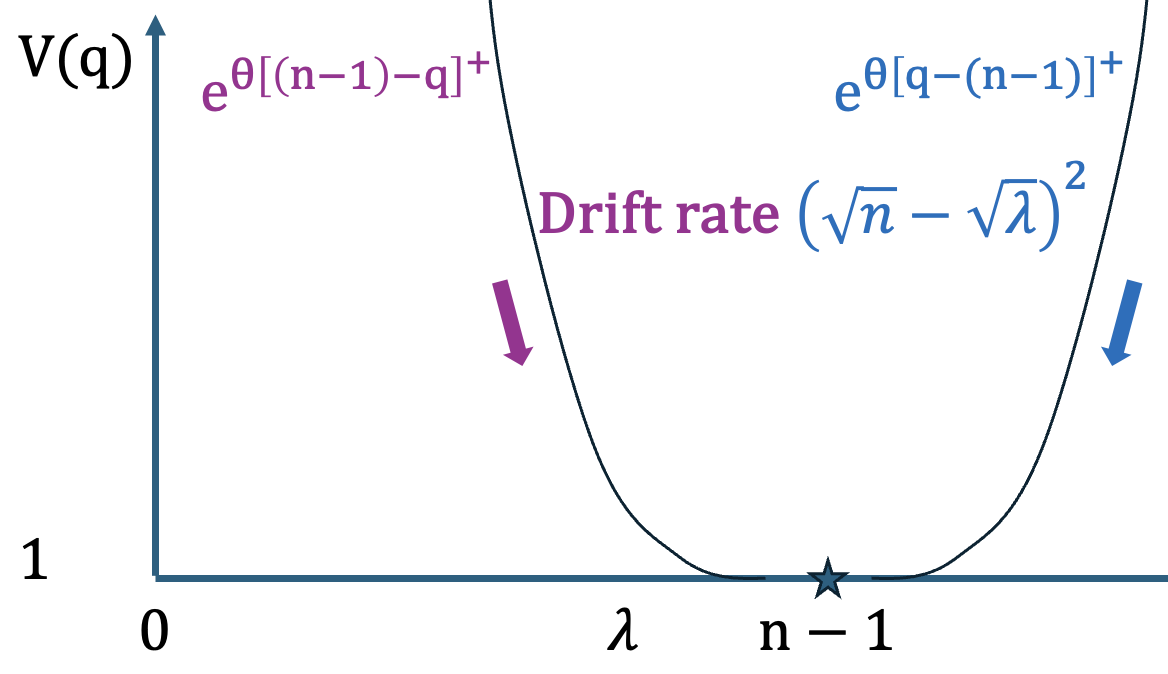}
        \caption{}
        \label{subfig: super-hw-lyapunov}
    \end{subfigure}%
    \hfill
    \begin{subfigure}[t]{0.45\textwidth}
        \centering
        \includegraphics[width=\linewidth,height=1.8in,keepaspectratio]{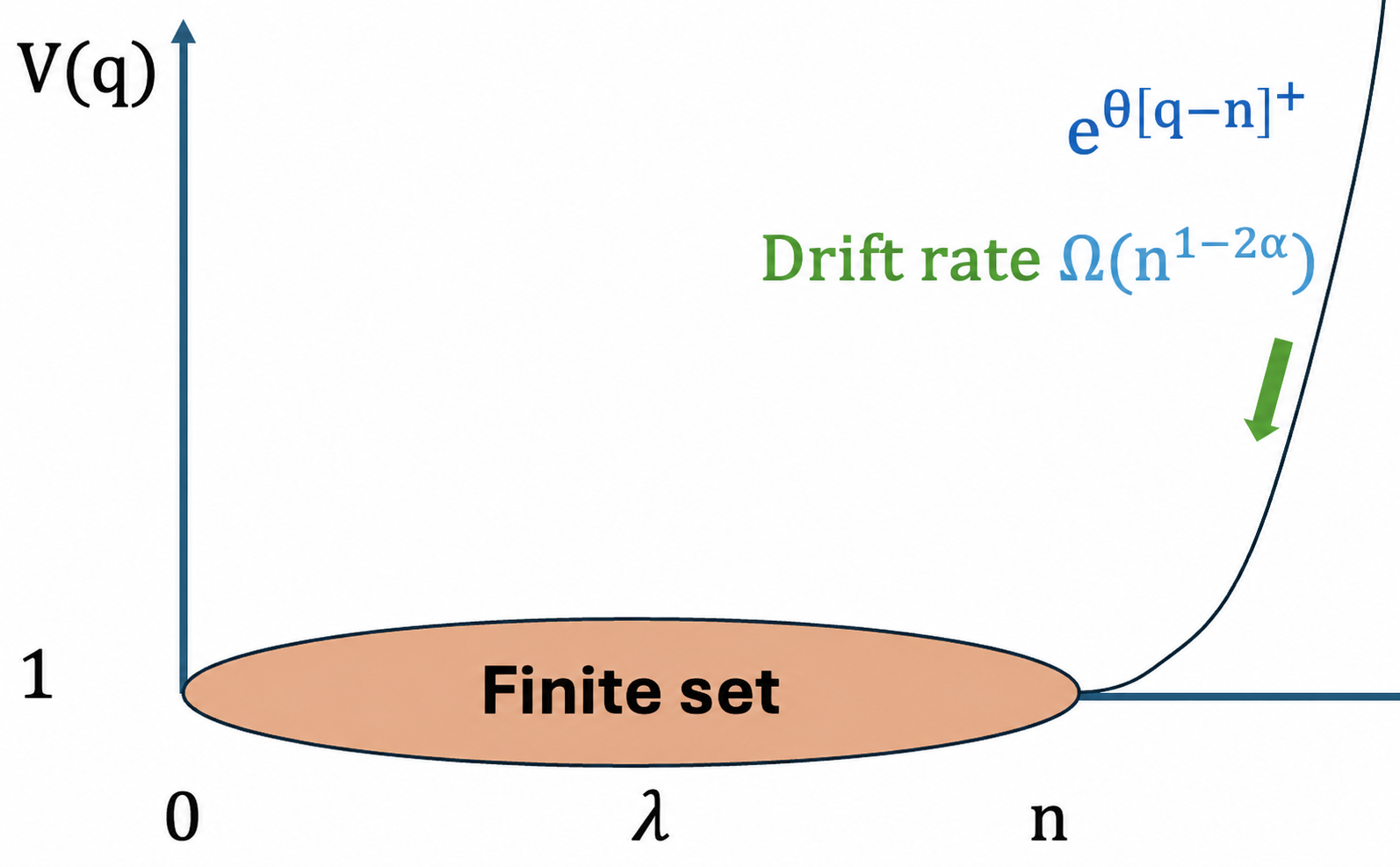}
        \caption{}
        \label{subfig: light-traffic-lyapunov}
    \end{subfigure}
    \caption{Visualization of chosen Lyapunov functions for the Super-Halfin-Whitt regime and the Sub-Halfin-Whitt regime. Figure \ref{subfig: super-hw-lyapunov} corresponds to the piecewise exponential Lyapunov function in the Super-Halfin-Whitt regime. Figure \ref{subfig: light-traffic-lyapunov} corresponds to the constant and exponential Lyapunov function in the Mean-Field and the Sub-Halfin-Whitt regime.}
\end{figure}

\subsection{Regime 1: \texorpdfstring{$\alpha \in [1,\infty)$}{alpha in [1,infinity)}}
\label{ssec: proof-sketch-alpha-geq-1}
In \textbf{Step 1}, we wish to choose an appropriate Lyapunov function to correctly capture the drift. To gain some intuition, let's start with an $M/M/1$ queue {whose arrival rate is $\lambda$ and unit service rate}. The $M/M/1$ queue has the negative drift $\lambda-1 = -\varepsilon$ everywhere except when the queue is empty. And thus, to obtain exponential convergence, we choose the Lyapunov function $V(q) = e^{\theta q}$. Intuitively, the fluid analog of the $M/M/1$ queue would then satisfy $\dot{V} \approx - \varepsilon \theta V$, providing the negative drift as required. More precisely, for $\theta = -\frac{\log \lambda}{2}$, we obtain
\begin{align}
    \label{eqn: mm1-correct-drift-equation}
    \cL V \leq -\br{1-\sqrt{\lambda}}^2 V + b 1_{\{0\}}
\end{align}
where $b$ is some universal constant. The rate of $\br{1-\sqrt{\lambda}}^2$ in the above equation matches previously established mixing results \citep{Van_Doorn1985-decay-bounds, philippe-robert-stochastic-networks-and-queues} for the $M/M/1$ queue, suggesting that $V$ is indeed the correct Lyapunov function. As seen in \eqref{eqn: mm1-correct-drift-equation}, we have a negative drift everywhere except the singleton $\{0\}$, so we simply apply Proposition \ref{prop: lyapunov-poincare-singleton} to obtain the Poincar\'e constant, which gives us the desired mixing rate $\rate \geq \br{1-\sqrt{\lambda}}^2$.

For the $M/M/n$ queue, the system dynamics are more complicated since there is another upward drift to $\lambda$ when $q < \lambda$ in addition to the downward drift when $q > \lambda$. So, we appropriately adjust our Lyapunov function and define it as follows:
\begin{align}
    \label{eqn: mmn-lyapunov}
    V(q) = e^{\theta \sqbr{q-(n-1)}^+ + \theta \sqbr{(n-1)-q}^+}
\end{align}
with appropriately chosen $\theta > 0$ to accurately capture the negative drift above $\lambda$ and the positive drift below $\lambda$. While $\lambda$ is the center of the drift, we consider the distance of the queue length to $n-1$ in the exponent because of algebraic convenience in the proof. From this choice of Lyapunov function, we obtain the following drift lemma for the $\alpha \geq 1$ regime.

\begin{lemma}
    \label{lemma: drift-lemma-alpha-geq-1}
    Let $V$ be defined as in \eqref{eqn: mmn-lyapunov} and $\cL$ be the generator of the $M/M/n$ queue in the $\alpha \in [1,\infty)$ regime with arrival rate $\lambda = n-n^{1-\alpha}$. Then, there exists $b \geq 0$ such that:
    \begin{align}
        \label{eqn: alpha-geq-1-drift-equation}
        \cL V(q) \leq -\br{\sqrt{n}-\sqrt{\lambda}}^2 V(q) + b 1_{\{n-1\}} \, \forall q \in \Z_{\geq 0}.
    \end{align}
\end{lemma}
Observe that we have a negative drift outside of some singleton when $\alpha \geq 1$, and so applying Proposition \ref{prop: lyapunov-poincare-singleton} directly gives us the mixing rate lower bound.
\proof{\textit{Proof of Theorem \ref{thm: unifying-theorem} for $\alpha \geq 1$:}}
    From Proposition \ref{prop: lyapunov-poincare-singleton} and Lemma \ref{lemma: drift-lemma-alpha-geq-1}, we have that the system admits the Poincar\'e constant $C_P = \frac{1}{\br{\sqrt{n}-\sqrt{\lambda}}^2}$. And so, from Proposition \ref{prop: poincare-equals-mixing}, we have $\rate \geq \br{\sqrt{n}-\sqrt{\lambda}}^2$. $\square$
\endproof

\subsection{Regime 2: \texorpdfstring{$\alpha \in [1/2,1)$}{alpha in [1/2,1)}}
\label{ssec:lyapunov-canonical-path-approach}
For $\alpha \in [1/2,1)$, the Lyapunov function \eqref{eqn: mmn-lyapunov} no longer gives a negative drift outside of a singleton, but rather a finite set $K$ whose cardinality can grow as $n \to \infty$. This is because as the queue length gets closer to the center of the drift $\lambda$, the drift gets weaker, which would not be enough to give a competitive mixing bound. Hence, we have the following drift lemma for the $\alpha \in [1/2,1)$ regime.

\begin{lemma}
    \label{lemma: drift-lemma-super-HW}
    Let $V$ be defined as in \eqref{eqn: mmn-lyapunov} and $\cL$ be the generator of the $M/M/n$ queue in the $\alpha \in [1/2,\infty)$ regime with arrival rate $\lambda = n-n^{1-\alpha}$. Then, there exists a function $b: \Z_{\geq 0} \rightarrow \R_{\geq 0}$ such that:
    \begin{align}
        \label{eqn: drift-equation}
        \cL V(q) \leq -\br{\sqrt{n}-\sqrt{\lambda}}^2 V(q) + b(q) \, \forall q \in \Z_{\geq 0}
    \end{align}
    where
    \begin{align}
        \label{eqn: super-hw-finite-set}
        K = \left\{\min\left\{n-1,\sqbr{\lfloor 2\lambda \rfloor - n + 1}^+\right\},\ldots,n-1\right\}.
    \end{align}
    For $\alpha \in [1/2,1)$, this is $K = \{\sqbr{\lfloor 2\lambda \rfloor - n + 1}^+,\ldots,n-1\}$; for $\alpha \geq 1$, the same definition gives $K = \{n-1\}$. Moreover, we also have
    \begin{align}
        \label{eqn: b-term-super-hw-drift}
        b(q) =
        \begin{cases}
            \frac{n^{1-2\alpha}}{1-n^{-\alpha}} e^{\frac{n^{1-2\alpha}}{1-n^{-\alpha}}} \,\, &\forall q \in K-\{n-1\} \\
            n^{1-\alpha}\bigg(1 + \frac{3}{4(n^\alpha-1)}\bigg) &\text{if } q = n-1 \\
             0 \,\, &\forall q \not \in K.
        \end{cases}
    \end{align}
\end{lemma}
The above lemma completes \textbf{Step 1} of our methodology by establishing a negative drift with rate $\br{\sqrt{n}-\sqrt{\lambda}}^2$ outside the finite set $K$, the detailed proof of this Lemma is deferred to Appendix \ref{ssec: proof-detail-drift-lemma-super-hw}.  Unlike \citep{taghvaei-lyapunov-poincare}, the function $b$ in our drift condition \eqref{eqn: drift-equation} is not a constant function, which allows a more refined analysis and precipitates the use of our Stitching Theorem \ref{theorem: lyapunov-poincare}. To understand why the chosen Lyapunov function works for the  $M/M/n$ queue, consider the following ODE that is a fluid analog of $M/M/n$ queue: $\dot{q} = \lambda - \min \{q, n\}$. When $q \geq \lambda +n^{1-\alpha}$ and $\theta_1 \approx n^{-\alpha}$, we have $\dot{q} = \lambda-n = -n^{1-\alpha}$ which gives $\dot{V} \leq -\br{\sqrt{n}-\sqrt{\lambda}}^2 V = -\Theta(n^{1-2\alpha}) V$, which is the desired negative drift. Similarly, when $q \leq \lambda - n^{1-\alpha}$, we also have $\dot{V} \leq -\br{\sqrt{n}-\sqrt{\lambda}}^2 V$. This allows us to establish a negative drift with rate $-\br{\sqrt{n}-\sqrt{\lambda}}^2$ outside of a set $K$ with cardinality $|K| = \max\{1,\lceil 2n^{1-\alpha} \rceil-1\}$.

Observe that for $\alpha \geq 1$, we have $K = \{n-1\}$ is a singleton, whereas $|K|$ can scale with order $n^{1-\alpha}$ when $\alpha \in (1/2,1)$. And so, we come to \textbf{Step 2}, where we wish to show that the system mixes quickly in the finite set $K$ by showing that the system admits a local Poincar\'e inequality inside the finite set. To this end, we apply the local canonical path method (described in Lemma \ref{lemma: canonical-path}) to show that the weighted local Poincar\'e constant $C_b = \Theta(n^{2-4\alpha})$ in Lemma \ref{lemma: weighted-poincare-super-HW}, which means that $C_L$ goes to zero as $n$ goes to infinity and allows us to show rapid mixing in $K$.

\textbf{Reasons why we need a weighted local Poincar\'e analysis}: Note that from \eqref{eqn: b-term-super-hw-drift}, if we choose $B = \max_{q \in \Z_{\geq 0}} b(q)$ and then naively apply Theorem 1 in \citep{taghvaei-lyapunov-poincare} (also Corollary \ref{corollary: lyapunov-poincare-constant-b} in our work), we get $B = b(n-1) = O(n^{1-\alpha})$. Combining with the fact that the unweighted local Poincar\'e constant $C_L = \Theta(n^{1-2\alpha})$, we would get the Poincar\'e constant
\begin{align}
    \label{eqn: stitching-error-super-hw}
    C_P = \frac{1+O\br{n^{2-3\alpha}}}{(\sqrt{n}-\sqrt{\lambda})^2}
\end{align}
which matches the limiting behavior as $n \rightarrow \infty$ only when $\alpha \in (2/3,\infty)$ rather than for the entire Super-Halfin-Whitt regime $\alpha \in (1/2,\infty)$. Thus, naively applying previously proposed {``stitching theorems"} would not yield the desired bound. To address this gap, we need a more refined local mixing analysis, by treating the $b$ term carefully as follows.

\begin{lemma}
    \label{lemma: weighted-poincare-super-HW}
    (Weighted Poincar\'e for $M/M/n$ in the Super-Halfin-Whitt regime) Under Assumptions \ref{assumption: ctmc}, \ref{assumption: foster-lyapunov}, let $b$ be given by \eqref{eqn: b-term-super-hw-drift} and $K = \{\sqbr{\lfloor 2\lambda \rfloor - n + 1}^+,\ldots,n-1\}$. In addition, denote $\nu_n$ as the stationary distribution, we have Assumption \ref{assumption: weighted-poincare} holds with {$C_b = b_1 n^{2-4\alpha} + b_2 n^{3-6\alpha}$ where $b_1, b_2 > 0$ are some universal constants}
    for $\alpha \in [1/2,1)$ and for every test function $f \in \ell_{2, \nu_n}$.
\end{lemma}

The proof and the precise characterization of the constant $C_b$ are provided in the Appendix \ref{sssec: proof-of-weighted-poincare-super-HW}. Since $\alpha > 1/2$, we have $C_b \rightarrow 0$ as $n \rightarrow \infty$, which allows our finite-time mixing bound to match the limiting behavior at the asymptotic. It is evident that Lemma \ref{lemma: weighted-poincare-super-HW} requires a method that can establish a tight weighted Poincar\'e constant bound, but it is unclear how this can be achieved using previously known methods such as the minorization condition. Hence, we propose a local canonical path method to obtain a weighted local Poincar\'e inequality as follows.

\subsubsection{Weighted local Poincar\'e bounds via the local canonical path method}
{ Under the observation that the Markov chain behaves like a symmetric random walk inside $K$, we shall adapt the classical canonical path method in finite-state mixing literature, which can give tight mixing bounds for symmetric random walks on path graphs, and propose a local version as follows.}
\begin{lemma}
    \label{lemma: canonical-path}
    (Local canonical path method) For any birth-and-death process with the generator $\cL$, a subset $K$ of the state space $\Z_{\geq 0}$ that consists of consecutive integers and the stationary distribution $\nu$, we have the following inequality holds:
    \begin{align}
        \Var_{\nu_K}(f) \leq C_L \sum_{k: k,k+1 \in K} Q_K(k,k+1) (f(k+1)-f(k))^2 \, \forall f \in \ell_{2,\nu_K}
    \end{align}
    where the constant $C_L$ is defined as:
    \begin{align}
        \label{eqn: CL-definition}
        C_L = \max_{k: k,k+1 \in K} \frac{\sum_{x,y \in K: x \leq k \leq y-1} |x-y| \nu_K(x) \nu_K(y)}{Q_K(k,k+1)}
    \end{align}
    where $\nu_K(x) = \frac{\nu(x)}{\sum_{x \in K}\nu(x)} \, \forall x \in K$ and $\nu_K(x) = 0$ otherwise and $Q_K(u,v) = \nu_K(u) \cL(u,v) = \nu_K(v) \cL(v,u) \forall u,v \in K$.
\end{lemma}
Here, $C_L$ can also be understood as the congestion ratio, which is lower bounded by the inverse mixing rate \citep{Levin2017-mixing-book}. When $\nu_K$ is roughly uniform and the transitions are roughly symmetric, the constant $C_L$ is proportional to the sum of the path length of the path passes through an edge {$e_k = (k,k+1)$} that maximizes the constant.  The full proof of Lemma \ref{lemma: canonical-path} is deferred to Appendix \ref{sssec: local-canonical-path-proof}.

\subsubsection{Proof sketch of Lemma \ref{lemma: weighted-poincare-super-HW}}
\label{sssec: weighted-poincare-super-hw-proof-sketch}
To prove Lemma \ref{lemma: weighted-poincare-super-HW}, recall that $\vartheta(\cdot) = \frac{b(\cdot) \nu_K(\cdot)}{\sum_{q \in K} b(q) \nu_K(q)}$, and we want to prove a bound in the form of
\begin{align*}
    \Var_\vartheta(f) \leq \frac{C_b}{\nu_n(K)} \cE(f,f)
\end{align*}
for some constant $C_b > 0$. Denote $\tau(\cdot) = b(\cdot) \nu_K(\cdot)$, we have the following decomposition
\begin{align*}
    \Var_\vartheta(f) &= \frac{1}{2}\sum_{x,y \in K} \vartheta (x) \vartheta(y) (f(x)-f(y))^2 = \frac{1}{2}\sum_{x,y \in K} \frac{\tau (x) \tau(y)}{\tau(K)^2} (f(x)-f(y))^2 \\
    &= \frac{1}{\tau(K)^2} \sqbr{\underbrace{\sum_{x \in K} \tau(n-1) \tau(x) \br{f(n-1)-f(x)}^2}_{T_1} + \underbrace{\frac{1}{2}\sum_{x,y \in K-\{n-1\}} \tau(x) \tau(y) (f(x)-f(y))^2}_{T_2}}
\end{align*}
since we need to address the stitching error at the boundary term $n-1$. To set up the bound, we first need to study the steady-state distribution inside the finite set, that is to show that $\nu_K$ is a roughly uniform distribution whose support is the set $K$ as follows.
\begin{lemma}
    \label{lemma: roughly-uniform}
    Let {$\lambda = n-n^{1-\alpha}, K = \{\sqbr{\lfloor 2\lambda \rfloor - n + 1}^+,\ldots,n-1\}$} and $\nu_K(x) = \frac{\nu(x)}{\sum_{x \in K} \nu(x)} \forall x \in K$ and $\nu_K(x) = 0$ otherwise. For $\alpha \in [1/2,1)$ and $n \geq 110$, we have $L_K(n) \times n^{\alpha-1} \leq \nu_K(x) \leq U_K(n) \times n^{\alpha-1}$ where $\lim_{n \rightarrow \infty} L_K(n) = \lim_{n \rightarrow \infty} U_K(n) = 1/2$ {for $\alpha \in (1/2,1)$ and $\lim_{n \rightarrow \infty} L_K(n) = 1/(2e)$, $\lim_{n \rightarrow \infty} U_K(n) = e/2$ for $\alpha = 1/2$}.
\end{lemma}
The main idea of the proof is to apply Stirling's approximation and the fact that the steady state is roughly uniform inside $K$. We defer the full statement and proof of the lemma to Appendix \ref{sssec: roughly-uniform-lemma-proof}. Now that we have shown that the steady state inside $K$ is roughly uniform, i.e., $\nu_K(x) = \Theta(n^{\alpha-1})$, we have $\tau(x) = \Theta(n^{-\alpha})$ for $x \in K \setminus \{n-1\}$ and $\tau(n-1) = O(1)$ from \eqref{eqn: b-term-super-hw-drift}. Hence,
\begin{align*}
    T_1 &= O(n^{-\alpha}) \sum_{x \in K} (f(n-1)-f(x))^2 \leq O(n^{2-3\alpha}) \sum_{x: x,x+1 \in K} (f(x+1)-f(x))^2 \\
    &= O(n^{2-4\alpha}) \sum_{x: x,x+1 \in K} \lambda \nu_K(x) (f(x+1)-f(x))^2 \leq \frac{O(n^{2-4\alpha})}{\nu_n(K)} \cE(f,f)
\end{align*}
from the Cauchy-Schwarz inequality. On the other hand, within $K \setminus \{n-1\}$, the system behaves locally like a symmetric random walk on a path graph. Applying the local canonical path method therefore yields a local mixing-time bound of order $O\br{\frac{d^2}{n}} = O(n^{1-2\alpha})$, where $d = O(n^{1-\alpha})$ is the path length and $\Theta(n)$ is the total transition rate of the system. Hence, we obtain
\begin{align*}
    T_2 = O\br{n^{2-4\alpha}} \sum_{x,y \in K \setminus \{n-1\}} \nu_K(x) \nu_K(y) (f(x)-f(y))^2 = \frac{O\br{n^{3-6\alpha}}}{\nu_n(K)} \cE(f,f),
\end{align*}
as desired. For the full proof of Lemma \ref{lemma: weighted-poincare-super-HW}, we refer the readers to Appendix \ref{sssec: proof-of-weighted-poincare-super-HW}.

\subsubsection{Apply Stitching Theorem \ref{theorem: lyapunov-poincare}}
In the final step, we combine the results in Lemma \ref{lemma: drift-lemma-super-HW} and Lemma \ref{lemma: weighted-poincare-super-HW} and apply Theorem \ref{theorem: lyapunov-poincare} to obtain the Super-Halfin-Whitt mixing result in Theorem \ref{thm: unifying-theorem}. { Before applying Theorem \ref{theorem: lyapunov-poincare}, we need one final ingredient, which is a bound on the $\sum_{q \in K} \underbrace{b(q) \nu_K(q)}_{\tau(q)}$ term in Theorem \ref{theorem: lyapunov-poincare} via the following claim.}
\begin{claim}
    \label{claim: tau-bound-super-hw} For $\alpha \in [1/2,1)$ and $n \geq n_0$, let $K = \{\sqbr{\lfloor 2\lambda\rfloor - n + 1}^+,\ldots,n-1\}$, we have {$L_\tau \leq \sum_{q \in K} \tau(q) \leq C_\tau$ for some constants $L_\tau, C_\tau > 0$}. Here, we have $\lim_{n \rightarrow \infty} \sum_{q \in K} \tau(q) = 1/2$ for $\alpha \in (1/2,1)$, and $\frac{1}{2e}+1 \leq \liminf_{n \rightarrow \infty} \sum_{q \in K} \tau(q)$ and $\limsup_{n \rightarrow \infty} \sum_{q \in K} \tau(q) \leq e+1$ for $\alpha = 1/2$.
\end{claim}

Claim \ref{claim: tau-bound-super-hw} essentially says that the weighted contribution of $b$ is $\Theta(1)$, which is a significant improvement from the $O(n^{1-\alpha})$ bound when setting $B = \max_{q \in K} b(q)$ and naively applying Corollary \ref{corollary: lyapunov-poincare-constant-b}. With Lemma \ref{lemma: drift-lemma-super-HW}, Lemma \ref{lemma: weighted-poincare-super-HW} and Claim \ref{claim: tau-bound-super-hw} (whose proof is in Appendix \ref{sssec: tau-bound-claim-proof}), we obtain the mixing rate of $M/M/n$ in the Super-Halfin-Whitt regime and $M/M/1$ as follows.

\proof{\textit{Proof of Theorem \ref{thm: unifying-theorem} for $\alpha \in (1/2,1)$:}}
    From Lemma \ref{lemma: drift-lemma-super-HW}, Lemma \ref{lemma: weighted-poincare-super-HW}, Claim \ref{claim: tau-bound-super-hw} and Theorem \ref{theorem: lyapunov-poincare}, we obtain the Poincar\'e constant to be:
    \begin{align*}
        C_P &= \frac{1+ \br{\sum_{q \in K} b(q) \nu_K(q)} C_b}{(\sqrt{n}-\sqrt{\lambda})^2} \leq \frac{1 + 1.04 \times 10^5 n^{2-4\alpha} + 1.6 \times 10^5 n^{3-6\alpha}}{(\sqrt{n}-\sqrt{\lambda})^2}
    \end{align*}
    From Proposition \ref{prop: poincare-equals-mixing} we have
    \begin{align}
        \rate \geq \frac{(\sqrt{n}-\sqrt{\lambda})^2}{1 + 1.04 \times 10^5 n^{2-4\alpha} + 1.6 \times 10^5 n^{3-6\alpha}}.
    \end{align}
$\square$ \endproof

\textbf{Remark}: Previously, several works have attempted different methods to handle the local mixing behavior of the finite set. In particular, \citep{meyn1994computable,Baxendale_convergence_2005,taghvaei-lyapunov-poincare} rely on the minorization assumption, which tells us that there exists a common nonzero measure dominated uniformly by transition laws from states in $K$, to establish some form of local mixing guarantee, such as a local Poincar\'e inequality, inside the finite set. However, obtaining an explicit minorization constant with sufficiently sharp dependence on the system parameters (such as dimensionality \citep{qianqin2020minorizationlimitations}) can be analytically challenging. In addition, a few other works have obtained local Poincar\'e inequalities using the size of the local region \citep{raginsky2017nonconvexsgld,divol-niles-weed2024optimal-transport-estimation}, but generally, the local Poincar\'e constant bounds in these works are conservative. Thus, as our goal is to obtain a tight characterization of the mixing rate, we take a different approach and propose a local version of the canonical path method that exploits the birth-and-death structure of the transition matrix. In fact, the non-local version has been used to bound the mixing time of random walks on different graph topologies \citep{UBCmixingbook}.

In summary, the mixing rate $(\sqrt{n}-\sqrt{\lambda})^2$ in the regime $\alpha \geq 1$ is tight and matches the result in \citep{zeifman_lognorm}. In addition, the mixing rate for the regime $\alpha \in (1/2,1)$ is asymptotically tight, i.e., it matches with $(\sqrt{n}-\sqrt{\lambda})^2$ for large $n$. For $\alpha = 1/2$, the Poincar\'e constant inside the finite set $K$ does not vanish. We show that it is bounded above by a constant, implying that the mixing rate is bounded below by a universal constant.

\subsection{Regime 3: \texorpdfstring{$\alpha \in (0,1/2)$}{alpha in (0,1/2)}}
\label{ssec:lyapunov-coupling-approach}

Recall that in the proof of the Super Halfin-Whitt regime, we used a two-sided exponential Lyapunov function, which resulted in the finite set $K = \{\min\{n-1,\sqbr{\lfloor 2\lambda_n \rfloor-n+1}^+\},\ldots,n-1\}$ and the negative drift rate $\Theta(n^{1-2\alpha})$ outside of $K$. In the Super-Halfin-Whitt regime, this choice of Lyapunov function gives us $b(q) = O\br{n^{1-2\alpha} e^{n^{1-2\alpha}}} \, \forall q \neq n-1$, which vanishes to $0$ for $\alpha \in (1/2,\infty)$. However, this term can grow at a super-polynomial rate when $\alpha \in (0,1/2)$, which produces a suboptimal guarantee for $\rate$.

Thus, the above insight motivates us to design the Lyapunov function differently so that it gives rise to a better finite set and a better positive drift term $b$ so that we can capture the local mixing behavior more accurately. Observe that in the Sub-Halfin-Whitt regime, the queue length is very unlikely to be above $n$ since the drift above $n$ is very strong. From this intuition, we choose the Lyapunov function $V(q) = e^{\theta [q-n]^+}$, which captures the $M/M/1$-like drift when the queue length is above $n$ and has zero drift when the queue length is below $n$. We have the following drift lemma for the Sub-Halfin-Whitt regime.
\begin{lemma}
    \label{lemma: drift-lemma-sub-hw-mean-field}
    Consider the continuous-time $M/M/n$ system as in Subsection \ref{ssec:model} with arrival rate $\lambda \in (0,n)$. Let $z \in \br{1, \frac{n}{\lambda}}$ and define
    \begin{align}
        V(q) =
        \begin{cases}
            1, & 0\leq q<n,\\
            z^{q-n}, & q\geq n.
        \end{cases}
    \end{align}
    Then, we have
    \begin{align}
        \cL V(q) \leq -\gamma V(q) + b(q),
        \qquad
        \forall q\in \Z_{\geq 0},
    \end{align}
    where $\gamma = \lambda(z-1)\br{\frac{n}{\lambda z}-1} > 0$ and
    \begin{align}
        \label{eqn: light-traffic-b-function}
        b(q) &=
        \begin{cases}
            \gamma, & 0\leq q<n,\\
            \gamma+\lambda(z-1), & q=n,\\
            0, & q>n.
        \end{cases}
    \end{align}
\end{lemma}
Evidently, this choice of Lyapunov function gives a smaller $b$ term at the expense of a larger finite set $K = \{0,1,...,n\}$. Thus, simply using the local canonical path method as in the Super-Halfin-Whitt regime would not be sufficient to obtain tight Poincar\'e constant bounds. Instead, observe that when the queue length does not exceed $n$, the system behaves like an $M/M/\infty$, which has a known Poincar\'e constant \citep{Klaassen1985OnAI,kontoyiannis-entropy-law-of-small-number,Chafa_2006_entropic_inequalities_infty_queue}. Thus, the set $K = \{0,1,...,n\}$ is a good approximation of the $M/M/\infty$ behavior. To formalize this, we have the following truncation lemma.
\begin{lemma}
    \label{lemma: truncated poincare}
    Let $K = \{m,...,M\}$ be a connected subset of the state space $\cS \subset \Z$ and $\nu$ be the stationary distribution {of the birth-and-death process} and $\nu_K(x) = \nu(x) 1_K(x)/\sum_{x \in K} \nu(x) \, \forall x \in \cS$. Furthermore, let $\cL$ be the generator and assume that $\Var_\nu(f) \leq C_P \langle f,-\cL f\rangle_\nu \, \forall f \in \ell_{2,\nu}$, we have
    \begin{align}
        \label{eqn: poincare-inequality-entire-chain}
        \Var_{\nu_K}(f) \leq C_P \sum_{x:x,x+1 \in K} Q_K(x,x+1) \br{f(x)-f(x+1)}^2 \, \forall f \in \ell_{2,\nu_K}
    \end{align}
    where $Q_K(x,y) = \nu_K(x) \cL(x,y) \, \forall x,y \in \cS$.
\end{lemma}
In words, the above lemma allows us to obtain local mixing results if we know the mixing behavior of another system that subsumes the dynamics inside the finite set. The above lemma then allows us to establish the local Poincar\'e constant of $\approx 1$ from the known Poisson Poincar\'e constant, and we call this approach the ``truncation argument'' to obtain a local Poincar\'e inequality.

Similarly to the $\alpha \in [1/2,1)$ case, naively setting $B = \max_q b(q)$ and applying Corollary \ref{corollary: lyapunov-poincare-constant-b} will not be enough to obtain asymptotically tight bounds for the entire Sub-Halfin-Whitt regime. Hence, we need to employ a weighted local mixing analysis to show the mixing rate approaches $1$ all the way to the phase transition point. To achieve this, the key observation is that $\nu_K(n) = e^{-\Omega(n^{1-2\alpha})}$, which easily diminishes the additional $\lambda (z-1)$ term that comes from stitching the exponential Lyapunov function with a constant function. Hence, with some careful analysis, we can obtain a tight weighted local Poincar\'e constant $C_b \rightarrow 1$ as $n \rightarrow \infty$, and we are done.

\textbf{Remark on the choice of Lyapunov function}: While at the first glance, the Lyapunov function should be $\big|q-\lambda_n\big|$, as this corresponds to the first non-trivial Charlier polynomial $x-\lambda_n$ for $\mu = 1$ \cite{dominguez-book-orthogonal-polynomial}, this approach would in fact not lead to the right mixing rate. In particular, if $\lambda_n \not \in \Z^+$, then this Lyapunov function can create a non-singleton finite set, from which a non-vanishing stitching error from a non-zero local Poincar\'e constant would arise. Indeed, the Lyapunov function $|q-\lambda_n|$ would yield $b(q) = \Omega(n)$ for $q \in K$, and the local Poincar\'e constant would be $\Omega(1/n)$ from the finite set having $\Omega(1)$ elements and the transition rate being $O(n)$. Thus, we would have an $\Omega(1)$ stitching error, preventing us from achieving the right mixing rate in the limit. To obtain an asymptotically tight mixing rate, designing a Lyapunov function to perfectly capture the drift is not enough. Instead, one needs to design an appropriate finite set to correctly capture the mixing behavior as well.

\subsection{Proof sketch of the limiting results}
\label{ssec: limiting-results-proof-sketch}
{ To show the mixing rate equality in the asymptotics, proving Poincar\'e inequality is not enough since it only provides a lower bound on $\rate$. Instead, we also need an upper bound as well, which is achieved by providing a lower bound on the best Poincar\'e constant. To that end, it is well-known that the best Poincar\'e constant is the inverse of the spectral gap of the CTMC, and we have the following variational representation \citep{Mu-Fa-Chen1991-l2-convergence,Levin2017-mixing-book}.
\begin{prop}
    \label{prop: spectral-gap-poincare-variational-form}
    Under Assumption \ref{assumption: ctmc}, the spectral gap of the CTMC is characterized by the variational formula
    \begin{align}
        \label{eqn: spectral-gap-variational-form}
        \operatorname{Gap}(\cL)
        &:=
        \inf_{\substack{f\in \ell_{2,\nu}:\\ \Var_\nu(f)>0}}
        \frac{\cE(f,f)}{\Var_\nu(f)} = \frac{1}{C_P^\star},
    \end{align}
    where $C_P^\star$ is the optimal Poincar\'e constant, i.e., the smallest constant for which \eqref{eqn: poincare-inequality} holds with the convention that $1/0=\infty$.
\end{prop}
From the definition of the mixing rate in \eqref{eqn: mixing-rate-form}, Proposition \ref{prop: poincare-equals-mixing} implies that the mixing rate is lower bounded by $\frac{1}{C_P}$ whenever the system admits Poincar\'e constant $C_P$. On the other hand, the variational representation \eqref{eqn: spectral-gap-variational-form} also allows us to obtain an upper bound on the spectral gap given an appropriate choice of test function $f$. Hence, our goal is to design suitable test functions to obtain a matching upper bound depending on the heavy-traffic regime. Inspired by the eigenfunction of the $M/M/1$ system, we choose $g(q) = a^{[q-n]^+}$ for $a \uparrow \sqrt{n/\lambda}$ for $\alpha \in (1/2,\infty)$. On the other hand, it is well-known that the first non-trivial eigenfunction for the $M/M/\infty$ system is a linear function \citep{dominguez-book-orthogonal-polynomial}. Hence, we choose $g(q) = q$ for $\alpha \in (0,1/2)$, which results in matching upper bounds. To complete the proof for the asymptotic upper bound, it is sufficient to compute $\limsup_{n \to \infty} \frac{\cE(g,g)}{\Var_{\nu_n}(g)}$, and we are done.}

\section{Numerical experiments}
\label{sec: experiments}
To empirically validate our theoretical results, we provide finite-$n$ numerical illustrations of the transient mean queue length behavior. While staffing approximations often focus on systems operating near steady state \citep{HalfinWhitt1981HeavyTrafficLF,GreenKolesarWhitt2007}, many service and supply-chain systems in practice are exposed to sudden disruptions and demand shocks. Examples include random demand surges in supply chains, disruption-induced changes in supply availability, call-volume spikes triggered by external events, and emergency-department arrival surges \citep{Tomlin2006,HuangSongTong2016,IbrahimYeLEcuyerShen2016,CurryMoyaErraguntlaBanerjee2022}. Thus, our mean queue length convergence result provides a quantitative estimate of how long it takes the system to dissipate the excess workload created by such a shock, represented by a potentially large initial queue length $q(0)$. In particular, given $\chi(\pi_{n,0},\nu_n)$ can be exponentially large in terms of $q(0)$, Corollary \ref{corollary: mean-queue-length-mmn} suggests that the recovery time scales approximately as
\begin{align}
    \frac{q(0)}{\rate},
\end{align}
up to lower-order and constant factors.

To investigate how our mean queue length estimations perform in practice, we compute the transient distribution by exponentiating a finitely truncated generator, where the truncation level is chosen so that the omitted stationary tail mass is negligible. Let
\begin{align}
    \Delta_n(t) = \left|\E_{\pi_{n,t}}[q]-\E_{\nu_n}[q]\right|.
\end{align}
Our finite-time mean queue length bound in Corollary \ref{corollary: mean-queue-length-mmn} contains the additional $O(\sqrt{n}+n^\alpha)$ pre-exponential factor from the variance estimate, in
addition to the initial $\chi$-distance term. This factor is useful for a rigorous uniform bound, but it can be conservative for predicting the observed finite-time mean queue length. To determine the quality of the exponential rate, we therefore instead plot the rate-only estimator
\begin{align}
    \widehat{\Delta}_{n,R}(t) = e^{-\rate \, \times \, t} \left|\E_{\pi_{n,0}}[q]-\E_{\nu_n}[q]\right|,
\end{align}
where $\rate$ is the limiting mixing rate given by our Theorem \ref{thm: limiting-mixing-rates}, depending on the regime. To investigate the tightness of our result and the empirical transient behavior of the $M/M/n$ system, we shall vary the parameter $\alpha$, the number of servers $n$, and the starting position $q(0)$.

\subsection{Varying \texorpdfstring{$\alpha$}{alpha}}
\label{ssec: varying-alpha}
We first fix $n=500$ and vary the heavy-traffic parameter $\alpha$ across the Sub-Halfin-Whitt regime, the Halfin-Whitt regime, and the Super-Halfin-Whitt regime. For each value of $\alpha$, we compute the empirical decay rate by fitting a line to $\log \Delta_n(t)$ over the simulated time grid.
\begin{figure}[ht!]
    \centering
    \includegraphics[width=0.7\textwidth]{\detokenize{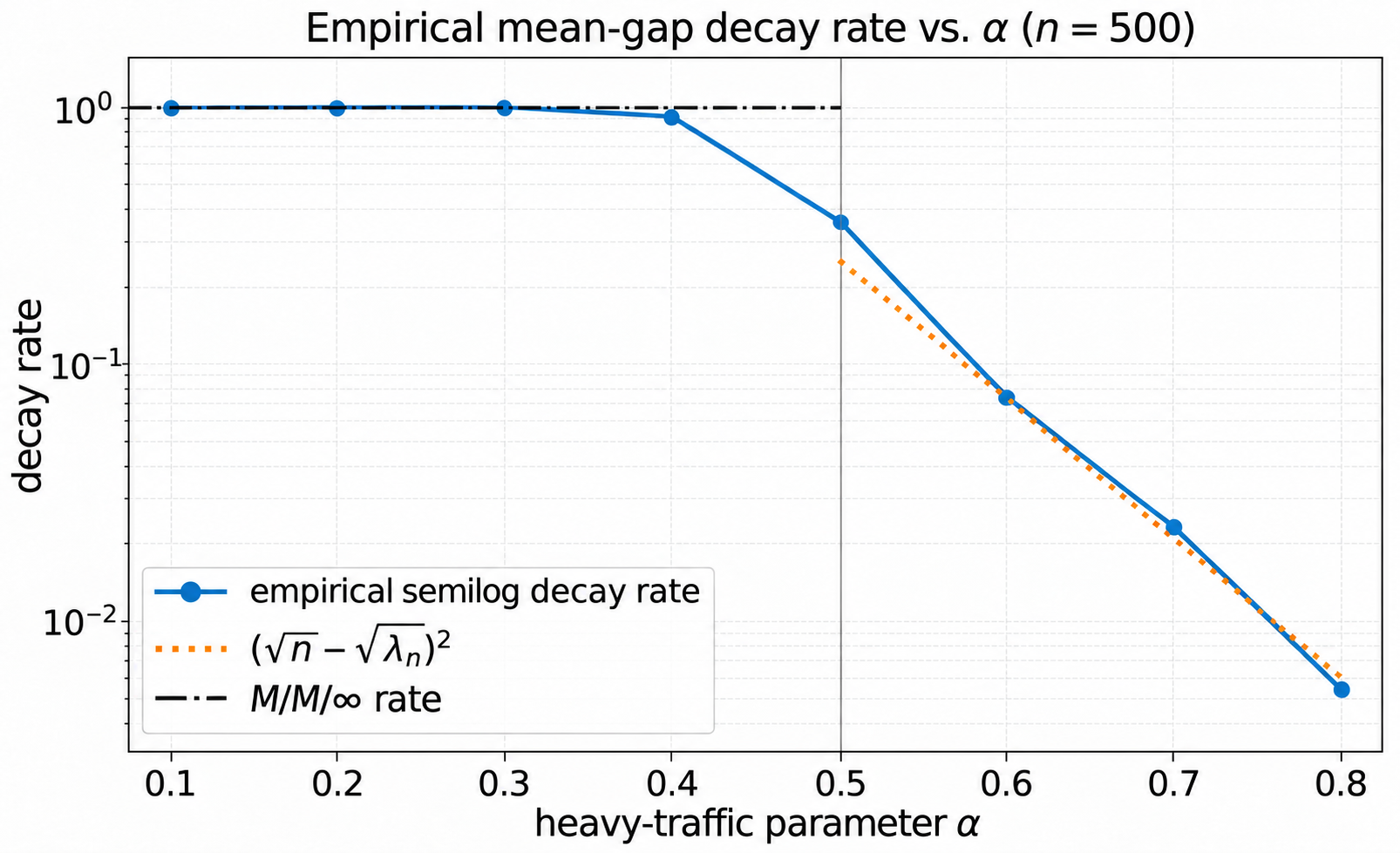}}
    \caption{Empirical decay rate of the transient mean queue length gap as a function of $\alpha$ for $n=500$. The solid line is the fitted empirical decay rate, the dotted horizontal line is the $M/M/\infty$ reference rate $1$, and the dashed curve is the $M/M/1$ reference rate $(\sqrt{n}-\sqrt{\lambda_n})^2$.}
    \label{fig: varying-alpha-n500}
\end{figure}
Figure \ref{fig: varying-alpha-n500} illustrates the phase-transition behavior predicted by Theorem \ref{thm: limiting-mixing-rates}. In the Sub-Halfin-Whitt regime with $\alpha \in (0,1/2)$, the empirical decay rate is close to the $M/M/\infty$ reference rate $1$. At the Halfin-Whitt boundary $\alpha=1/2$, the rate remains of constant order and is also close to the rate $(\sqrt{n}-\sqrt{\lambda_n})^2$. Once $\alpha>1/2$, the empirical rate follows the decreasing $M/M/1$-type reference $(\sqrt{n}-\sqrt{\lambda_n})^2$, showing the slower relaxation associated with persistent congestion above the many-server threshold. Thus, the experiment supports the interpretation that $\alpha$ determines which limiting mechanism governs finite-time recovery: an infinite-server-type mechanism below the Halfin-Whitt boundary and an $M/M/1$-type mechanism above it.

\subsection{Many-server and initial distribution analysis}
\label{ssec: many-server-plots}
Recall that the $M/M/n$ system does not fully exhibit the complete resource pooling behavior of the $M/M/1$ system with service rate $n$. And so, its convergence to stationarity can be slower than the $M/M/1$ counterpart, and this gap becomes especially visible in the $\alpha \in (1/2,1)$ regime. Moreover, Theorem \ref{thm: unifying-theorem} and Theorem \ref{thm: limiting-mixing-rates} suggest that the system mixes with constant rate in the truncated space $\{0,1,\ldots,n\}$ and rate $(\sqrt{n}-\sqrt{\lambda_n})^2 \approx n^{1-2\alpha}/4$ when the queue length is above $n$. Hence, the transient mean queue length decay rate can depend on whether the system starts below or above $n$. Moreover, \cite{prakirt-confluence-large-deviation} shows that the probability that the $M/M/n$ queue behaves like its limiting system in equilibrium goes to $1$. This suggests that many-server experiments provide a useful way to examine how the asymptotic rate manifests at finite $n$.

\begin{figure}[hbt!]
    \centering
    \begin{subfigure}[t]{0.49\textwidth}
        \centering
        \includegraphics[width=\linewidth]{\detokenize{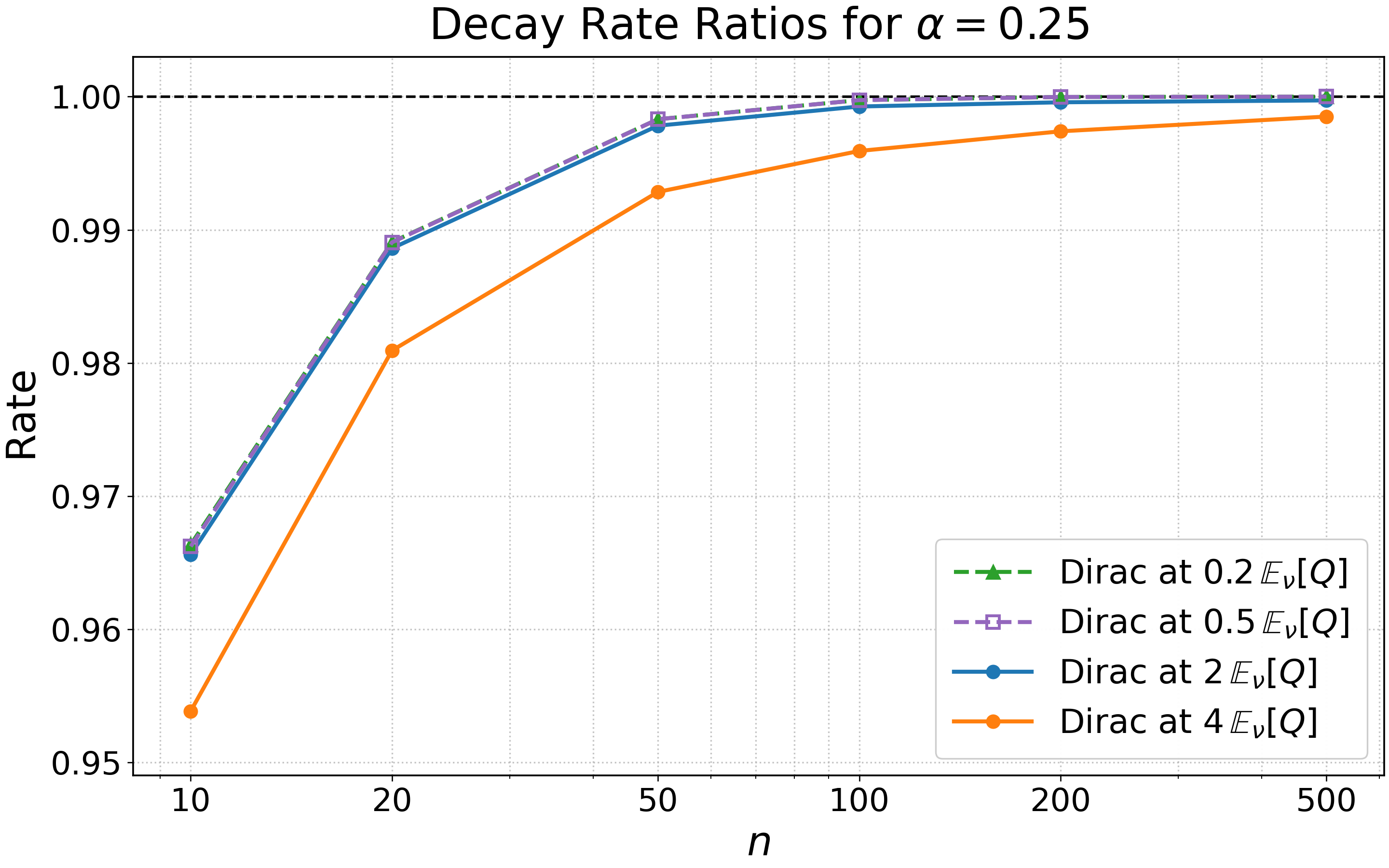}}
        \caption{$\alpha=0.25$.}
        \label{subfig: asymptotic-rate-ratio-alpha025}
    \end{subfigure}\hfill
    \begin{subfigure}[t]{0.49\textwidth}
        \centering
        \includegraphics[width=\linewidth]{\detokenize{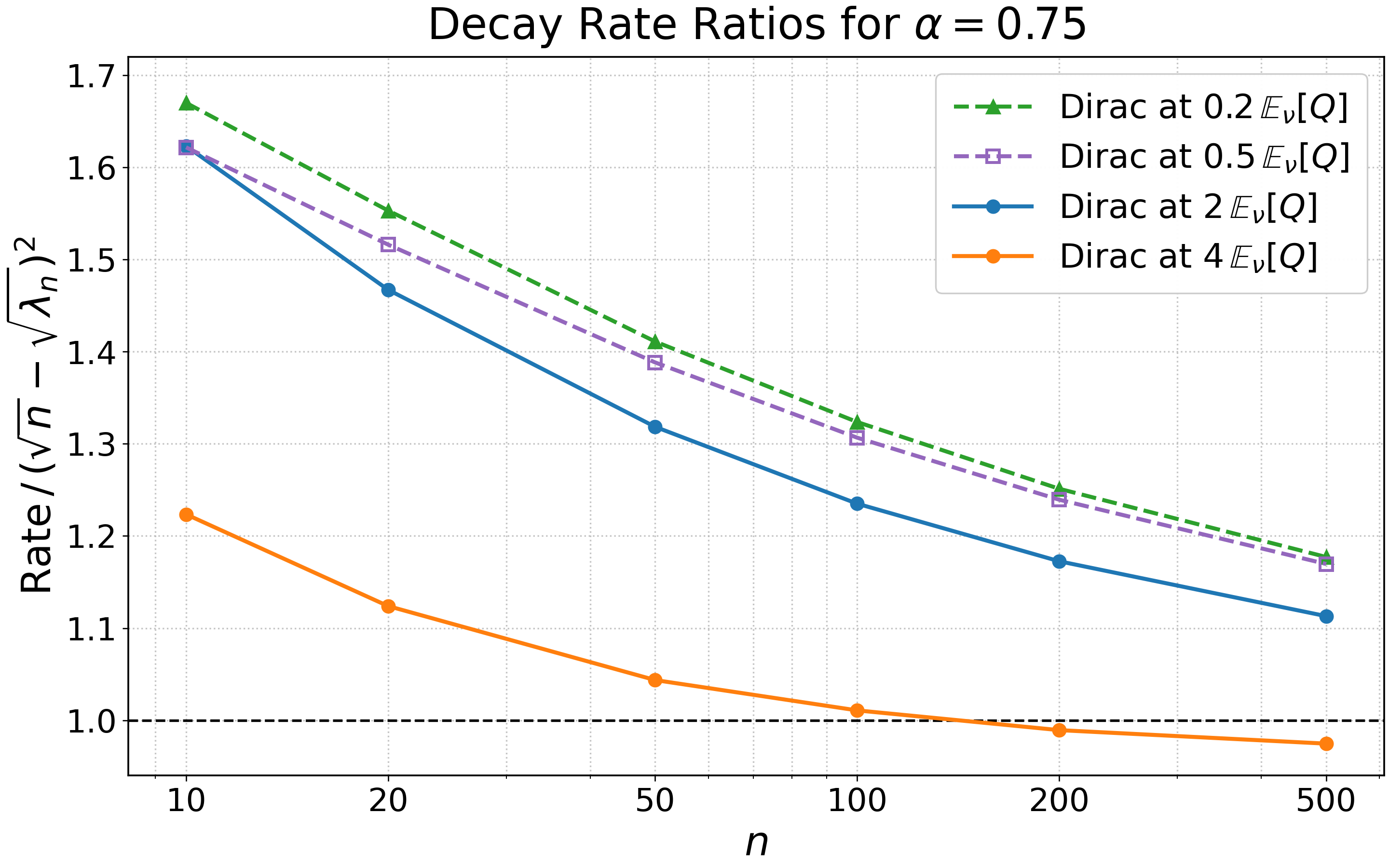}}
        \caption{$\alpha=0.75$.}
        \label{subfig: asymptotic-rate-ratio}
    \end{subfigure}
    \caption{Empirical decay-rate ratios in the Super-Halfin-Whitt regime with $\alpha=3/4$. The empirical rate is obtained from a semilog fit of the transient mean queue length gap, and is normalized by $(\sqrt{n}-\sqrt{\lambda_n})^2$.}
    \label{fig: asymptotic-rate-ratio}
\end{figure}

In this experiment, we fix $\alpha=1/4$ and $\alpha=3/4$ for each experiment, and then we let $n=10,20,50,100,200,500$ to see the behavior of $\rate$ as $n$ becomes larger. In each experiment, the initial distributions are a Dirac delta with mean $2\E_{\nu_n}[q]$ and $4\E_{\nu_n}[q]$, respectively, so the system starts from a large workload shock whose scale grows with the number of servers. The time horizon is also scaled with $n$, namely $t \in [0,5n]$, and we compare the exact transient gap $\Delta_n(t)$ against the rate-only estimator $\widehat{\Delta}_{n,R}(t)$ with $\rate = (\sqrt{n}-\sqrt{\lambda_n})^2$. We estimate the empirical exponential decay rate by fitting the slope of $\log \Delta_n(t)$ over the simulated time grid, and then normalize this fitted rate by the asymptotic rate $(\sqrt{n}-\sqrt{\lambda_n})^2$ as described in Theorem \ref{thm: limiting-mixing-rates}. Figure \ref{fig: asymptotic-rate-ratio} reports this ratio for initial conditions $q(0) \approx 0.2\E_{\nu_n}[q], 0.5\E_{\nu_n}[q], 2\E_{\nu_n}[q], 4\E_{\nu_n}[q]$, which approaches $1$ as $n$ grows.

Additionally, Figure \ref{fig: asymptotic-rate-ratio} also illustrates how the initial queue length distribution influences the decay rate. In the Sub-Halfin–Whitt regime ($\alpha<1/2$), smaller initial queues yield empirical rates closer to the $M/M/\infty$ rate of $1$. In the Super-Halfin–Whitt regime ($\alpha>1/2$), larger initial queues yield rates closer to the $M/M/1$ reference rate $(\sqrt{n}-\sqrt{\lambda_n})^2$. This behavior reflects the initial time needed to clear excess jobs in the former regime or build up congestion in the latter. Nevertheless, for a sufficiently large $n$, such initial dependence will become less important.

\section{Conclusion and Future Work}
\label{sec: conclusion}
In this paper, we establish bounds on the Chi-squared distance between the finite-time distribution and the steady-state distribution of the $M/M/n$ queue. We demonstrate that the mixing rate is tight (possibly up to a constant) in the many-server-heavy-traffic regimes. We also obtain bounds on the mean, moments, and tail of the queue length for a finite time and finite $n$. To prove these results, we build a Lyapunov-Poincar\'e framework and also propose two different ways to obtain local Poincar\'e inequalities, which are of independent interest. In particular, our proposed local canonical path method and the truncation method have the potential to provide stronger bounds than the classical drift and minorization method \citep{meyn-tweedie-exponential-ergodicity-1995, rosenthal1995minorization}.

While we focus on analyzing a birth-and-death chain in this paper, {we believe that our Lyapunov-Poincar\'e methodology could be flexible enough to be applied to other reversible CTMCs}.  Generalizing our results to other reversible Markov chains would pose the main challenge of constructing a suitable Lyapunov function resulting in a suitable negative drift outside a suitable finite set. Nonetheless, the results in Subsection \ref{ssec: lyapunov-poincare-method} are versatile, and one can hope to employ this methodology to obtain mixing results beyond birth-and-death chains. Beyond queuing theory, our finite set approaches can be further applied to other areas such as Learning, Generative AI, and Sampling \citep{Durmus2014QuantitativeBoundsLangevin,raginsky2017nonconvexsgld}.

\bibliographystyle{plainnat}
\bibliography{refs} %

\clearpage
\appendix

\section{Proofs of foundational results}
\label{sec: foundational-results-proof}
In the following Section, we will provide the missing proofs for the key results in the main text.

\subsection{Proof of Proposition \ref{prop: chi-square-implies-wasserstein-1}}
\label{ssec:chi-square-to-w1-proof}
\proof{Proof of Proposition \ref{prop: chi-square-implies-wasserstein-1}}
    If $\chi^2(P,Q)=\infty$, then the desired inequality is trivial. Hence, assume $\chi^2(P,Q)<\infty$. Since $P\ll Q$, define
    \begin{align*}
        h(x) &= \frac{P(x)}{Q(x)}
        \qquad \,\,
        \text{for all } x\in\cS \text{ such that } Q(x)>0, \\
        h(x) &= 0 \qquad\qquad \text{otherwise}.
    \end{align*}
    By Kantorovich--Rubinstein duality on the metric space $(\cS,|\cdot|)$, we have
    \begin{align*}
        W_1(P,Q) = \sup_{\operatorname{Lip}(f)\leq 1} \left|\sum_{x\in\cS} f(x)\br{P(x)-Q(x)}\right| = \sup_{\operatorname{Lip}(f)\leq 1} \left|\sum_{x\in\cS} f(x)\br{h(x)-1}Q(x)\right|,
    \end{align*}
    where the supremum is over all functions $f:\cS\rightarrow\R$ satisfying $|f(x)-f(y)|\leq |x-y|$ for all $x,y\in\cS$. For any such $f$, the assumption $\Var_Q(X)<\infty$ implies $\Var_Q(f(X))<\infty$. Since
    \begin{align*}
        \sum_{x\in\cS}\br{h(x)-1}Q(x)=0,
    \end{align*}
    we may subtract $\E_Q[f]$ and apply Cauchy--Schwarz to obtain
    \begin{align*}
        \left|\sum_{x\in\cS} f(x)\br{h(x)-1}Q(x)\right| &= \left|\sum_{x\in\cS}\br{f(x)-\E_Q[f]}\br{h(x)-1}Q(x)\right| \\
        &\leq \sqrt{\Var_Q(f(X))} \sqrt{\sum_{x\in\cS}\br{h(x)-1}^2Q(x)} \\
        &= \sqrt{\Var_Q(f(X))}\sqrt{\chi^2(P,Q)}.
    \end{align*}
    It remains to bound $\Var_Q(f(X))$. If $X,Y$ are independent random variables with law $Q$, then
    \begin{align*}
        \Var_Q(f(X)) &= \frac{1}{2}\sum_{x,y\in\cS} \br{f(x)-f(y)}^2Q(x)Q(y) \\
        &\leq \frac{1}{2} \sum_{x,y\in\cS} \br{x-y}^2Q(x)Q(y) = \Var_Q(X),
    \end{align*}
    where the inequality uses the fact that $f$ is $1$-Lipschitz. Taking the supremum over all such $f$ gives
    \begin{align*}
        W_1(P,Q) \leq \sqrt{\Var_Q(X)}\sqrt{\chi^2(P,Q)}.
    \end{align*}
\endproof

\subsection{Proof of Proposition \ref{prop: poincare-equals-mixing}}
\label{ssec:mixing-prop-proof}
\proof{\textit{Proof of Proposition \ref{prop: poincare-equals-mixing}}:}
Since we have $\frac{\pi_0}{\nu} \in \ell_{2,\nu}$, we have $\chi^2\br{\pi_0,\nu} = \Var_\nu\br{\frac{\pi_0}{\nu}} < \infty$. Moreover, since $\frac{\pi_t}{\nu} = P_t \frac{\pi_0}{\nu}$ for $t \geq 0$, we have that $\frac{\pi_t}{\nu} \in \ell_{2,\nu}$ as well since $\E_\nu\sqbr{\br{\frac{\pi_t}{\nu}}^2} = \E_\nu\sqbr{\br{P_t \frac{\pi_0}{\nu}}^2} \leq \E_\nu\sqbr{\br{\frac{\pi_0}{\nu}}^2} < \infty$ by Jensen's inequality. From here, we have that
\begin{align}
    \label{eqn: chi-square derivative}
    \frac{d}{dt} \chi^2(\pi_t,\nu)  = \frac{d}{dt} \Var_\nu \br{P_t \frac{\pi_0}{\nu}} = \frac{d}{dt} \norm{P_t \frac{\pi_0}{\nu}-1}_{2,\nu}^2 = 2\left\langle P_t \frac{\pi_0}{\nu} - 1, \frac{d}{dt} P_t \frac{\pi_0}{\nu} \right\rangle_\nu.
\end{align}
This derivative identity is well-known (see Theorem 4.2.5 of \cite{Bakry2013AnalysisAGMarkovDiffusionbook}). Now, since we have $\frac{d}{dt} P_t \frac{\pi_0}{\nu} = \lim_{\delta \rightarrow 0} \br{\frac{P_{t+\delta}-P_t}{\delta}} \frac{\pi_0}{\nu} = \lim_{\delta \rightarrow 0} \br{\frac{P_\delta-1}{\delta}} P_t \frac{\pi_0}{\nu} = \cL P_t \frac{\pi_0}{\nu}$ (due to the fact that $P_{t+s} = P_t P_s$ for any $t,s$), this gives us
\begin{align*}
    \frac{d}{dt} \chi^2(\pi_t,\nu)
    &= 2\left\langle P_t \frac{\pi_0}{\nu}, \cL P_t \frac{\pi_0}{\nu} \right\rangle_\nu = - 2\cE\br{P_t \frac{\pi_0}{\nu}, P_t \frac{\pi_0}{\nu}} \\
    &\leq - \frac{2}{C_P} \Var_\nu \br{P_t \frac{\pi_0}{\nu}} = - \frac{2}{C_P} \Var_\nu \br{\frac{\pi_t}{\nu}} \text{ from the Poincar\'e inequality \eqref{eqn: poincare-inequality}} \\
    &= -\frac{2}{C_P} \chi^2(\pi_t,\nu).
\end{align*}
From Gronwall's inequality, we have
\begin{align*}
    \chi^2\br{\pi_t,\nu} \leq e^{-\frac{2t}{C_P}} \chi^2(\pi_0,\nu).
\end{align*}
Consequently, we have
\begin{align*}
    \chi\br{\pi_t,\nu} \leq e^{-\frac{t}{C_P}} \chi(\pi_0,\nu).
\end{align*}
Hence proved.
$\square$ \endproof

\textbf{Remark on the convergence result}: To the best of our knowledge, we believe that we are the first to establish a Chi-square convergence result for $M/M/n$ systems, which places a stepping stone for the wider analysis of countable state Markov chains. Our approach to handle the Markov semigroup and the Poincar\'e inequality follows similarly from the arguments in \citep{bakry2008lyapunov-poincare, Bakry2013AnalysisAGMarkovDiffusionbook}. In addition, \citep{chewi2020mirror-exponential-ergodicity} previously established Chi-square convergence from the Poincar\'e inequality in the context of Langevin dynamics. We combine these ideas and reestablish the convergence result in the context of CTMCs with a countable state space.

\textbf{Comparison with TV distance convergence and coupling-based methods}: Previously, most mixing results for queuing systems and countable state space Markov chains \citep{philippe-robert-stochastic-networks-and-queues, Luczak_supermarket_2006, luczak-supermarket-equilibrium, yuanzhe2025jsqmixing} are established using the coupling method, which yields convergence in TV distance. We can obtain convergence bounds directly in terms of TV distance for the Erlang-C system from our Chi-square convergence bounds in Theorem \ref{thm: unifying-theorem} by using the inequality
\begin{align}
    d_{TV}(\pi_{n,t}, \nu_n) \leq \frac{\chi(\pi_{n,t}, \nu_n)}{2}
\end{align}
where the readers can refer to (7.37) in \cite{Polyanskiy_Wu_2024}.

\subsection{Proof of Proposition \ref{prop: lyapunov-poincare-singleton} and Theorem \ref{theorem: lyapunov-poincare}}
\label{ssec: poincare-implies-mixing-results-proof}
To help explain the intuitions of the proof, we will first present the proof for Theorem \ref{theorem: lyapunov-poincare} and then prove its singleton variant.

\subsubsection{Proof of Theorem \ref{theorem: lyapunov-poincare}}
\label{sssec: stitching-theorem-proof}

\begin{lemma}
\label{lemma: claim-1}
For any continuous-time reversible Markov chain with the state space $\cS$, the generator $\cL$ and the stationary distribution $\pi$, the following inequality holds for any test function $f \in \ell_{2, \pi}, m \in \R$ and the Lyapunov function $V$ satisfying Assumption \ref{assumption: foster-lyapunov}:
\begin{align}
    \label{eqn: dirichlet-form-bound-from-lyapunov}
    \langle (f-m\mathbf{1})^2/V, -\cL V \rangle_\pi &\leq \langle f, -\cL f \rangle_\pi
\end{align}
\end{lemma}

\proof{\textit{Proof of Lemma \ref{lemma: claim-1}}:}
We will replicate the proof in \citep{taghvaei-lyapunov-poincare} to obtain a result for the continuous-time Markov chain. Denote $g=f-m\mathbf{1}$ where $\mathbf{1}$ is the vector of $1$s (see Subsection \ref{sssec:notations}). Since it is known that $\cL \mathbf{1}=0$, observe that
\begin{align}
    \label{eqn: glg-to-flf-part-1}
    \langle f, \cL f \rangle_{\pi} = \langle f, \cL f \rangle_{\pi}- m\underbrace{\langle \cL\mathbf{1},  f\rangle_{\pi}}_{= 0} \overset{(a)}{=} \langle f, \cL f \rangle_{\pi}- m\langle \mathbf{1}, \cL f \rangle_{\pi} = \langle f-m\mathbf{1}, \cL f \rangle_{\pi} = \langle g, \cL f \rangle_{\pi}
\end{align}
where (a) follows from the self-adjoint property. Similarly, we also have
\begin{align}
    \label{eqn: glg-to-flf-part-2}
     \langle g, \cL f \rangle_{\pi} = \langle g, \cL (f-m \mathbf{1}+m\mathbf{1}) \rangle_{\pi} = \langle g, \cL g + m \underbrace{\cL \mathbf{1}}_{= 0} \rangle_\pi = \langle g, \cL g \rangle_\pi.
\end{align}
And so, from \eqref{eqn: glg-to-flf-part-1} and \eqref{eqn: glg-to-flf-part-2}, we have
\begin{align}
    \label{eqn: glg-equal-flf}
    \langle f, -\cL f\rangle_\pi = \langle g, -\cL g\rangle_\pi.
\end{align}
Which means that showing \eqref{eqn: dirichlet-form-bound-from-lyapunov} is equivalent to showing:
    \begin{align}
        \left\langle \frac{g^2}{V}, -\cL V \right\rangle_\pi
        &\leq \langle g, -\cL g \rangle_\pi.
        \label{eq:claim11}
    \end{align}
    Indeed, note that $\pi(x) \cL(x,y) = \pi(y) \cL(y,x), \cL(x,y) \geq 0$ whenever $x \neq y$ and $-\infty < \sum_{y \neq x} \cL(x,y) V(y) < \infty$ for all $x \in \cS$, so we have
    \begin{align*}
        \left\langle \frac{g^2}{V},-\cL V\right\rangle_\pi &= -\sum_{x\in\cS}\sum_{y\neq x} \pi(x)\cL(x,y) \frac{g(x)^2}{V(x)} \br{V(y)-V(x)} \\
        &= \frac12\sum_{x\in\cS}\sum_{y\neq x} \pi(x)\cL(x,y) \sqbr{g(x)^2+g(y)^2-g(x)^2\frac{V(y)}{V(x)}-g(y)^2\frac{V(x)}{V(y)}} \\
        &\leq \frac{1}{2}\sum_{x\in\cS}\sum_{y\neq x} \pi(x)\cL(x,y) \left[g(x)^2+g(y)^2-2g(x)g(y)\right] \\
        &= \frac12\sum_{x\in\cS}\sum_{y\neq x} \pi(x)\cL(x,y)\left(g(y)-g(x)\right)^2 =\cE(g,g).
    \end{align*}
    The inequality follows because
    \begin{align*}
        g(x)^2\frac{V(y)}{V(x)}
        +g(y)^2\frac{V(x)}{V(y)}
        -2g(x)g(y)
        &=
        V(x)V(y)
        \left(\frac{g(x)}{V(x)}-\frac{g(y)}{V(y)}\right)^2
        \geq 0.
    \end{align*}
    And since we have $g = f-m\mathbf{1}$ and \eqref{eqn: glg-equal-flf}, we have that
    \begin{align*}
        \langle (f-m \mathbf{1})^2/V, -\cL V \rangle_\pi &\leq \langle f, -\cL f \rangle_\pi.
    \end{align*}
    Hence proved.
$\square$ \endproof

Now that the key lemmas are established, we proceed to prove Theorem \ref{theorem: lyapunov-poincare}.

\proof{Proof of Theorem \ref{theorem: lyapunov-poincare}:}
It is well-known that $\norm{f-\E_\nu[f] \mathbf{1}}_{2,\nu}\leq \norm{f - m \mathbf{1}}_{2,\nu}$ for all constants $m\in\R$. Therefore, in order to prove the Poincar\'e inequality,  it suffices to show that:
\begin{equation}
\label{eq:temp-PE}
    \norm{f-m \mathbf{1}}^2_{2,\nu}\leq C_P \langle f, -\cL f \rangle_\nu ,\quad \forall f \in \ell_{2,\nu},
\end{equation}
for the designated $C_P$ and for some constant $m$ to be chosen later.  Consider the general drift condition~\eqref{eqn: foster-lyapunov}
\begin{align*}
    \cL V(q) \leq -\gamma V(q) + b(q) \, \forall q \in \cS.
\end{align*}
Multiply both sides by $\frac{(f(q)-m)^2}{V(q)}$ for some $f \in \ell_{2, \nu}$ to obtain
\begin{align*}
    \frac{(f(q)-m)^2}{V(q)} \cL V(q) &\leq -\gamma (f(q)-m)^2 + \frac{b(q)}{V(q)}(f(q)-m)^2 \\
    &\leq -\gamma (f(q)-m)^2 + b(q)(f(q)-m)^2 \text{ since } V(q) \geq 1,
\end{align*}
where the second inequality follows from $V(q) \geq 1 \, \forall q \in \cS$. Rearranging the terms, we have
\begin{equation*}
    \gamma (f(q)-m)^2  \leq -\frac{(f(q)-m)^2}{V(q)}\cL V(q) + b(q)(f(q)-m)^2,
\end{equation*}
multiplying both sides by $\nu(q)$ and summing with respect to $q \in S$, we have:
\begin{align}
\label{eqn: rearranged-stitching-equation-lyapunov-poincare-proof}
    \gamma \norm{f-m \mathbf{1}}^2_{2,\nu} \leq \langle \br{f-m \mathbf{1}}^2/V, -\cL V \rangle_\nu  + \sum_{q \in \cS} b(q)\nu(q)(f(q)-m)^2.
\end{align}
{Denote $K = \{q: b(q) > 0\}$ and $\vartheta(q) = \frac{b(q) \nu(q)}{\sum_{q \in K} b(q) \nu(q)} = \frac{b(q) \nu_K(q)}{\sum_{q \in K} b(q) \nu_K(q)}$ where $\nu_K(q) = \frac{\nu(q)}{\sum_{x \in K} \nu(x)}$. From Assumption \ref{assumption: weighted-poincare} and choose $m= \frac{\sum_{q \in K} b(q)\nu_K(q) f(q)}{\sum_{q \in K} b(q)\nu_K(q)} = \E_\vartheta[f]$, which is finite by Assumption \ref{assumption: weighted-poincare} in the nontrivial case $\cE(f,f)<\infty$ as it implies $\Var_\vartheta(f) < \infty$. From here, we have:
\begin{align}
    \Var_\vartheta(f) &\leq \frac{C_b}{\nu(K)} \langle f, -\cL f \rangle_\nu \\
    \Leftrightarrow \sum_{q \in K} b(q)\nu(q)(f(q)-m)^2 &\leq \frac{\br{\sum_{q \in K} b(q) \nu(q)}}{\nu(K)} C_b \langle f, -\cL f \rangle_\nu \notag \\
    \label{eqn: local-poincare-with-mass-of-tau}
    &= \br{\sum_{q \in K} b(q) \nu_K(q)} C_b \langle f, -\cL f \rangle_\nu
\end{align}
From Lemma \ref{lemma: claim-1} and \eqref{eqn: rearranged-stitching-equation-lyapunov-poincare-proof}, \eqref{eqn: local-poincare-with-mass-of-tau}, we have:
\begin{align}
    \Var_\nu(f) \leq \norm{f-m \mathbf{1}}^2_{2,\nu}\leq \frac{1 + \br{\sum_{q \in \cS} b(q) \nu_K(q)} C_b}{\gamma} \langle f,-\cL f \rangle_\nu \, \forall f \in \ell_{2,\nu},
\end{align}
and hence, the system admits a Poincar\'e constant $C_P = \frac{1 + \br{\sum_{q \in \cS} b(q) \nu_K(q)} C_b}{\gamma}$.}
$\square$ \endproof

\subsubsection{Proof of Corollary \ref{corollary: lyapunov-poincare-constant-b}}
\label{sssec: proof-of-constant-b-corollary}
\proof{\textit{Proof of Corollary \ref{corollary: lyapunov-poincare-constant-b}}:}
Let $b(q) = B 1_K(q)$ and $m = \frac{\sum_{q \in K} \nu(q) f(q)}{\nu(K)}$, we have
\begin{align}
    \vartheta(q) &= \frac{\nu(q) 1_K(q)}{\sum_{x \in K} \nu(x)} = \frac{\nu(q) 1_K(q)}{\nu(K)} = \nu_K(q), \\
    \sum_{q \in \cS} b(q) \nu(q) &= B \sum_{q \in \cS} 1_K(q) \nu(q) = B \nu(K).
\end{align}
We have Assumption \ref{assumption: weighted-poincare} satisfied for constant $C_L$ means that
\begin{align}
    \Var_{\nu_K}(f) \leq \frac{C_L}{\nu(K)} \langle f, -\cL f \rangle_\nu \, \forall f \in \ell_{2,\nu}
\end{align}
which gives
\begin{align}
    \sum_{q \in \cS} b(q)\nu(q)(f(q)-m)^2 = B \nu(K) \Var_{\nu_K}(f) \leq B C_L \langle f, -\cL f\rangle_\nu.
\end{align}
Substitute this into equation \eqref{eqn: rearranged-stitching-equation-lyapunov-poincare-proof} with our choice of $m$ and apply Lemma \ref{lemma: claim-1}, we have
\begin{align}
    \Var_\nu(f) \leq \frac{1 + B C_L}{\gamma} \langle f, -\cL f \rangle_\nu \, \forall f \in \ell_{2,\nu},
\end{align}
as desired.
$\square$ \endproof

\subsubsection{Proof of Proposition \ref{prop: lyapunov-poincare-singleton}}
\label{sssec: singleton-poincare-constant-proof}

Observe that if we have the finite set $K$ is a singleton, then we can show the local Poincar\'e constant of the singleton set is $0$. Indeed, let $K = \{x^*\}$ and the local measure corresponds to the finite set $K$ be $\nu_K$, then we have $E_{\nu_K}(f) = f(x^*)$ and so $\Var_{\nu_K}(f) = \E_{\nu_K}\sqbr{(f-\E_{\nu_K}(f))^2} = 0$. And so, we can perform a similar analysis to the proof of Theorem \ref{theorem: lyapunov-poincare} but slightly modify it so that we {can} ignore the positive drift term inside the finite set (which happens to be a singleton as well) and bypass the requirement that $V \geq 1$ everywhere.

\proof{\textit{Proof of Proposition \ref{prop: lyapunov-poincare-singleton}}:}
Let $K = \{x^*\}$, we will follow a similar approach to the proof of Theorem \ref{theorem: lyapunov-poincare} but with a slight modification. From {Assumption \ref{assumption: foster-lyapunov-singleton}}, we have
\begin{align}
    \cL V(q) &\leq -\gamma V(q) \quad \forall q \neq x^* \\
    \Rightarrow \frac{(f(q)-f(x^*))^2}{V(q)} \cL V(q) &\leq -\gamma (f(q)-f(x^*))^2 \quad \forall q \neq x^*.
\end{align}
Summing this over $q \in \cS - \{x^*\}$ with weight $\pi(q)$, we have
\begin{align}
    \nonumber
    \sum_{q \neq x^*} \pi(q) \frac{(f(q)-f(x^*))^2}{V(q)} \cL V(q) &\leq -\gamma \norm{f-f(x^*)}_{2,\pi}^2 \\
    \label{eqn: last-step-before-dirichlet-bound-singleton}
    \implies \gamma \norm{f-f(x^*)}_{2,\pi}^2 &\leq \sum_{q \neq x^*} -\frac{\pi(q)(f(q)-f(x^*))^2}{V(q)} \cL V(q).
\end{align}
Now, we want to upper bound the RHS of \eqref{eqn: last-step-before-dirichlet-bound-singleton}. Denote $g = f-f(x^*)$ We have
\begin{align}
    &\sum_{q \neq x^*} -\frac{\pi(q)(f(q)-f(x^*))^2}{V(q)} \cL V(q) \overset{(a)}{=} -\sum_{q \neq x^*} \frac{\pi(q)g(q)^2}{V(q)} \sum_{q' \in \cS} \cL(q,q') V(q')  \nonumber\\
     &= -\sum_{q \in \cS - \{x^*\}} \frac{\pi(q)g(q)^2}{V(q)} \sqbr{\cL(q,x^*) V(x^*) + \sum_{q' \in \cS-\{x^*\}} \cL(q,q') V(q')} \nonumber\\
     &\overset{(b)}{\leq} -\sum_{q, q' \in \cS-\{x^*\}} \pi(q)g(q)^2 \cL(q,q') \frac{V(q')}{V(q)} \nonumber\\
     &= -\frac{1}{2} \sum_{q, q' \in \cS-\{x^*\}} V(q)V(q') Q(q,q')\sqbr{\frac{g(q)^2}{V(q)^2} + \frac{g(q')^2}{V(q')^2}} \nonumber\\
     &\overset{(c)}{\leq} -\sum_{q, q' \in \cS -  \{x^*\}} Q(q,q') g(q)g(q') \nonumber\\
     &= -\sum_{q, q' \in \cS} \pi(q) \cL(q,q') g(q)g(q') \nonumber\\
     &= \langle g, -\cL g\rangle_\pi \nonumber\\
     \label{eqn: dirichlet-form-bound-singleton}
     &\overset{(d)}{=} \langle f, -\cL f \rangle_\pi
\end{align}
where $Q(x,y) = \pi(x) \cL(x,y) = Q(y,x)$ from the reversibility of the Markov chain. $(a)$ follows from the definition of generator, which is
\begin{align}
    \cL V(q) = \sum_{q' \in \cS} \cL(q,q') V(q').
\end{align}
(b) follows from the property of generators that $\cL(x,y) \geq 0 \, \forall x \neq y$ and the fact that $V(q) \geq 0$ for all $q \in \cS$. (c) is equivalent to
\begin{align}
    \frac{1}{2} \sum_{q,q' \in \cS-\{x^*\}} Q(q,q') V(q) V(q') \sqbr{\frac{g(q)}{V(q)}-\frac{g(q')}{V(q')}}^2 \geq 0
\end{align}
and note that $V(q) > 0$ for all $q \neq x^*$ and (d) is from \eqref{eqn: glg-equal-flf}. Apply the upper bound \eqref{eqn: dirichlet-form-bound-singleton} to \eqref{eqn: last-step-before-dirichlet-bound-singleton} and notice that $\Var_\pi(f) \leq \norm{f-f(x^*)}_{2,\pi}^2$, we have
\begin{align}
    \Var_\pi(f) \leq \norm{f-f(x^*)}_{2,\pi}^2 \leq \frac{1}{\gamma}\langle f, -\cL f\rangle_\pi,
\end{align}
which implies the system admits the Poincar\'e constant $\frac{1}{\gamma}$.
$\square$ \endproof
\textbf{Remark}: By avoiding $x^*$ in the proof, we can avoid the division by $0$ issue even in the case $V(x^*) = 0$ and lifts the restriction $V(q) \geq 1 \, \forall q \in \cS$ as in the general case. Thus, the advantage of having $K$ as a singleton would simplify the Lyapunov-Poincar\'e proof by a wide margin. Furthermore, this allows us to lift the stringent $V \geq 1$ assumption that is commonly used in many Markov chain mixing works \citep{taghvaei-lyapunov-poincare, meyn1994computable}, which somewhat restricts our Lyapunov choice. It turns out that allowing $V(x^*) = 0$ would allow us to obtain the tight mixing rate, which we will see in the $M/M/\infty$ setting and the Sub-Halfin-Whitt case. In addition, since the final Poincar\'e constant does not depend on $b$ nor the local Poincar\'e constant $C_L$ with respect to the set $K$, we will not suffer any loss in terms of the mixing rate and the obtained Poincar\'e constant would be tight when we have $K$ being a singleton given that we have the right Lyapunov drift.

\subsection{Proof of Proposition \ref{prop: variational-representation-chi-square}}
\label{ssec: variational-representation-proof}
The following proof follows from Example 7.4 in \citep{Polyanskiy_Wu_2024}.

\proof{\textit{Proof of Proposition \ref{prop: variational-representation-chi-square}}:}
    Let $f(x) = (x-1)^2$, we have the conjugate $f^*_{ext}(y) = y+\frac{y^2}{4}$. From Theorem 7.26 in \citep{Polyanskiy_Wu_2024}, we have
    \begin{align}
        \chi^2(P,Q) = D_f(P||Q) &= \sup_{h: \R \rightarrow \R} \E_P\sqbr{h(X)} - \E_Q\sqbr{f^*_{ext}(h(X))} \\
        &= \sup_{h: \R \rightarrow \R} \E_P\sqbr{h(X)} - \E_Q\sqbr{h(X) + \frac{h(X)^2}{4}} \text{ from } f^*_{ext}(y) = y+\frac{y^2}{4} \\
        &= \sup_{h: \R \rightarrow \R} 2\E_P\sqbr{h(X)} - \E_Q\sqbr{h(X)^2} - 1
    \end{align}
    where in the last step we perform a change of variable $h \leftarrow 2h-2$. By the variational representation of the $L^2(Q)$ norm, we have
    \begin{align*}
        \chi^2(P,Q) &= \sup_{h} \br{2\E_P[h(X)]-\E_Q[h(X)^2]}-1.
    \end{align*}
    Since the supremum also ranges over $g$, the parametrization $h=a+\lambda(g-\E_Q[g])$ covers all admissible $h$. Therefore,
    \begin{align*}
        \chi^2(P,Q) &= \sup_{a,\lambda\in\R} \sup_{g: \Var_Q(g) > 0} \br{2a+2\lambda\br{\E_P[g(X)]-\E_Q[g(X)]}-a^2-\lambda^2\Var_Q(g(X))}-1 \\
        &= \sup_{\lambda\in\R} \sup_{g: \Var_Q(g) > 0} \br{2\lambda\br{\E_P[g(X)]-\E_Q[g(X)]}-\lambda^2\Var_Q(g(X))} \\
        &= \sup_{g: \Var_Q(g) > 0} \frac{\br{\E_P[g(X)]-\E_Q[g(X)]}^2}{\Var_Q(g(X))}.
    \end{align*}
$\square$ \endproof

\subsection{Proof of key properties of the \texorpdfstring{$M/M/n$}{M/M/n} system}
\label{ssec: mmn-properties-proof}
{ Observe that in Proposition \ref{prop: variational-representation-chi-square} and Proposition \ref{prop: chi-square-implies-wasserstein-1}, the upper bound contains a variance of the test function w.r.t. the target distribution $Q$. Thus, for various applications such as mean queue length bounds, one needs to control the mean and variance of $\nu_n$, the steady-state distribution of $M/M/n$. Thus, we have the following lemmas.
\begin{lemma}
    \label{lemma: steady-state-mean-queue-length}
    Let $\nu_n$ be the stationary distribution of the $M/M/n$ queue with arrival rate $\lambda = n-n^{1-\alpha}$ for some $\alpha > 0$ and unit service rate. Let $X\sim \nu_n$, and denote $\rho = \frac{\lambda}{n}, \pwait:=\P_{\nu_n}(X\geq n)$. Then
    \begin{align}
        \label{eqn: mmn-mean-queue-length}
        \E_{\nu_n}[q] = \lambda + \pwait\frac{\rho}{1-\rho} = \lambda + (n^\alpha-1) \pwait.
    \end{align}
\end{lemma}
\proof{Proof of Lemma \ref{lemma: steady-state-mean-queue-length}:} Consider
\begin{align*}
    B:=\min\{X,n\},
    \qquad
    W:=(X-n)^+.
\end{align*}
We have $X=B+W$, which is the decomposition into the number of busy servers and the number of waiting customers. We shall calculate the steady-state mean queue length by computing the expected value of these quantities in equilibrium.

First, to complete $\E_{\nu_n}[B]$, recall that the generator of the $M/M/n$ queue for arrival rate $\lambda$ and unit service rate is
\begin{align}
    \label{eqn: mmn-generator}
    \mathcal L f(q) = \lambda(f(q+1)-f(q)) + \min\{q,n\}(f(q-1)-f(q)).
\end{align}
We have $\E_{\nu_n}[\mathcal L f(X)]=0$ for all functions $f$ such that $\E_{\nu_n}[|f|] < \infty$. Taking $f(q)=q$, we obtain $\E_{\nu_n}[\mathcal L f(X)] = \lambda-\E_{\nu_n}[\min\{X,n\}] = 0$, and therefore
\begin{align}
    \label{eqn: EB-lambda}
    \E[B]=\lambda.
\end{align}
Next, from \eqref{eqn: geometric-excess-mean}, we have
\begin{align}
    \label{eqn: EW}
    \E[W] &= \pwait\E[G\mid X\geq n] = \pwait\frac{\rho}{1-\rho}.
\end{align}
Since $X=B+W$, \eqref{eqn: EB-lambda} and \eqref{eqn: EW} imply
\begin{align*}
    \E[X] = \lambda + \pwait\frac{\rho}{1-\rho} = \lambda + (n^\alpha-1) \pwait.
\end{align*}
$\square$
\endproof
From here, we have the following identity for the steady-state variance of the Erlang-C system.
\begin{lemma}
    \label{lemma: mmn-variance-bound}
    Let $\nu_n$ be the stationary distribution of the $M/M/n$ queue with arrival rate $\lambda = n-n^{1-\alpha}$ and unit service rate. Let $X\sim \nu_n$, and denote $\rho = \frac{\lambda}{n}, \pwait:=\P_{\nu_n}(X\geq n)$. Then
    \begin{align}
        \label{eqn: refined-mmn-variance-identity}
        \Var_{\nu_n}(X)
        &= \lambda(1+\pwait) + \pwait\frac{\rho}{(1-\rho)^2} + \pwait(1-\pwait) \left(\frac{\rho}{1-\rho}\right)^2.
    \end{align}
\end{lemma}
\proof{\textit{Proof of Lemma \ref{lemma: mmn-variance-bound}}:}
To obtain a tight bound on variance, and since we already have the first moment characterization in Lemma \ref{lemma: steady-state-mean-queue-length}, we need a tight characterization of the second moment. To this end, we decompose $X$ into the component below $n$ (number of busy servers) and the component above $n$ (number of waiting customers). First, the stationary distribution of the $M/M/n$ queue is given by
\begin{align}
    \label{eqn: mmn-stationary-distribution-below-n}
    \nu_n(q)
    &=
    \nu_n(0)\frac{\lambda^q}{q!},
    \qquad 0\leq q\leq n-1,
\end{align}
and
\begin{align}
    \label{eqn: mmn-stationary-distribution-above-n}
    \nu_n(n+k)
    &=
    \nu_n(0)\frac{\lambda^n}{n!}\rho^k,
    \qquad k\geq 0.
\end{align}
Therefore,
\begin{align}
    \label{eqn: p-wait-formula}
    \pwait &= \sum_{k=0}^{\infty}\nu_n(n+k) = \nu_n(0)\frac{\lambda^n}{n!}\frac{1}{1-\rho}.
\end{align}
Moreover, conditional on the event $\{X\geq n\}$, the excess random variable $G := X-n$ has the geometric distribution on $\mathbb Z_{\geq 0}$:
\begin{align}
    \label{eqn: geometric-excess-distribution}
    \P(G=k\mid X\geq n)
    &=
    (1-\rho)\rho^k,
    \qquad k\geq 0.
\end{align}
Hence,
\begin{align}
    \label{eqn: geometric-excess-mean}
    \E[G\mid X\geq n]
    &=
    \frac{\rho}{1-\rho}, \\
    \label{eqn: geometric-excess-second-moment}
    \E[G^2\mid X\geq n]
    &=
    \frac{\rho(1+\rho)}{(1-\rho)^2}, \\
    \label{eqn: geometric-excess-variance}
    \Var(G\mid X\geq n)
    &=
    \frac{\rho}{(1-\rho)^2}.
\end{align}
Now define
\begin{align*}
    B:=\min\{X,n\},
    \qquad
    W:=(X-n)^+.
\end{align*}
We would like to choose a quadratic function to control the second moment. Taking $f(q)=q^2$, we get $\E_{\nu_n}\left[\lambda(2X+1)+B(-2X+1)\right] = 2\lambda\E[X]+\lambda-2\E[XB]+\E[B] = 0$. Using \eqref{eqn: EB-lambda}, this gives
\begin{align}
    \label{eqn: EXB-identity}
    \E[XB] = \lambda(\E[X]+1).
\end{align}
Now, on the event $\{X\geq n\}$, we have $X=n+W$. Hence, by \eqref{eqn: geometric-excess-mean} and \eqref{eqn: geometric-excess-second-moment}, we have
\begin{align}
    \label{eqn: EXW}
    \E[XW] &= \pwait \E[(n+W)W\mid X\geq n] = \pwait \left(n\frac{\rho}{1-\rho} + \frac{\rho(1+\rho)}{(1-\rho)^2}\right).
\end{align}
Combining \eqref{eqn: EXB-identity} and \eqref{eqn: EXW}, we obtain
\begin{align}
    \label{eqn: EX2-before-simplify}
    \E[X^2] &= \E[X(B+W)] = \lambda(\E[X]+1) + \pwait \left(n\frac{\rho}{1-\rho}+\frac{\rho(1+\rho)}{(1-\rho)^2}\right).
\end{align}
Substituting \eqref{eqn: mmn-mean-queue-length} into \eqref{eqn: EX2-before-simplify}, we have
\begin{align}
    \E[X^2] &= \lambda + \lambda^2 + \lambda \pwait\frac{\rho}{1-\rho} + \pwait n\frac{\rho}{1-\rho} + \pwait \frac{\rho(1+\rho)}{(1-\rho)^2}.
\end{align}
Since we have from Lemma \ref{lemma: steady-state-mean-queue-length} that
\begin{align}
    \E[X]^2 &= \left(\lambda+\pwait\frac{\rho}{1-\rho}\right)^2 = \lambda^2 + 2\lambda \pwait\frac{\rho}{1-\rho} + \pwait^2
    \left(\frac{\rho}{1-\rho}\right)^2,
\end{align}
we obtain
\begin{align}
    \Var_{\nu_n}(X) &= \E[X^2]-\E[X]^2 \notag \\
    &= \lambda + \pwait(n-\lambda)\frac{\rho}{1-\rho} + \pwait\frac{\rho(1+\rho)}{(1-\rho)^2} - \pwait^2 \left(\frac{\rho}{1-\rho}\right)^2 \notag \\
    &= \lambda(1+\pwait) + \pwait\frac{\rho}{(1-\rho)^2} + \pwait(1-\pwait) \left(\frac{\rho}{1-\rho}\right)^2,
\end{align}
from $\lambda = n\rho = (n-\lambda)\frac{\rho}{1-\rho} = n(1-\rho)\frac{\rho}{1-\rho}$. This proves \eqref{eqn: refined-mmn-variance-identity}, as desired. $\square$ \endproof}

As a corollary, we have the following bound.
\begin{corollary}
    \label{corollary: mmn-variance-bound}
    Let $\nu_n$ be the stationary distribution of the $M/M/n$ queue with arrival rate $\lambda = n-n^{1-\alpha}$ and unit service rate. Let $X\sim \nu_n$. For every $\alpha >0$ and $n\geq 2$ such that $\lambda>0$, we have
    \begin{align}
        \label{eqn: mmn-variance-finite-bound}
        \lambda \leq \Var_{\nu_n}(X) \leq 2n+\frac{5}{4}n^{2\alpha}.
    \end{align}
\end{corollary}

\proof{\textit{Proof of Corollary \ref{corollary: mmn-variance-bound}}:}
It remains to prove the upper bounds from the exact variance identity in Lemma \ref{lemma: mmn-variance-bound}. Since $0\leq \pwait$, we obviously have
\begin{align*}
    \lambda \leq \Var_{\nu_n}(X),
\end{align*}
which proves the LHS. Moreover, since $\pwait \leq 1$, we also have
\begin{align}
    \lambda(1+\pwait) &\leq 2\lambda,
\end{align}
and
\begin{align}
    \pwait\frac{\rho}{(1-\rho)^2} &\leq \frac{\rho}{(1-\rho)^2}.
\end{align}
Moreover, since $\pwait(1-\pwait)\leq 1/4$,
\begin{align}
    \pwait(1-\pwait)
    \left(\frac{\rho}{1-\rho}\right)^2
    &\leq
    \frac14
    \left(\frac{\rho}{1-\rho}\right)^2.
\end{align}
Combining these three bounds with \eqref{eqn: refined-mmn-variance-identity} yields
\begin{align}
    \Var_{\nu_n}(X) &\leq 2\lambda + \frac{\rho}{(1-\rho)^2} + \frac14 \left(\frac{\rho}{1-\rho}\right)^2.
\end{align}
Finally, since $\lambda=n-n^{1-\alpha}$, we have $\rho = \frac{\lambda}{n} = 1-n^{-\alpha}$ and
\begin{align}
    \frac{\rho}{(1-\rho)^2}
    &=
    \frac{1-n^{-\alpha}}{n^{-2\alpha}}
    =
    n^{2\alpha}-n^\alpha \\
    \left(\frac{\rho}{1-\rho}\right)^2
    &=
    \left(\frac{1-n^{-\alpha}}{n^{-\alpha}}\right)^2
    =
    (n^\alpha-1)^2.
\end{align}
Substituting these identities into the preceding upper bound and $n \geq 1$ gives
\begin{align}
    \Var_{\nu_n}(X) &\leq 2(n-n^{1-\alpha}) + \left(n^{2\alpha}-n^\alpha\right) + \frac14\left(n^\alpha-1\right)^2 \leq 2n + \frac{5}{4} n^{2\alpha}.
\end{align}
$\square$
\endproof
In essence, Corollary \ref{corollary: mmn-variance-bound} tells us that $\sqrt{\Var_{\nu_n}\sqbr{X}} = O\br{\sqrt{n}+n^\alpha}$. {For $\alpha \in (0,1/2)$ and for a sufficiently large $n$, the steady state $\nu_n$ behaves similarly to a Poisson distribution with mean $\lambda$, and so the variance of $\nu_n$ should match the order of $\lambda$. We can further show that $\Var_{\nu_n}[X] = \lambda + o(\lambda)$ and observe from \eqref{eqn: mmn-mean-queue-length} that
\begin{align*}
    \E_{\nu_n}[q] = \lambda + O(n^\alpha).
\end{align*}
Hence, if $\alpha \in (0,1/2)$, we have that the queue length fluctuates around $\lambda$ with scale $O(\sqrt{\lambda}) = O(\sqrt{n})$, which means that the queue length is below $n$ with high probability. For fixed $\alpha>1/2$ and the load $\rho$, we have the mean queue length scales with order $O\br{n + \frac{1}{1-\rho}}$, which is in accordance with Kingman's bound as $\rho \to 1$ \citep{kingman1962inequalities}. Since the mean queue length is above $n$ in this regime, there will be waiting jobs, and therefore the system exhibits $M/M/1$ behavior.}

\section{Proof of main mixing results}
\label{sec:unifying-theorem-proof}
In this Section, we present the proofs of the main mixing Theorem (Theorem \ref{thm: unifying-theorem}) {for $\alpha \in (0,1/2)$ and $\alpha \in [1/2,\infty)$, as well as the proofs for $M/M/1$ and $M/M/\infty$ systems for completeness.} We present the complete proof of key technical lemmas in Appendix \ref{ssec:super-HW-proof}, which complements the proof of the Super-Halfin-Whitt regime in Subsection \ref{ssec:lyapunov-canonical-path-approach} of the main text. In Appendix \ref{ssec:sub-HW-proof}, we present the complete proof for the mixing results in the Sub-Halfin-Whitt regime. {We first perform the drift analysis in Appendix \ref{sssec: sub-hw-mean-field-drift-proof}, then the weighted local Poincar\'e bound in Appendix \ref{sssec: sub-hw-weighted-local-poincare}, and finally obtain the mixing time bound in Appendix \ref{sssec: sub-hw-mixing-final-proof}.}

\subsection{Proof of Proposition \ref{thm: mm1-mixing}}
\label{ssec: mm1-proof}
\proof{\textit{Proof of Proposition \ref{thm: mm1-mixing}}:}
    { Let $V(q) = \br{\frac{1}{\sqrt{\lambda}}}^q$, we have $V(q) \geq 1 \, \forall q \in \Z_{\geq 0}$. Using the Lyapunov drift argument in Subsection \ref{ssec: proof-sketch-alpha-geq-1}, we have
    \begin{align}
        \cL V(q) = \br{\frac{1}{\sqrt{\lambda}}-1}\br{\lambda - \sqrt{\lambda}}V(q) = -\br{1-\sqrt{\lambda}}^2 V(q)
    \end{align}
    for all $q \in \Z^+$ and $\cL V(0) = \lambda\br{\frac{1}{\sqrt{\lambda}}-1} \in (0,\infty)$. Hence, from Proposition \ref{prop: lyapunov-poincare-singleton}, we have $\rate \geq (1-\sqrt{\lambda})^2$. On the other hand, we can further show that $\rate \leq (1-\sqrt{\lambda})^2$. Indeed, we can choose $f_a(q) = a^q$ for every $q \in \Z_{\geq 0}$ and $a \in \left(1,\frac{1}{\sqrt{\lambda}}\right)$, which ensures $f_a \in \ell_{2,\nu_1}$. Since $\nu_1(q)=(1-\lambda)\lambda^q$, we have
    \begin{align*}
        \cE(f_a,f_a)
        &=
        \sum_{q \in \Z_{\geq 0}} \lambda \nu_1(q) a^{2q} (a-1)^2
        =
        \frac{\lambda(1-\lambda)(a-1)^2}{1-\lambda a^2}.
    \end{align*}
    Moreover, we have
    \begin{align*}
        \Var_{\nu_1}(f_a) &= \frac{1-\lambda}{1-\lambda a^2} - \br{\frac{1-\lambda}{1-\lambda a}}^2 = \frac{\lambda(1-\lambda)(a-1)^2}{(1-\lambda a^2)(1-\lambda a)^2}.
    \end{align*}
    Hence, by the variational characterization in Proposition \ref{prop: spectral-gap-poincare-variational-form},
    \begin{align*}
        \rate \leq \frac{\cE(f_a,f_a)}{\Var_{\nu_1}(f_a)} =
        (1-\lambda a)^2.
    \end{align*}
    Letting $a \uparrow 1/\sqrt{\lambda}$ gives $\rate \leq (1-\sqrt{\lambda})^2$. Combined with the lower bound $\rate \geq (1-\sqrt{\lambda})^2$ above, this proves $\rate=(1-\sqrt{\lambda})^2$ for $n=1$.}
\endproof

\subsection{Proof of Proposition \ref{thm: mminf-mixing}}
\label{ssec: mminf-proof}

\proof{\textit{Proof of Proposition \ref{thm: mminf-mixing}}:}
    Considering the Lyapunov function $V(q) = |q-\lambda|$, we want to show that $\cL V \leq -V$ whenever $V$ takes some value outside of the singleton set $K = \{\lambda\}$ since $\lambda \in \Z^+$. {Indeed, when $q \not \in K$, we have that $\bigg|q-\lambda\bigg| \geq 1$ which means that $q-\lambda+1, q-\lambda, q-\lambda-1$ all have the same signs. For $q \not \in K = \{\lambda\}$, we have:
    \begin{align*}
        \cL V(q) &= \lambda |q+1-\lambda| + q\mu |q-1-\lambda| - (\lambda+q) |q-\lambda| = - |q-\lambda| = -V(q).
    \end{align*}
    For $q \in K$, we have:
    \begin{align*}
        \cL V\br{\lambda} &= \lambda |\lambda+1-\lambda| + \lambda |\lambda-1-\lambda| - (\lambda + \lambda)  |\lambda - \lambda| = 2\lambda.
    \end{align*}
    Thus, we have shown that $M/M/\infty$ satisfies a Foster-Lyapunov condition with rate $-1$ outside the singleton set $K = \{\lambda\}$ when $\lambda$ is an integer. From Proposition \ref{prop: poincare-equals-mixing} and Proposition \ref{prop: lyapunov-poincare-singleton}, the following holds for $\lambda \in \Z^+$
    \begin{align*}
        \chi(\pi_{\infty, t},\nu_\infty) \leq e^{-t} \chi(\pi_{\infty, 0},\nu_\infty),
    \end{align*}
    which implies $\rate \geq 1$. To show that $\rate \leq 1$, we choose $g(q) = q$ as our test function. Note that the steady-state of the $M/M/\infty$ system is a Poisson distribution with rate $\lambda$, hence, we have $\Var_{\nu_\infty}(g) = \lambda$. Moreover, we have
    \begin{align*}
        \cE(g,g) = \sum_{q \in \Z_{\geq 0}} \nu_\infty(q) \cL(q,q+1) \br{g(q+1)-g(q)}^2 = \lambda \sum_{q \in \Z_{\geq 0}} \nu_\infty(q) = \lambda.
    \end{align*}
    Hence, from Proposition \ref{prop: spectral-gap-poincare-variational-form}, we have
    \begin{align*}
        \rate \leq \frac{\cE(g,g)}{\Var_{\nu_\infty}(g)} = \frac{\lambda}{\lambda} = 1.
    \end{align*}
    This proves that $\rate = 1$, as desired. $\square$}
\endproof

\textbf{Remark}: The $e^{-t}$ convergence rate matches the well-known transient solution of $M/M/\infty$. Previously, there have been multiple mixing proofs of $M/M/\infty$, including an entropic functional inequality proof \citep{Chafa_2006_entropic_inequalities_infty_queue}. Our proof is the first Lyapunov-Poincar\'e proof for $M/M/\infty$. However, it is unfortunate that this proof only works for $\lambda \in \Z$, as otherwise we will not be able to obtain the mixing rate $1$ for $M/M/\infty$ due to the finite set in this case is no longer a singleton.

\subsection{Proof of Theorem \ref{thm: unifying-theorem} in the \texorpdfstring{$\alpha \in [1/2,\infty)$}{alpha in [1/2,infinity)} and Proposition \ref{thm: mm1-mixing}}
\label{ssec:super-HW-proof}
In the following Subsection, we provide the proof of the mixing time results for the $M/M/1$ system and the $M/M/n$ system in the Super-Halfin-Whitt regime, that is $\alpha \in (1/2,\infty)$, and the Halfin-Whitt regime ($\alpha = 1/2$). For the final step of the proof, please refer to Appendix \ref{sssec: mixing-super-HW-proof-final-step}. For the drift lemma (Lemma \ref{lemma: drift-lemma-super-HW}), we refer the reader to Appendix \ref{ssec: proof-detail-drift-lemma-super-hw}. To show Lemma \ref{lemma: roughly-uniform}, that is to show that $\nu_K$ is roughly uniform, we refer the reader to Appendix \ref{sssec: roughly-uniform-lemma-proof}. To prove the local canonical path Lemma \ref{lemma: canonical-path}, we refer the reader to Subsection \ref{sssec: local-canonical-path-proof}. Finally, for the proof of Lemma \ref{lemma: weighted-poincare-super-HW} and Claim \ref{claim: t1}, Claim \ref{claim: t2}, Claim \ref{claim: tau-bound-super-hw}, we refer the readers to Appendix \ref{sssec: proof-of-weighted-poincare-super-HW}.

\subsubsection{Proof of the Lyapunov drift (Lemma \ref{lemma: drift-lemma-super-HW})}
\label{ssec: proof-detail-drift-lemma-super-hw}
\proof{\textit{Proof of Lemma \ref{lemma: drift-lemma-super-HW}}:}
Recall that $V(q) = e^{\theta[q-(n-1)]^+ + \theta [(n-1)-q]^+}$, where we fix $\theta = \log \sqrt{\frac{n}{\lambda}} \leq \frac{1}{2(n^\alpha-1)}$. Now, when $q \geq n$, we have $V(q) = e^{\theta (q-(n-1))}$. We then have
\begin{align}
    \cL V(q) &= \br{\lambda e^{\theta} + \frac{n}{e^{\theta}} - (\lambda + n)}V(q) = \br{e^{\theta}-1}\br{\lambda - \frac{n}{e^{\theta}}} V(q) = -(\sqrt{n}-\sqrt{\lambda})^2 V(q).  \label{eqn: drift-from-above}
\end{align}
The choice of $\theta$ is evident in the third equality as it minimizes $\br{e^{\theta}-1}\br{\lambda - \frac{n}{e^{\theta}}}$. For $q \leq n-2$, we have $V(q) = e^{\theta (n-1-q)}$ {and $K$ is defined by \eqref{eqn: super-hw-finite-set}. This gives
\begin{align}
    \nonumber
    \cL V(q) &= \br{e^\theta-1}\br{q - \frac{\lambda}{e^\theta}} V(q) = \br{e^\theta-1}\bigg(q - (2\lambda - n) + \underbrace{2\lambda - n - \frac{\lambda}{e^\theta}}_{\leq \lambda - \sqrt{\lambda n}}\bigg) V(q) \\
    \label{eqn: drift-bound-below-n}
    &\leq -\br{\sqrt{n}-\sqrt{\lambda}}^2 V(q) + \br{e^\theta-1}\br{q - (2\lambda - n)} V(q).
\end{align}
For $q \leq n-2$ with $q \notin K$, the definition of $K$ gives $q \leq \lfloor 2\lambda \rfloor-n \leq 2\lambda-n$. Thus, \eqref{eqn: drift-bound-below-n} implies
\begin{align*}
     \cL V(q) &\leq -(\sqrt{n}-\sqrt{\lambda})^2 V(q).
\end{align*}
Now, for $q \in K\setminus\{n-1\}$, we have $\lfloor 2\lambda \rfloor-n+1 \leq q \leq n-2$, so $0 < q-(2\lambda-n) \leq 2n^{1-\alpha}$ and $0 \leq (n-1)-q \leq 2n^{1-\alpha}$. This gives us,
\begin{align*}
    \underbrace{\br{e^\theta-1}}_{O(n^{-\alpha})}\underbrace{\br{q - (2\lambda - n)}}_{O(n^{1-\alpha})} \underbrace{V(q)}_{O(1)} \leq \frac{1}{2(n^\alpha-1)} \times 2n^{1-\alpha} \times e^{\frac{n^{1-\alpha}}{n^\alpha-1}} = \underbrace{\frac{n^{1-2\alpha}}{1-n^{-\alpha}} e^{\frac{n^{1-2\alpha}}{1-n^{-\alpha}}}}_{b(q)}.
\end{align*}
And so, we have
\begin{align}
    \cL V(q) \leq -\br{\sqrt{n}-\sqrt{\lambda}}^2 V(q) + \frac{n^{1-2\alpha}}{1-n^{-\alpha}} e^{\frac{n^{1-2\alpha}}{1-n^{-\alpha}}}.
\end{align}
For $q = n-1$, we have
\begin{align*}
    \cL V(n-1) &= (\lambda + n-1)\br{e^\theta - 1} = \br{2n-n^{1-\alpha}-1}\br{\sqrt{\frac{n}{\lambda}}-1} \\
    &= \frac{\br{2n-n^{1-\alpha}-1}\br{\frac{n}{\lambda}-1}}{\sqrt{\frac{n}{\lambda}}+1} \\
    &\leq \frac{\br{2n-n^{1-\alpha}-1}\br{n-(n-n^{1-\alpha})}}{2(n-n^{1-\alpha})} \text{ from } \lambda \leq n \\
    &= n^{1-\alpha} + \frac{n^{1-\alpha}-1}{2\br{n^\alpha-1}} \\
    &= -\br{\sqrt{n}-\sqrt{\lambda}}^2 \underbrace{V(n-1)}_{= 1} + \bigg(n^{1-\alpha} + \frac{n^{1-\alpha}-1}{2\br{n^\alpha-1}} + \underbrace{\br{\sqrt{n}-\sqrt{\lambda}}^2}_{\leq \frac{n^{2-2\alpha}}{4\lambda} = \frac{n^{1-\alpha}}{4(n^\alpha-1)}}\bigg) \\
    &\leq -\br{\sqrt{n}-\sqrt{\lambda}}^2 V(n-1) + n^{1-\alpha}\bigg(1 + \frac{3}{4(n^\alpha-1)}\bigg).
\end{align*}}
Hence we have for $q \leq n$:
\begin{align*}
    \cL V(q) \leq -\gamma_n V(q) + b(q) 1_K.
\end{align*}
with $\gamma_n = \br{\sqrt{n}-\sqrt{\lambda}}^2$ and $b$ as defined in \eqref{eqn: b-term-super-hw-drift}, which completes the proof.
$\square$ \endproof

\subsubsection{Proof of the local canonical path lemma}
\label{sssec: local-canonical-path-proof}
\proof{\textit{Proof of Lemma \ref{lemma: canonical-path}}:}
    We emulate the proof for the canonical path method as in \cite{UBCmixingbook} to handle the continuous-time Markov chain. Denote $\gamma_{x,y}$ as one of the paths connecting $x,y$, note that $\cL(x,y) > 0$ if and only if $|x-y| = 1$, thus all paths in $K$ are sequences of adjacent states. For any test function $f \in \ell_{2, \nu_K}$ and let $Q_K(x,y) = \nu_K(x)\cL(x,y) \forall x \neq y$, we have:
    \begin{align*}
        \Var_{\nu_K}(f) &= \sum_{x < y \in K} \nu_K(x)\nu_K(y) \br{f(x)-f(y)}^2\\
        &\overset{(a)}{\leq} \sum_{x < y \in K} |x-y| \nu_K(x)\nu_K(y) \br{\sum_{x \leq k \leq y-1} (f(k)-f(k+1))^2} \\
        &= \sum_{x < y \in K} |x-y| \nu_K(x)\nu_K(y) \br{\sum_{x \leq k \leq y-1} \frac{(f(k)-f(k+1))^2 Q_K(k,k+1)}{Q_K(k,k+1)}} \\
        &\overset{(b)}{\leq} \sum_{k: k,k+1 \in K} \br{\frac{ \sum_{x,y \in K: x \leq k \leq y-1} |x-y| \nu_K(x) \nu_K(y)}{Q_K(k,k+1)}} \br{f(k)-f(k+1)}^2 Q_K(k,k+1).
    \end{align*}
    Here, (a) follows from the Cauchy-Schwarz inequality and (b) follows from the interchange of sums. Hence proved.
$\square$ \endproof

\subsubsection{Proof of Lemma \ref{lemma: roughly-uniform}}
\label{sssec: roughly-uniform-lemma-proof}

{Recall that the transitions are $\cL(q,q+1) = \lambda, \cL(q,q-1) = q$ when $0 \leq q \leq n$. When $q \in K$, we have $|q-\lambda|$ is small, which means the downward transition rate $q$ is close to the upward transition rate $\lambda$. Hence, we have the CTMC behaves similarly to a symmetric random walk inside $K$ and we will exploit this fact in the following lemma.}

\begin{lemma}[Full statement of Lemma \ref{lemma: roughly-uniform}]
    Let $\lambda = n-n^{1-\alpha}, K = \{\sqbr{\lfloor 2\lambda \rfloor - n + 1}^+,\ldots,n-1\}$ and $\nu_K(x) = \frac{\nu(x)}{\sum_{x \in K} \nu(x)} \forall x \in K$ and $\nu_K(x) = 0$ otherwise. For $\alpha \in [1/2,1)$ and $n \geq n_0 = 110$, we have
    \begin{align}
        \underbrace{\frac{\br{1-\frac{1}{(n^\alpha-2)^2}}^n \sqrt{1-2n^{-\alpha}}}{e^{\frac{1}{12\lceil n-2n^{1-\alpha}\rceil}} \br{2+n^{\alpha-1}}}}_{L_K(n)} \times n^{\alpha-1} \leq \nu_K(x) \leq \underbrace{\frac{e^{\frac{1}{12\lceil n-2n^{1-\alpha}\rceil}}}{\br{2-n^{\alpha-1}}\br{1-\frac{1}{(n^\alpha-2)^2}}^n \sqrt{1-2n^{-\alpha}}}}_{U_K(n)} \times n^{\alpha-1}
    \end{align}
     {where $\lim_{n \rightarrow \infty} L_K(n) = \lim_{n \rightarrow \infty} U_K(n) = 1/2$ for $\alpha \in (1/2,1)$ and $\lim_{n \rightarrow \infty} L_K(n) = 1/(2e)$, $\lim_{n \rightarrow \infty} U_K(n) = e/2$ for $\alpha = 1/2$.}
\end{lemma}

\proof{\textit{Proof of Lemma \ref{lemma: roughly-uniform}}:}
    {Observe that when we have $x \in K$, this implies $x \geq \sqbr{\lfloor 2\lambda \rfloor - n + 1}^+ \geq \lceil n-2n^{1-\alpha}\rceil \geq n-2n^{1-\alpha}$. Hence, for $n \geq n_0 = 110$ and $\alpha \in [1/2,1)$, we have $n-2n^{1-\alpha} \geq n-2\sqrt{n} > 89$ and so $x \in K$ implies $x \geq 90$. In particular, $\lceil n-2n^{1-\alpha}\rceil$ remains a valid lower bound for the Stirling estimates below.} From Stirling's approximation, we have
    \begin{align*}
        \sqrt{2\pi n} \br{\frac{n}{e}}^n \leq n! \leq \sqrt{2\pi n} \br{\frac{n}{e}}^n e^{\frac{1}{12n}}.
    \end{align*}
    Thus, we have
    \begin{align}
        \nonumber
        \frac{e^{-\lambda}\lambda^x}{x!} &\leq e^{-\lambda} \frac{\lambda^x}{\sqrt{2\pi x} \br{\frac{x}{e}}^x} = \frac{e^{x-\lambda}}{\sqrt{2\pi x}}\br{1 + \frac{\lambda-x}{x}}^x \\
        \nonumber
        &\leq \frac{e^{x-\lambda}}{\sqrt{2\pi x}} e^{\lambda-x} = \frac{1}{\sqrt{2\pi x}} \text{ from } 1+t \leq e^t \\
        \label{eqn: poisson-upper-bound}
        &\leq \frac{1}{\sqrt{2\pi(n-2n^{1-\alpha})}} \, \forall x \in K.
    \end{align}
    We also have from Stirling's approximation that
    \begin{align}
        \nonumber
        \frac{e^{-\lambda}\lambda^x}{x!} &\geq \frac{e^{x-\lambda}}{\sqrt{2 \pi x}} \br{1+\frac{\lambda-x}{x}}^x e^{-\frac{1}{12x}}.
    \end{align}
    As mentioned above, we have $\lfloor 2\lambda \rfloor - n + 1 \geq 90$ {for all $n \geq n_0 = 110$}, thus we have $0 \not \in K$ and $90 \leq x \leq n \, \forall x \in K$. We have:
    \begin{align}
        \label{eqn: poisson-lower-bound}
        \frac{e^{-\lambda}\lambda^x}{x!} &\geq \frac{e^{x-\lambda}}{\sqrt{2 \pi x}} \br{1+\frac{\lambda-x}{x}}^x e^{-\frac{1}{12x}} \geq \frac{e^{x-\lambda}}{\sqrt{2 \pi n}} \br{1+\frac{\lambda-x}{x}}^x e^{-\frac{1}{12\lceil n-2n^{1-\alpha}\rceil}}.
    \end{align}
    Let $t = \frac{\lambda}{x}-1$, we have for {$n \geq n_0 = 110$} that $|t| \leq \frac{1}{n^\alpha-2} \leq \frac{5}{4} n^{-\alpha} \, \forall x \in K$ and $1+t \leq e^t \leq 1+t+t^2 e^t$ from Lemma \ref{lemma: taylor-expansion-bound}. This gives
    \begin{align*}
        &\br{\frac{1+t}{e^t}}^x = \br{1+\frac{1+t-e^t}{e^t}}^x \geq \br{1-t^2}^x \geq \br{1-\frac{1}{(n^\alpha-2)^2}}^n.
    \end{align*}
    Putting this back to \eqref{eqn: poisson-lower-bound}, we obtain
    \begin{align}
        \label{eqn: poisson-lower-bound-2}
        \frac{e^{-\lambda}\lambda^x}{x!} \geq \frac{\br{1-\frac{1}{(n^\alpha-2)^2}}^n e^{-\frac{1}{12\lceil n - 2n^{1-\alpha}\rceil}}}{\sqrt{2\pi n}} \, \forall x \in K.
    \end{align}
    Since $|K| = n-1-\lfloor n-2n^{1-\alpha}\rfloor = \lceil 2n^{1-\alpha}\rceil-1$, we have $2n^{1-\alpha}-1 \leq |K| < 2n^{1-\alpha}$. This implies that:
    \begin{align*}
        \sum_{x \in K} e^{-\lambda}\lambda^x/x! &\leq \sum_{x \in K} \frac{1}{\sqrt{2\pi(n-2n^{1-\alpha})}} \leq \frac{\br{2+n^{\alpha-1}}}{\sqrt{2\pi\br{1-2n^{-\alpha}}}} \times n^{1/2-\alpha} \text{ from } \eqref{eqn: poisson-upper-bound}, \\
        \sum_{x \in K} e^{-\lambda}\lambda^x/x! &\geq \sum_{x \in K} \frac{\br{1-\frac{1}{(n^\alpha-2)^2}}^n e^{-\frac{1}{12\lceil n-2n^{1-\alpha}\rceil}}}{\sqrt{2\pi n}} \geq \frac{\br{2-n^{\alpha-1}} \br{1-\frac{1}{(n^\alpha-2)^2}}^n}{\sqrt{2\pi} e^{\frac{1}{12\lceil n-2n^{1-\alpha}\rceil}}} \times n^{1/2-\alpha} \text{ from } \eqref{eqn: poisson-lower-bound-2}
    \end{align*}
    Hence, we have shown that $\sum_{x \in K} e^{-\lambda}\lambda^x/x! = \Theta\br{n^{1/2-\alpha}}$, which gives
    \begin{align}
        \label{eqn: nu-K-lower-bound}
        \nu_K(x) &= \frac{e^{-\lambda}\lambda^x/x!}{\sum_{x \in K} e^{-\lambda}\lambda^x/x!} \geq \underbrace{\frac{\br{1-\frac{1}{(n^\alpha-2)^2}}^n \sqrt{1-2n^{-\alpha}}}{e^{\frac{1}{12\lceil n-2n^{1-\alpha}\rceil}} \br{2+n^{\alpha-1}}}}_{L_K(n)} \times n^{\alpha-1} \text{ and }\\
        \label{eqn: nu-K-upper-bound}
        \nu_K(x) &= \frac{e^{-\lambda}\lambda^x/x!}{\sum_{y \in K} e^{-\lambda}\lambda^y/y!} \leq \underbrace{\frac{e^{\frac{1}{12\lceil n-2n^{1-\alpha}\rceil}}}{\br{2-n^{\alpha-1}}\br{1-\frac{1}{(n^\alpha-2)^2}}^n \sqrt{1-2n^{-\alpha}}}}_{U_K(n)} \times n^{\alpha-1}.
    \end{align}
    Moreover, since $\alpha \in (1/2,1)$, we have that $\lim_{n \rightarrow \infty} \br{1-\frac{1}{(n^\alpha-2)^2}}^n = 1$, which implies
    \begin{align}
        \lim_{n \rightarrow \infty} L_K(n) = \lim_{n \rightarrow \infty} U_K(n) = \frac{1}{2}.
    \end{align}
    And so, we have shown that $\nu_K(x) = \Theta\br{n^{\alpha-1}} = \Theta\br{|K|^{-1}}$, which implies that $\nu_K$ is a roughly uniform distribution inside $K$ and converges to a uniform distribution at the asymptotic. {For $\alpha = \frac{1}{2}$, we have from \eqref{eqn: nu-K-lower-bound} and \eqref{eqn: nu-K-upper-bound} that
    \begin{align}
        \lim_{n \rightarrow \infty} L_K(n) = \frac{1}{2e}, \qquad \lim_{n \rightarrow \infty} U_K(n) &= \frac{e}{2}.
    \end{align}}
$\square$ \endproof

As a corollary, we have the following bound on $U_K(n)$ when $n \geq n_0$.

\begin{corollary}
    \label{corollary: U_K(n)-bound-n0}
    For $n \geq n_0 = 110$ and $\alpha \in [1/2,1)$, we have
    \begin{align}
        U_K(n) &\leq 5.177 \qquad \text{ if } \alpha \in (1/2,1) \\
        U_K(n) &\leq 2.718 \qquad \text{ if } \alpha = 1/2.
    \end{align}
\end{corollary}
\proof{Proof of Corollary \ref{corollary: U_K(n)-bound-n0}:}
    For $n \geq n_0$ and {$\alpha \in (1/2,1)$}, we have
    \begin{align}
        \label{eqn: uk-n0-upper-bound-super-hw}
        U_K(n) \leq \frac{e^{\frac{1}{12\lceil n-2n^{1-\alpha}\rceil}}}{\br{2-n^{\alpha-1}}\br{1-\frac{1}{(n^\alpha-2)^2}}^n \sqrt{1-2n^{-\alpha}}} \leq \frac{e^{\frac{1}{1080}}}{\br{1-\frac{1}{(\sqrt{110}-2)^2}}^{110} \sqrt{1-\frac{2}{\sqrt{110}}}} \leq 5.177.
    \end{align}
    When $\alpha = 1/2$, we can further bound $2-n^{\alpha-1} \geq 2-\frac{1}{\sqrt{110}}$, which improves this bound to
    \begin{align}
        \label{eqn: uk-n0-upper-bound-hw}
        U_K(n) \leq \frac{e^{\frac{1}{1080}}}{\br{2-\frac{1}{\sqrt{110}}}\br{1-\frac{1}{(\sqrt{110}-2)^2}}^{110} \sqrt{1-\frac{2}{\sqrt{110}}}} \leq 2.718.
    \end{align}
    $\square$
\endproof

From the bound on $\nu_K$, we obtain a bound on $Q_K(x,y) = \nu_K(x) \cL(x,y)$ for $\alpha \in [1/2,1)$ as follows.

\begin{corollary}
    \label{corollary: q-bound}
    Let $Q_L(n) = \frac{\br{1-\frac{1}{(n^\alpha-2)^2}}^n \br{1-n^{-\alpha}} \sqrt{1-2n^{-\alpha}}}{e^{\frac{1}{12\lceil n - 2n^{1-\alpha}\rceil}} \br{2+n^{\alpha-1}}}$, we have for $\alpha \in [1/2,1)$ and $n \geq n_0 = 110$ that
    \begin{align}
        Q_L(n) n^\alpha \leq Q_K(x,x+1)
    \end{align}
    for all {$x \in K' = \{\sqbr{\lfloor 2\lambda \rfloor - n + 1}^+,\ldots,n-2\}$} where $\lambda = n-n^{1-\alpha}$.
\end{corollary}

\proof{Proof of Corollary \ref{corollary: q-bound}}
    For {$e_x = (x,x+1)$} where $x \in K' = K-\{n-1\}$, we have that $Q_K(e_x) = Q_K(x,x+1) = \nu_K(x) \cL(x,x+1) = \lambda \nu_K(x)$. Recall from Lemma \ref{lemma: roughly-uniform} that we have
    \begin{align}
        \nu_K(x) \geq L_K(n) \times n^{\alpha-1} = \frac{\br{1-\frac{1}{(n^\alpha-2)^2}}^n \sqrt{1-2n^{-\alpha}}}{e^{\frac{1}{12\lceil n-2n^{1-\alpha}\rceil}} \br{2+n^{\alpha-1}}} \times n^{\alpha-1}.
    \end{align}
    This gives
    \begin{align}
        \label{eqn: qk-lower-bound}
        \underbrace{\frac{\br{1-\frac{1}{(n^\alpha-2)^2}}^n \br{1-n^{-\alpha}} \sqrt{1-2n^{-\alpha}}}{e^{\frac{1}{12\lceil n - 2n^{1-\alpha}\rceil}} \br{2+n^{\alpha-1}}}}_{Q_L(n)} n^\alpha
        \leq Q_K(x,x+1) = \lambda \nu_K(x).
    \end{align}
    for all $x \in K'$. $\square$
\endproof

\subsubsection{Proof of Lemma~\ref{lemma: weighted-poincare-super-HW}}
\label{sssec: proof-of-weighted-poincare-super-HW}
Observe that $f \in \ell_{2, \nu_n}$ trivially implies $f \in \ell_{2, \nu_K}$ and $f \in \ell_{2, \nu_K}$ is equivalent to $f \in \ell_{2, \tau}$ for $\tau(q) = b(q)\nu_K(q) \, \forall q \in \cS$, given that $b$ is sufficiently well-behaved. This essentially tells us that if a test function $f$ is valid for the global distribution, then it also works for the weighted distribution and the local distribution.

\proof{\textit{Proof of Lemma~\ref{lemma: weighted-poincare-super-HW}}:}
    Denote $\tau(q) = b(q)\nu_K(q) \, \forall q \in \Z_{\geq 0}$, where $b$ is defined as in \eqref{eqn: b-term-super-hw-drift}. Note that from Lemma \ref{lemma: drift-lemma-super-HW} and Lemma \ref{lemma: roughly-uniform} and the fact that $n \geq n_0$, we have: {
    \begin{align}
        \label{eqn: tau-upper-bound-1}
        \tau(q) &\leq \br{1 + \frac{3}{4(n^\alpha-1)}} U_K(n) \text{ for } q = n-1, \\
        \label{eqn: tau-upper-bound-2}
        \tau(q) &\leq \frac{e^{\frac{n^{1-2\alpha}}{1-n^{-\alpha}}}}{1-n^{-\alpha}} n^{-\alpha} U_K(n) \text{ for } q \in K-\{n-1\}.
    \end{align}}
    In addition, we also have {
    \begin{align}
        \label{eqn: tau-lower-bound-1}
        \tau(q) &\geq \br{1 + \frac{3}{4(n^\alpha-1)}} L_K(n) \text{ for } q = n-1, \\
        \label{eqn: tau-lower-bound-2}
        \tau(q) &\geq \frac{e^{\frac{n^{1-2\alpha}}{1-n^{-\alpha}}}}{1-n^{-\alpha}} n^{-\alpha} L_K(n) \text{ for } q \in K \setminus \{n-1\}.
    \end{align}
    For $\vartheta(q) = \frac{\tau(q)}{\tau(K)} \, \forall q \in K$, we have
    \begin{align*}
        \Var_\vartheta(f) &= \frac{1}{2}\sum_{x,y \in K} \vartheta (x) \vartheta(y) (f(x)-f(y))^2 = \frac{1}{2}\sum_{x,y \in K} \frac{\tau (x) \tau(y)}{\tau(K)^2} (f(x)-f(y))^2 \\
        &= \frac{1}{\tau(K)^2} \sqbr{\underbrace{\sum_{x \in K} \tau(n-1) \tau(x) \br{f(n-1)-f(x)}^2}_{T_1} + \underbrace{\frac{1}{2}\sum_{x,y \in K-\{n-1\}} \tau(x) \tau(y) (f(x)-f(y))^2}_{T_2}}.
    \end{align*}}
    The terms $T_1, T_2$ can be bounded by the following claims.
    \begin{claim}
    \label{claim: t1}
    Let $g_1: \N \rightarrow \R^+$ be a function to be defined later, we have
    \begin{align*}
        T_1 \leq g_1(n) \times n^{2-4\alpha} \sum_{k \in K'} \br{f(k)-f(k+1)}^2 Q_K(k,k+1).
    \end{align*}
    Here, we have {$g_1(n) \leq 6631 \, \forall n \geq n_0$} and
    \begin{align}
        \lim_{n \rightarrow \infty} g_1(n) &= 2 \text{ for } \alpha > 1/2, \\
        \lim_{n \rightarrow \infty} g_1(n) &= 2e^4 \text{ for } \alpha = 1/2.
    \end{align}
    \end{claim}
    Next, using the local canonical path method, we have the following claim:
    \begin{claim}
    \label{claim: t2}
    Let $g_2: \N \rightarrow \R^+$ be a function to be defined later, we have
    \begin{align*}
        T_2 \leq g_2(n) \times n^{3-6\alpha} \sum_{k \in K'} \br{f(k)-f(k+1)}^2 Q_K(k,k+1).
    \end{align*}
    Here, we have {$g_2(n) \leq 10258 \, \forall n \geq n_0$ and
    \begin{align}
        \lim_{n \rightarrow \infty} g_2(n) &= 1 \text{ for } \alpha > 1/2, \\
        \lim_{n \rightarrow \infty} g_2(n) &= e^5 \text{ for } \alpha = 1/2.
    \end{align}}
    \end{claim}
    The full proofs of these claims can be found in Appendix \ref{sssec: proof-of-claims}. Combining the claims gives us
    {
    \begin{align*}
        \Var_\vartheta(f) = \frac{T_1+T_2}{\tau(K)^2}
        &\leq \frac{\br{g_1(n) n^{2-4\alpha} + g_2(n) n^{3-6\alpha}}}{\tau(K)^2} \sum_{k \in K'} \br{f(k)-f(k+1)}^2 Q_K(k,k+1) \\
        &= \frac{g_1(n) n^{2-4\alpha} + g_2(n) n^{3-6\alpha}}{\nu_n(K) \tau(K)^2} \sum_{k \in K'} \br{f(k)-f(k+1)}^2 \nu_n(k) \cL(k,k+1) \\
        &\leq \frac{g_1(n) n^{2-4\alpha} + g_2(n) n^{3-6\alpha}}{\nu_n(K) \tau(K)^2} \sum_{k = 0}^\infty \br{f(k)-f(k+1)}^2 \nu_n(k) \cL(k,k+1) \\
        &= \frac{g_1(n) n^{2-4\alpha} + g_2(n) n^{3-6\alpha}}{\nu_n(K) \tau(K)^2} \langle f,-\cL f \rangle_{\nu_n}
    \end{align*}
    holds for all $f \in \ell_{2, \nu_n}$. Since from Claim \ref{claim: tau-bound-super-hw}, Claim \ref{claim: t1} and Claim \ref{claim: t2}, we have that $\tau(K) = \Theta(1), g_1(n), g_2(n) = O(1)$, we have $C_b = b_1 n^{2-4\alpha} + b_2 n^{3-6\alpha}$ for some constants $b_1, b_2 > 0$, as desired.}
$\square$ \endproof

\subsubsection{Proof of Claim \ref{claim: tau-bound-super-hw} (total mass of \texorpdfstring{$\tau$}{tau} is bounded by a constant)}
\label{sssec: tau-bound-claim-proof}

\proof{\textit{Proof of Claim~\ref{claim: tau-bound-super-hw}}:}
    We will show that the total mass of $\tau$ is only a constant factor different from the total mass of $\nu_K$. Indeed, from \eqref{eqn: tau-upper-bound-1} and \eqref{eqn: tau-upper-bound-2}, we have for all {$n \geq n_0, \alpha \in [1/2,1)$:
    \begin{align}
        \nonumber
        \sum_{q \in K} b(q) \nu_K(q) &= \sum_{q \in K} \tau(q) = \tau(n-1) + \sum_{q \in K-\{n-1\}} \tau(q) \\
        \nonumber
        &\leq n^{1-\alpha}\br{1 + \frac{3}{4(n^\alpha-1)}} \nu_K(n-1) + \frac{e^{\frac{n^{1-2\alpha}}{1-n^{-\alpha}}}}{1-n^{-\alpha}} \br{1-\nu_K(n-1)} \\
        \label{eqn: sum-of-tau-bound}
        &\leq n^{1-\alpha}\br{1 + \frac{3}{4(n^\alpha-1)}} \nu_K(n-1) + \frac{e^{\frac{n^{1-2\alpha}}{1-n^{-\alpha}}}}{1-n^{-\alpha}}.
    \end{align}
    Observe that $1 = \sum_{q\in K}\nu_K(q) \geq \sum_{q=\lceil n-n^{1-\alpha}\rceil}^{n-1}\nu_K(q) \geq \lfloor n^{1-\alpha}\rfloor\,\nu_K(n-1)$, this gives us
    \begin{align}
        \label{eqn: nu-K-endpoint-upper-bound-super-hw}
        \nu_K(n-1) \leq \frac{1}{\lfloor n^{1-\alpha}\rfloor} \leq 2n^{\alpha-1}.
    \end{align}}
    Thus, combining this with \eqref{eqn: sum-of-tau-bound}, we have for $n \geq n_0$ and $\alpha \in (1/2,1)$ that
    \begin{align}
        \nonumber
        \sum_{q \in K} \tau(q) = \sum_{q \in K} b(q) \nu_K(q) &\leq \br{1 + \frac{3}{4(n^\alpha-1)}} \frac{n^{1-\alpha}}{\lfloor n^{1-\alpha} \rfloor} + \frac{e^{\frac{n^{1-2\alpha}}{1-n^{-\alpha}}}}{1-n^{-\alpha}} \\
        \label{eqn: sum-of-tau-bound-n0-super-hw}
        &\leq 2\br{1+\frac{3}{4(\sqrt{n_0}-1)}} + \frac{e^{\frac{1}{1-\frac{1}{\sqrt{n_0}}}}}{1-\frac{1}{\sqrt{n_0}}} \leq 5.5.
    \end{align}
    So the total mass of $\tau$ is at most a constant factor more than the total mass of $\nu_K$. For $\alpha = 1/2$, one can obtain an even stronger bound
    \begin{align}
        \label{eqn: sum-of-tau-bound-n0-alpha-1/2}
        \sum_{q \in K} \tau(q) \leq \br{1+\frac{3}{4(\sqrt{n_0}-1)}}\frac{\sqrt{n_0}}{\sqrt{n_0}-1} + \frac{e^{\frac{1}{1-\frac{1}{\sqrt{n_0}}}}}{1-\frac{1}{\sqrt{n_0}}} \leq 4.532.
    \end{align}
    On the other hand, we also have from \eqref{eqn: tau-lower-bound-1}, \eqref{eqn: tau-lower-bound-2}, $n \geq n_0 = 110$ and $\alpha \in [1/2,1)$ that
    \begin{align}
        \nonumber
        \sum_{q \in K} \tau(q) &\geq \br{1+\frac{3}{4(n^\alpha-1)}}L_K(n) + \sum_{q\in K\setminus\{n-1\}} \frac{e^{\frac{n^{1-2\alpha}}{1-n^{-\alpha}}}}{1-n^{-\alpha}} n^{-\alpha} L_K(n) \\
        \nonumber
        &=
        \br{1+\frac{3}{4(n^\alpha-1)}+\underbrace{\br{n-2-\lfloor n-2n^{1-\alpha}\rfloor}}_{=|K|-1\,\geq\,2n^{1-\alpha}-2}\frac{e^{\frac{n^{1-2\alpha}}{1-n^{-\alpha}}}}{n^\alpha-1}}L_K(n)
        \\
        \label{eqn: sum-of-tau-lower-bound}
        &\geq
        \br{1+\frac{\frac{3}{4}+2(n^{1-\alpha}-1)e^{\frac{n^{1-2\alpha}}{1-n^{-\alpha}}}}{n^\alpha-1}}L_K(n) \\
        \label{eqn: sum-of-tau-lower-bound-n0}
        &\geq 1 \times \frac{\br{1-\frac{1}{(\sqrt{110}-2)^2}}^{110} \times \sqrt{1-\frac{2}{\sqrt{110}}}}{e^{\frac{1}{1080}} \times 3} \geq 0.0643.
    \end{align}
    In the special case $\alpha = 1/2$, we have
    \begin{align}
        \sum_{q \in K} \tau(q) &\geq \br{1+2e} \frac{\br{1-\frac{1}{(\sqrt{110}-2)^2}}^{110} \times \sqrt{1-\frac{2}{\sqrt{110}}}}{e^{\frac{1}{1080}} \times \br{2+\frac{1}{\sqrt{110}}}} \geq 0.41387.
    \end{align}
    {Now, we consider this sum in the limit. If $\alpha \in (1/2,1)$ then we have
    \begin{align*}
        &\lim_{n \rightarrow \infty} \sum_{q \in K} \tau(q) =  \lim_{n \rightarrow \infty} \sum_{q \in K} b(q) \nu_K(q) \\
        &= \lim_{n \rightarrow \infty} \sqbr{\sum_{q \in K \setminus \{n-1\}} b(q) \frac{1}{2n^{1-\alpha}}} + \lim_{n \rightarrow \infty} \sqbr{b(n-1) \frac{1}{2n^{1-\alpha}}} \\
        &= \lim_{n \rightarrow \infty} \sqbr{\frac{\underbrace{n-2-\lfloor n-2n^{1-\alpha}\rfloor}_{0\,\leq\,|K|-1\,\leq\,2n^{1-\alpha}}}{2n^{1-\alpha}} \frac{n^{1-2\alpha}}{1-n^{-\alpha}} e^{\frac{n^{1-2\alpha}}{1-n^{-\alpha}}}} + \lim_{n \rightarrow \infty} \sqbr{n^{1-\alpha}\br{1+\frac{3}{4(n^\alpha-1)}} \frac{1}{2n^{1-\alpha}}} \\
        &= 0 + \frac{1}{2} = \frac{1}{2}.
    \end{align*}}
    If $\alpha = 1/2$ then from \eqref{eqn: sum-of-tau-bound} and \eqref{eqn: sum-of-tau-lower-bound}, we have
    \begin{align}
        \frac{1}{2e}+1 \leq \liminf_{n \rightarrow \infty} \sum_{q \in K} \tau(q) \leq \limsup_{n \rightarrow \infty} \sum_{q \in K} \tau(q) \leq e + 1.
    \end{align}
$\square$ \endproof

\subsubsection{Proof of Claim \ref{claim: t1} and Claim \ref{claim: t2} (bounds on \texorpdfstring{$T_1, T_2$}{T1, T2})}
\label{sssec: proof-of-claims}

\proof{\textit{Proof of Claim~\ref{claim: t1}}:}
Denote
\begin{align*}
    U_1(n)
    =4\br{1+\frac{3}{4(n^\alpha-1)}}
    \frac{e^{\frac{n^{1-2\alpha}}{1-n^{-\alpha}}}}
         {1-n^{-\alpha}}U_K(n)^2,
\end{align*}
and let $Q_L(n)$ be defined as in Corollary
\ref{corollary: q-bound}. We have
{
    \begin{align*}
        T_1 &= \sum_{x \in K} \tau(n-1) \tau(x) \br{f(n-1)-f(x)}^2 \\
        &\overset{(a)}{\leq} \sum_{x\in K} \br{1+\frac{3}{4(n^\alpha-1)}}U_K(n) \times \frac{e^{\frac{n^{1-2\alpha}}{1-n^{-\alpha}}}}{1-n^{-\alpha}} n^{-\alpha}U_K(n) \br{f(n-1)-f(x)}^2 \\
        &\overset{(b)}{\leq} \underbrace{\br{1+\frac{3}{4(n^\alpha-1)}} \frac{e^{\frac{n^{1-2\alpha}}{1-n^{-\alpha}}}}{1-n^{-\alpha}}U_K(n)^2}_{=U_1(n)/4} \times 2n^{1-2\alpha} \max_{x\in K'}\br{f(n-1)-f(x)}^2 \\
        &\overset{(c)}{\leq} \frac{U_1(n)}{4} \times 2n^{1-2\alpha} \times 2n^{1-\alpha} \sum_{x\in K'}\br{f(x+1)-f(x)}^2 = U_1(n)n^{2-3\alpha} \sum_{x\in K'}\br{f(x+1)-f(x)}^2.
    \end{align*}
}
    Here, (a) follows from \eqref{eqn: tau-upper-bound-1} and \eqref{eqn: tau-upper-bound-2}, (b) follows from the cardinality of $K' = K-\{n-1\}$ being at most $2n^{1-\alpha}$, and (c) is obtained from applying Cauchy-Schwarz along the path $n-1$ to $x$. Since $K' = K \setminus \{n-1\} \Rightarrow K = K' \cup \{n-1\}$, we can bound $T_1$ as
    \begin{align*}
        T_1 &\leq U_1(n) \times n^{2-3\alpha} \sum_{x \in K'} \br{f(x+1)-f(x)}^2 \\
        &= \frac{U_1(n)}{Q_L(n)} n^{2-4\alpha} \times Q_L(n) n^\alpha \sum_{k: k,k+1 \in K} \br{f(k)-f(k+1)}^2 \\
        &\leq \underbrace{\frac{U_1(n)}{Q_L(n)}}_{= g_1(n)} n^{2-4\alpha} \sum_{k: k,k+1 \in K} \br{f(k)-f(k+1)}^2 Q_K(k,k+1) \text{ from Corollary \ref{corollary: q-bound}}.
    \end{align*}
    Now, observe that for $n \geq n_0$ and $\alpha \in (1/2,1)$, we have {
    \begin{align}
        U_1(n) &= 4\br{1+\frac{3}{4(n^\alpha-1)}} \frac{e^{\frac{n^{1-2\alpha}}{1-n^{-\alpha}}}}{1-n^{-\alpha}}U_K(n)^2 \notag\\
        &\leq 4\br{1+\frac{3}{4(\sqrt{110}-1)}} \frac{e^{\frac{1}{1-\frac{1}{\sqrt{110}}}}}{1-\frac{1}{\sqrt{110}}} \times 5.177^2 \leq 386.226.
    \end{align}
    For $\alpha=1/2$, using the stronger bound $U_K(n)\leq2.718$, we obtain
    \begin{align}
        U_1(n) &\leq 4\br{1+\frac{3}{4(\sqrt{110}-1)}} \frac{e^{\frac{1}{1-\frac{1}{\sqrt{110}}}}}{1-\frac{1}{\sqrt{110}}} \times 2.718^2 \leq 106.460.
    \end{align}}
    This gives
    \begin{align}
        Q_L(n) \geq Q_L(n_0) \geq \frac{\br{1-\frac{1}{(\sqrt{110}-2)^2}}^{110} \br{1-\frac{1}{\sqrt{110}}}\sqrt{1-\frac{2}{\sqrt{110}}}}{e^{\frac{1}{1080}} \br{2 + 1}} \geq 0.05825
    \end{align}
    for $\alpha \in (1/2,1)$ and
    \begin{align}
        Q_L(n) \geq \frac{\br{1-\frac{1}{(\sqrt{110}-2)^2}}^{110} \br{1-\frac{1}{\sqrt{110}}}\sqrt{1-\frac{2}{\sqrt{110}}}}{e^{\frac{1}{1080}} \br{2 + \frac{1}{\sqrt{110}}}} \geq 0.0834
    \end{align}
    for $\alpha = 1/2$. Combining these two bounds together, we have {
    \begin{align}
        \label{eqn: g1-n0 bound}
        g_1(n) \leq g_1(n_0) \leq \frac{386.226}{0.05825} \leq 6631
    \end{align}
    for $\alpha \in (1/2,1)$ and
    \begin{align}
        \label{eqn: g1-n0 bound alpha = 1/2}
        g_1(n) \leq g_1(n_0) \leq \frac{106.460}{0.0834} \leq 1277
    \end{align}
    for $\alpha = 1/2$. When taking $n \rightarrow \infty$, for $\alpha \in (1/2,1)$, we have
    \begin{align}
        \lim_{n \rightarrow \infty} U_1(n) = 1, \qquad \lim_{n \rightarrow \infty} Q_L(n) = \frac{1}{2}.
    \end{align}
    Which gives $\lim_{n \rightarrow \infty} g_1(n) = \frac{1}{\frac{1}{2}} = 2$. On the other hand, for $\alpha = 1/2$, we have
    \begin{align}
        \lim_{n \rightarrow \infty} U_1(n) = 4\times1\times e\times\br{\frac e2}^2=e^3, \qquad \lim_{n \rightarrow \infty} Q_L(n) = \frac{1}{2e}.
    \end{align}
    Which gives $\lim_{n \rightarrow \infty} g_1(n) = \frac{e^3}{\frac{1}{2e}} = 2e^4$.} This completes the proof.
$\square$ \endproof

In summary, the proof of Claim \ref{claim: t1} is simply performing Cauchy-Schwarz along the path joining $n-1$ and $x$. Note that as $n \rightarrow \infty$, we have $U_1(n) \rightarrow 1, Q_L(n) \rightarrow 1/2$ for $\alpha > 1/2$ and $U_1(n) \rightarrow e^3, Q_L(n) \rightarrow 1/(2e)$ for $\alpha = 1/2$, which gives $g_1(n) \rightarrow 2e^4$ for $\alpha = 1/2$ and significantly improves from the bound in \eqref{eqn: g1-n0 bound}. For Claim \ref{claim: t2}, the analysis is slightly more involved as we need to use the canonical path method.

\proof{\textit{Proof of Claim~\ref{claim: t2}}:}
    First, if $|K \setminus \{n-1\}| \leq 1$ then $T_2 = 0$ so the bound holds trivially. Otherwise, we shall bound $T_2$ using the local canonical path method. To this end, we bound the number of paths passing through a chosen bottleneck. For {$e_k = (k,k+1)$} where $n-3 \geq k \geq \lfloor 2\lambda \rfloor - n + 1$, note that there are $n+k-\lfloor 2\lambda \rfloor$ ways to choose $x \leq k$ and $n-k-2$ ways to choose $y \geq k+1$ in $K \setminus \{n-1\}$. Hence, there are at most $\br{n+k-\lfloor 2\lambda \rfloor} \times \br{n-k-2} \leq \br{\frac{2n-\lfloor 2\lambda \rfloor-2}{2}}^2 \leq \br{n-\lambda}^2 = n^{2-2\alpha}$ paths containing the edge {$e_k$}, each with length at most $(n-2) - (\lfloor 2\lambda\rfloor - n + 1) \leq 2(n-\lambda) = 2n^{1-\alpha}$. Let $K' = K-\{n-1\}$, observe that {$\tau(q) = b(q) \nu_K(q) \, \forall q \in K'$}, i.e. $\tau(q)$ is proportional to $\nu_K(q)$ whenever $q \in K'$. Thus, we can bound $T_2$ using Lemma \ref{lemma: canonical-path} with respect to the distribution $\tau$ as follows {
    \begin{align*}
        &T_2 = \frac{1}{2}\sum_{x,y \in K'} \tau(x) \tau(y) (f(x)-f(y))^2 \\
        &= \br{\frac{n^{1-2\alpha}}{1-n^{-\alpha}} e^{\frac{n^{1-2\alpha}}{1-n^{-\alpha}}}}^2 \times \frac{1}{2}\sum_{x,y \in K'} \nu_K(x) \nu_K(y) (f(x)-f(y))^2 \text{ from Lemma \ref{lemma: drift-lemma-super-HW}} \\
        &\overset{(a)}{\leq} \br{\frac{n^{1-2\alpha}}{1-n^{-\alpha}} e^{\frac{n^{1-2\alpha}}{1-n^{-\alpha}}}}^2 \max_{k: k,k+1 \in K'} \frac{\sum_{x,y \in K': x \leq k \leq y-1} |x-y| \nu_K(x) \nu_K(y)}{Q_K(k,k+1)} \\
        &\times \sum_{k: k,k+1 \in K'} \br{f(k)-f(k+1)}^2 Q_K(k,k+1)  \\
        &\overset{(b)}{\leq} \br{\frac{n^{1-2\alpha}}{1-n^{-\alpha}} e^{\frac{n^{1-2\alpha}}{1-n^{-\alpha}}}}^2 \times \frac{n^{2-2\alpha} \times 2n^{1-\alpha} \times U_K(n) n^{\alpha-1} \times U_K(n) n^{\alpha-1}}{Q_L(n) n^\alpha} \\
        &\times \sum_{k: k,k+1 \in K'} \br{f(k)-f(k+1)}^2 Q_K(k,k+1) \\
        &= \underbrace{2\br{\frac{1}{1-n^{-\alpha}} e^{\frac{n^{1-2\alpha}}{1-n^{-\alpha}}}}^2 \times \frac{U_K(n)^2}{Q_L(n)}}_{= g_2(n)} \times n^{3-6\alpha} \sum_{k: k,k+1 \in K'} \br{f(k)-f(k+1)}^2 Q_K(k,k+1).
    \end{align*}}
    Here, we have (a) follows from Lemma \ref{lemma: canonical-path} and (b) follows from the number of paths upper bound and the bounds on $\nu_K$ from Lemma \ref{lemma: roughly-uniform} and the lower bound of $Q_K$ from Corollary \ref{corollary: q-bound}. For $\alpha \geq 1/2$, we have $g_2(n)$ is a decreasing function, and so we have {
    \begin{align}
        g_2(n) \leq 2\br{\frac{1}{1-\frac{1}{\sqrt{110}}} e^{\frac{1}{1-\frac{1}{\sqrt{110}}}}}^2 \times \frac{U_K(n)^2}{Q_L(n)} \leq 2\br{\frac{1}{1-\frac{1}{\sqrt{110}}}e^{\frac{1}{1-\frac{1}{\sqrt{110}}}}}^2 \times\frac{5.177^2}{0.05825} \leq 10258.
    \end{align}
    where $Q_L(n) \geq Q_L(n_0) \geq 0.05825$ from the proof of Claim \ref{claim: t1} and
    \begin{align}
        U_K(n)^2 \leq \frac{e^{\frac{1}{540}}}{\br{1-\frac{1}{(\sqrt{110}-2)^2}}^{220}\br{1-\frac{2}{\sqrt{110}}}} \leq 26.799
    \end{align}
    when $\alpha \in (1/2,1)$ and
    \begin{align}
        U_K(n)^2 \leq 7.388
    \end{align}
    when $\alpha = 1/2$. Thus, in the special case $\alpha = 1/2$, we have $Q_L(n) \geq 0.0834$ and so we can further obtain
    \begin{align}
        g_2(n) \leq 2\br{\frac{1}{1-\frac{1}{\sqrt{110}}}e^{\frac{1}{1-\frac{1}{\sqrt{110}}}}}^2 \times\frac{7.388}{0.0834} \leq 1975.
    \end{align}}
    Moreover, if $\alpha > 1/2$ then {from Lemma \ref{lemma: roughly-uniform}, we have
    \begin{align}
        \lim_{n \rightarrow \infty} g_2(n) = 2 \times \frac{\br{\frac{1}{2}}^2}{\frac{1}{2}} = 1.
    \end{align}}
    Otherwise, if $\alpha = 1/2$ then {from Corollary \ref{corollary: q-bound}, we have $\lim_{n \rightarrow \infty} Q_L(n) = \frac{1}{2e}$. And so, we have
    \begin{align}
        \lim_{n \rightarrow \infty} g_2(n) = 2 \times e^2 \times \frac{\br{\frac{e}{2}}^2}{\frac{1}{2e}} = e^5.
    \end{align}}
$\square$ \endproof

\subsubsection{Final step: Proof of Theorem \ref{thm: unifying-theorem} and Theorem \ref{thm: limiting-mixing-rates} for \texorpdfstring{$\alpha \in (1/2,1)$}{alpha in (1/2,1)} and Proposition \ref{thm: mm1-mixing}}
\label{sssec: mixing-super-HW-proof-final-step}

\proof{\textit{Proof of Theorem \ref{thm: unifying-theorem} for $\alpha \in (1/2,1)$ and Proposition \ref{thm: mm1-mixing}:}}
    We will consider two separate cases: the case where $\alpha \geq 1$ or $n = 1$ and the case where $\alpha \in (1/2,1)$.

    \noindent
    \textbf{Case 1 ($\alpha \geq 1$ or $n = 1$)}: From Proposition \ref{prop: lyapunov-poincare-singleton}, Lemma \ref{lemma: drift-lemma-super-HW} and observe that in this case, we have $K = \{n-1\}$, we have that the system admits the Poincar\'e constant $C_P = \frac{1}{(\sqrt{n}-\sqrt{\lambda})^2}$. And so, from Proposition \ref{prop: poincare-equals-mixing}, we have
        \begin{align}
            \rate \geq (\sqrt{n}-\sqrt{\lambda})^2.
        \end{align}
        {
        For $n=1$, we now prove the matching upper bound. Choose $f_a(q)=a^q$ for every $q\in\Z_{\geq 0}$ and $a\in\left(1,\frac{1}{\sqrt{\lambda}}\right)$, which ensures $f_a\in\ell_{2,\nu_1}$. Since the stationary distribution of the $M/M/1$ queue is $\nu_1(q)=(1-\lambda)\lambda^q$, we have
        \begin{align}
            \cE(f_a,f_a)
            &=
            \sum_{q=0}^\infty \lambda\nu_1(q) a^{2q}(a-1)^2
            =
            \frac{\lambda(1-\lambda)(a-1)^2}{1-\lambda a^2}.
        \end{align}
        Moreover,
        \begin{align}
            \Var_{\nu_1}(f_a)
            &=
            \frac{1-\lambda}{1-\lambda a^2}
            -
            \br{\frac{1-\lambda}{1-\lambda a}}^2 \notag \\
            &=
            \frac{\lambda(1-\lambda)(a-1)^2}{(1-\lambda a^2)(1-\lambda a)^2}.
        \end{align}
        Hence, by the variational characterization of the spectral gap in Proposition \ref{prop: spectral-gap-poincare-variational-form},
        \begin{align}
            \rate
            \leq
            \frac{\cE(f_a,f_a)}{\Var_{\nu_1}(f_a)}
            =
            (1-\lambda a)^2.
        \end{align}
        Letting $a\uparrow 1/\sqrt{\lambda}$ gives $\rate \leq (1-\sqrt{\lambda})^2$. Combining this with the lower bound above proves $\rate=(1-\sqrt{\lambda})^2$ for the $M/M/1$ queue.}\\
        \noindent
        \textbf{Case 2 ($\alpha \in (1/2,1)$)}: From Lemma \ref{lemma: drift-lemma-super-HW}, Lemma \ref{lemma: weighted-poincare-super-HW}, \eqref{eqn: sum-of-tau-lower-bound-n0} in Claim \ref{claim: tau-bound-super-hw} and Theorem \ref{theorem: lyapunov-poincare}, {we have
        \begin{align*}
            \br{\sum_{q \in K} b(q) \nu_K(q)} C_b &\leq \frac{6631n^{2-4\alpha}+10258n^{3-6\alpha}}{0.0643} \\
            &\leq 103126n^{2-4\alpha} + 159534n^{3-6\alpha}.
        \end{align*}
        as $n^{1-2\alpha} \leq 1$ for $\alpha \in (1/2,1)$. And so, we obtain the Poincar\'e constant to be:
        \begin{align}
            \label{eqn: alpha-1/2-1-poincare-constant}
            C_P &= \frac{1+ \br{\sum_{q \in K} b(q) \nu_K(q)} C_b}{(\sqrt{n}-\sqrt{\lambda})^2} \leq \frac{1 + 1.04 \times 10^5 n^{2-4\alpha} + 1.6 \times 10^5 n^{3-6\alpha}}{(\sqrt{n}-\sqrt{\lambda})^2}.
        \end{align}
        From Proposition \ref{prop: poincare-equals-mixing}, {this implies that}
        \begin{align}
            \rate \geq \frac{(\sqrt{n}-\sqrt{\lambda})^2}{1 + 1.04 \times 10^5 n^{2-4\alpha} + 1.6 \times 10^5 n^{3-6\alpha}}
        \end{align}
        when $\alpha \in (1/2,1)$.}
$\square$ \endproof
{ \textbf{Remark on $n_0$}: In the proof for the regime $\alpha \in [1/2,\infty)$, we assume $n \geq n_0 = 110$ only to obtain an explicit mixing rate for this regime and to get reasonable constants in the proof. Obtaining an explicit bound for all $n \geq 2$ is possible at the cost of worse universal constants in our mixing rate bounds.}

\subsection{Proof of Theorem \ref{thm: unifying-theorem} in the \texorpdfstring{$\alpha \in (0,1/2)$}{alpha in (0,1/2)} regime}
\label{ssec:sub-HW-proof}

\subsubsection{Drift analysis proof}
\label{sssec: sub-hw-mean-field-drift-proof}

\proof{\textit{Proof of Lemma \ref{lemma: drift-lemma-sub-hw-mean-field}}:}
Since $z \in \br{1,n/\lambda}$, we have
\begin{align*}
    \gamma = (z-1)\br{\frac{n}{z}-\lambda} > 0.
\end{align*}
We verify the drift inequality by considering three cases.

First, suppose that $q>n$. Then $V(q)=z^{q-n}$, $V(q+1)=z V(q)$, and
$V(q-1)=z^{-1}V(q)$. Since the total departure rate is $n$ in this region,
\begin{align*}
    \cL V(q) &= \lambda\br{V(q+1)-V(q)} + n\br{V(q-1)-V(q)} \\
    &= \br{\lambda(z-1)+n(z^{-1}-1)} V(q) \\
    &= -\lambda(z-1)\br{\frac{n}{\lambda z}-1} V(q) \\
    &= -\gamma V(q).
\end{align*}
Since $b(q)=0$ for $q>n$, this gives $\cL V(q) = -\gamma V(q)+b(q)$.

Second, suppose that $0\leq q<n$. The function $V$ is constant and equal to $1$ on $\{0,1,\hdots,n\}$. Therefore, every possible transition from such a state has zero increment in $V$, including the transition from $n-1$ to $n$. Thus
\begin{align}
    \cL V(q) = 0 = -\gamma V(q)+\gamma = -\gamma V(q)+b(q), \qquad 0\leq q<n.
\end{align}

Finally, suppose that $q=n$. Since $V(n)=1$, $V(n-1)=1$, and $V(n+1)=z$, the departure transition contributes zero, and the arrival transition contributes
$\lambda(z-1)$. Hence
\begin{align*}
    \cL V(n) &= \lambda\br{V(n+1)-V(n)} + n\br{V(n-1)-V(n)} \\
    &= \lambda(z-1) \\
    &= -\gamma V(n) + \gamma + \lambda(z-1) \\
    &=
    -\gamma V(n)+b(n).
\end{align*}
Combining the three cases proves the claimed drift inequality.
\endproof

\subsubsection{Weighted local Poincar\'e bound}
\label{sssec: sub-hw-weighted-local-poincare}
We now establish the second ingredient in Theorem \ref{theorem: lyapunov-poincare}, that is to provide a tight weighted local Poincar\'e bound.

First, we observe that the stationary distribution of $M/M/n$ truncated at $n$ is a truncated Poisson distribution. And so, one can show that if the original distribution admits some Poincar\'e constant $C_P$ then the truncated distribution also admits that Poincar\'e constant. We prove this in Lemma \ref{lemma: truncated poincare} as above.

\proof{\textit{Proof of Lemma \ref{lemma: truncated poincare}}:}
    Since $\cL$ is the generator of a birth-and-death process, we have
    \begin{align}
        \label{eqn: birth-and-death-process-condition}
        \cL(x,y) = 0 \, \forall x,y \in \cS \text{ such that } |x-y| > 1.
    \end{align}
    Let $f \in \ell_{2,\nu_K}$ and let $\overline{f}$ be the extension of $f$ such that $\overline{f}(x) = f(m) \forall x \leq m$ and $\overline{f}(x) = f(M) \forall x \geq M$, we have $\overline{f} \in \ell_{2,\nu}$. Denote $Q_K(x,y) = \nu_K(x) \cL(x,y)$, observe that {
    \begin{align*}
        \Var_{\nu_K}(f) &\leq \frac{\Var_\nu(\overline{f})}{\nu(K)} \text{ by direct algebraic manipulation} \\
        &\leq \frac{C_P \langle \overline{f},-\cL \overline{f} \rangle_\nu}{\nu(K)} \text{ from } \eqref{eqn: poincare-inequality-entire-chain} \\
        &= C_P \sum_{x \in \cS} Q_K(x,x+1) \br{\overline{f}(x)-\overline{f}(x+1)}^2
        \\
        &= C_P \sum_{x:x,x+1 \in K} Q_K(x,x+1) \br{f(x)-f(x+1)}^2
    \end{align*}
    where the last two equalities follow $\text{from \eqref{eqn: birth-and-death-process-condition} and definition of } \overline{f}$.} Hence proved.
$\square$ \endproof

Now that we have established that the truncated distribution also admits the same Poincar\'e constant as the original distribution (but this is not necessarily the best Poincar\'e constant), we can easily establish the following Poincar\'e inequality result for the truncated Poisson distribution, which corresponds to the stationary distribution of the $M/M/n$ system restricted to $K = \{0,1,...,n\}$.

\begin{lemma}
    \label{lemma: truncated-mmn}
    (Truncated $M/M/n$) Let $\cL_n, \cL_\infty$ be the generator of the $M/M/n$ and $M/M/\infty$ system, both with unit service rate and arrival rate $\lambda$ and each with stationary distribution $\nu_n, \nu$ respectively. Furthermore, consider $K = \{0,1,...,n\}$ and let $\nu_K(x) = \nu_n(x)/(\sum_{x \in K} \nu_n(x)) \sim \frac{e^{-\lambda} \lambda^x}{x!}$ is the steady state of the $M/M/n$ queue restricted to $K$, we have
    \begin{align*}
        \Var_{\nu_K}(f) \leq \sum_{x:x,x+1 \in K} \nu_K(x) \cL_n(x,x+1) \br{f(x)-f(x+1)}^2.
    \end{align*}
\end{lemma}

\proof{\textit{Proof of Lemma \ref{lemma: truncated-mmn}}:}
    Let $\nu$ be the Poisson distribution with mean $\lambda$, we have that $\nu$ is also the steady state distribution of the $M/M/\infty$ system with arrival rate $\lambda$ and unit service rate. From the Poisson Poincar\'e inequality \citep{Klaassen1985OnAI, kontoyiannis-entropy-law-of-small-number, Chafa_2006_entropic_inequalities_infty_queue},
    we have {
    \begin{align}
        \label{eqn: poisson-poincare}
        \Var_\nu(f) \leq \langle f, -\cL_\infty f \rangle_\nu.
    \end{align}}
    Applying Lemma \ref{lemma: truncated poincare} to the set $K = \{0,1,...,n\}$, we have
    \begin{align}
        \label{eqn: truncated-poisson-poincare}
        \Var_{\nu_K}(f) \leq \sum_{x:x,x+1 \in K} \nu_K(x) \cL_\infty(x,x+1) \br{f(x)-f(x+1)}^2.
    \end{align}
    Now, note that the generator $\cL_n$ of the $M/M/n$ system is the same as the generator $\cL_\infty$ of the $M/M/\infty$ system with arrival rate $\lambda$ and unit service rate when the queue length is no more than $n$. And so, from \eqref{eqn: poisson-poincare} and \eqref{eqn: truncated-poisson-poincare}, we have
    \begin{align}
        \Var_{\nu_K}(f) \leq \sum_{x:x,x+1 \in K} \nu_K(x) \cL_n(x,x+1) \br{f(x)-f(x+1)}^2.
    \end{align}
    which implies that the local Poincar\'e of the $M/M/n$ system on the truncated Poisson distribution $\nu_K$ also admits Poincar\'e constant $1$. Hence proved.
$\square$ \endproof

\textbf{Remark}: It is well-known that the Poisson distribution admits Poincar\'e constant $1$ \citep{Klaassen1985OnAI, kontoyiannis-entropy-law-of-small-number}, so this result also implies the truncated Poisson distribution also admits this Poincar\'e constant. While this is not necessarily the best constant, \citep{prakirt-confluence-large-deviation} shows that the probability of having the queue length exceeding $n$ goes to $0$ as $n$ goes to infinity, so this is a good mixing approximation of the truncated $M/M/n$ system when $n$ is large.

{
Now, we will try to bound the weighted local Poincar\'e component in the Stitching Theorem \ref{theorem: lyapunov-poincare}. Recall that the stitching point is at $n$ where we incur an additional $\lambda (z-1)$ term in the drift analysis (see the proof of Lemma \ref{lemma: drift-lemma-sub-hw-mean-field}), so we need two key ingredients: The first one is a control on $b(n) \nu_K(n)$, and the second one is a competitive bound on the term $\br{f(n)-\E_{\nu_K}[f]}^2$, similarly to Claim \ref{claim: t1}. For the first ingredient, we have the following lemma.

\begin{lemma}
    \label{lemma: small-nu-K-n}
    Fix $\alpha\in(0,1/2)$ and let $\lambda=n-n^{1-\alpha}$. Let
    $K=\{0,1,\hdots,n\}$, and let $\nu_K$ be the Poisson$(\lambda)$
    distribution conditioned on $K$; that is, for every $0\leq k\leq n$,
    \begin{align}
        \nu_K(k)
        =
        \frac{e^{-\lambda}\lambda^k/k!}
        {\sum_{j=0}^{n}e^{-\lambda}\lambda^j/j!}.
    \end{align}
    Assume $n\geq 110$ and $\lambda\geq 3$. Then
    \begin{align}
        \nu_K(n)
        \leq
        \sqrt{\frac{2}{\pi n}}\exp\br{-\frac{1}{2}n^{1-2\alpha}}
        \leq
        \sqrt{\frac{2}{\pi e}}\frac{1}{n^{1-\alpha}}.
    \end{align}
\end{lemma}

\proof{\textit{Proof of Lemma \ref{lemma: small-nu-K-n}}:}
Let $Y_n$ be a Poisson$(\lambda)$ random variable and define
\begin{align}
    a_n=\P(Y_n=n)=e^{-\lambda}\frac{\lambda^n}{n!},
    \qquad
    Z_n=\P(Y_n\leq n).
\end{align}
Then $\nu_K(n)=a_n/Z_n$. Since $\E[Y_n]=\Var(Y_n)=\lambda$ and
$n-\lambda=n^{1-\alpha}$, Cantelli's inequality applied to the centered
random variable $Y_n-\lambda$ gives
\begin{align*}
    \P(Y_n>n)
    =
    \P\br{Y_n-\lambda>n^{1-\alpha}} \leq
    \P\br{Y_n-\lambda\geq n^{1-\alpha}}
    \leq
    \frac{\lambda}{\lambda+n^{2-2\alpha}}
    <
    \frac{1}{2},
\end{align*}
where the last inequality follows from $\lambda<n<n^{2-2\alpha}$. Hence, we have
$Z_n = \P[Y_n \leq n] > 1/2$, and so
\begin{align}
    \nu_K(n) = \frac{\P[Y_n = n]}{\P[Y_n \leq n]} \leq 2a_n.
\end{align}
By Stirling's lower bound $n!\geq \sqrt{2\pi n}(n/e)^n$, we have
\begin{align*}
    a_n
    &\leq
    \frac{1}{\sqrt{2\pi n}}
    \exp\br{
    n-\lambda+n\log\br{\frac{\lambda}{n}}
    } \\
    &=
    \frac{1}{\sqrt{2\pi n}}
    \exp\br{
    n^{1-\alpha}
    +
    n\log\br{1-n^{-\alpha}}
    } \\
    &\leq
    \frac{1}{\sqrt{2\pi n}}
    \exp\br{-\frac{1}{2}n^{1-2\alpha}},
\end{align*}
where we used $\log(1-x)\leq -x-x^2/2$ for $x\in(0,1)$ from Lemma \ref{lemma: log-lower-bound}. Therefore,
\begin{align}
    \nu_K(n)
    \leq
    2a_n
    \leq
    \frac{2}{\sqrt{2\pi n}}\exp\br{-\frac{1}{2}n^{1-2\alpha}}.
\end{align}
Since $x e^{-x^2/2}\leq e^{-1/2}$ for $x\geq 1$, substituting
$x=n^{1/2-\alpha} \geq 1$ also gives
\begin{align}
    \nu_K(n)
    \leq
    \frac{2}{\sqrt{2\pi e}}\frac{1}{n^{1-\alpha}},
\end{align}
which proves the lemma.
\endproof

Now that the first ingredient is established, we shall establish a bound on $\br{f(n)-\E_{\nu_K}[f]}^2$. We obtain this via the following result.

\begin{lemma}
\label{lemma: sandbox-endpoint-hardy}
Let $K=\{0,1,\hdots,n\}$, let $\nu_K$ be a probability distribution on $K$,
and define
\begin{align}
    F_K(k) = \sum_{i=0}^{k}\nu_K(i).
\end{align}
Let $Q(k,k+1)>0$ for $0\leq k\leq n-1$. Then, for every function $f$,
\begin{align}
    \nu_K(n)\br{f(n)-\E_{\nu_K}[f]}^2
    \leq
    H_n
    \sum_{k=0}^{n-1}Q(k,k+1)\br{f(k+1)-f(k)}^2,
\end{align}
where
\begin{align}
    H_n =
    \nu_K(n)
    \sum_{k=0}^{n-1}
    \frac{F_K(k)^2}{Q(k,k+1)}.
\end{align}
\end{lemma}

\proof{\textit{Proof of Lemma \ref{lemma: sandbox-endpoint-hardy}}:}
Note that we can decompose
\begin{align*}
    f(n)-\E_{\nu_K}[f]
    &=
    \sum_{i=0}^{n}\nu_K(i)\br{f(n)-f(i)} \\
    &=
    \sum_{i=0}^{n}\nu_K(i)\sum_{k=i}^{n-1}\br{f(k+1)-f(k)} \\
    &=
    \sum_{k=0}^{n-1}F_K(k)\br{f(k+1)-f(k)}.
\end{align*}
Hence, by Cauchy's inequality, we have
\begin{align}
    \br{f(n)-\E_{\nu_K}[f]}^2
    &\leq
    \br{
    \sum_{k=0}^{n-1}\frac{F_K(k)^2}{Q(k,k+1)}
    }
    \br{
    \sum_{k=0}^{n-1}Q(k,k+1)\br{f(k+1)-f(k)}^2
    }.
\end{align}
Multiplying both sides by $\nu_K(n)$ proves the claim.
$\square$ \endproof

From here, it is sufficient to bound the $\sum_{k=0}^{n-1}\frac{F_K(k)^2}{Q(k,k+1)}$ term. We have the following lemma.

\begin{lemma}
\label{lemma: sandbox-poisson-endpoint-congestion}
Fix $\alpha\in(0,1/2)$ and let $\lambda=n-n^{1-\alpha}$. Assume
$n\geq 110$ and $\lambda\geq 3$. Then, with $\rho = \frac{\lambda}{n}$,
\begin{align}
    H_n
    =
    \nu_K(n) \sum_{k=0}^{n-1}\frac{F_K(k)^2}{Q(k,k+1)}
    \leq
    \frac{C_H(\rho)}{n^{1-\alpha}}.
\end{align}
\end{lemma}

\proof{\textit{Proof of Lemma \ref{lemma: sandbox-poisson-endpoint-congestion}}:} To prove this result, we split the sum into two parts. We first bound
\begin{align}
    S_1 := \nu_K(n) \sum_{k=\lfloor \lambda \rfloor + 1}^{n-1}\frac{F_K(k)^2}{Q(k,k+1)}
\end{align}
and then handle
\begin{align}
    S_2 := \nu_K(n) \sum_{k=0}^{\lfloor \lambda \rfloor}\frac{F_K(k)^2}{Q(k,k+1)}.
\end{align}
For $S_1$, observe that $Q(k,k+1) = \lambda \nu_K(k)$ and $F_K(k) = \sum_{i=0}^k \nu_K(i) \leq 1$, we have
\begin{align}
    S_1 \leq \nu_K(n) \sum_{k=\lfloor \lambda \rfloor + 1}^{n-1} \frac{F_K(k)^2}{Q(k,k+1)} \leq  \sum_{k=\lfloor \lambda \rfloor + 1}^{n-1} \frac{\nu_K(n)}{\lambda \nu_K(k)}.
\end{align}
Thus, it is sufficient to bound $\sum_{k=\lfloor \lambda \rfloor + 1}^{n-1} \frac{\nu_K(n)}{\nu_K(k)}$. To that end, observe for $\lfloor \lambda \rfloor +1 \leq k \leq n-1$ that
\begin{align}
    \frac{\nu_K(n)}{\nu_K(k)} = \lambda^{n-k} \frac{k!}{n!} = \prod_{i=1}^{n-k} \frac{\lambda}{n-i+1} = \prod_{i=1}^{n-k} \sqbr{1-\frac{n^{1-\alpha} - (i-1)}{n - (i-1)}} \leq e^{-\sum_{i=1}^{n-k} \frac{n^{1-\alpha} - (i-1)}{n - (i-1)}}.
\end{align}
Since $1 \leq n-k < n^{1-\alpha}$, we have
\begin{align*}
    \sum_{i=1}^{n-k} \frac{n^{1-\alpha} - (i-1)}{n - (i-1)} &\geq \sum_{i=1}^{n-k} \frac{n^{1-\alpha} - (i-1)}{n} \\
    &= \frac{n-k}{n^\alpha} - \frac{(n-k-1)(n-k)}{2n} \\
    &= \frac{n-k}{n} \br{n^{1-\alpha} - \frac{n-k-1}{2}} \\
    &\geq \frac{n-k}{n} \times \frac{n^{1-\alpha}}{2} = \frac{n-k}{2n^\alpha}.
\end{align*}
From here, we have
\begin{align}
    \sum_{k=\lfloor \lambda \rfloor + 1}^{n-1} \frac{\nu_K(n)}{\nu_K(k)} \leq \sum_{k=\lfloor \lambda \rfloor + 1}^{n-1} e^{-\frac{n-k}{2n^\alpha}} \leq \sum_{j=1}^{\lceil n^{1-\alpha}\rceil-1}
    e^{-\frac{j}{2n^\alpha}} \leq \sum_{j=1}^{\infty}
    e^{-\frac{j}{2n^\alpha}} = \frac{1}{e^{\frac{1}{2n^\alpha}}-1} \leq 2n^\alpha
\end{align}
where the last inequality follows from Lemma \ref{lemma: taylor-expansion-bound}. Hence, we have
\begin{align}
    \label{eqn: hn-s1-bound}
    S_1 = \nu_K(n) \sum_{k=\lfloor \lambda \rfloor + 1}^{n-1}\frac{F_K(k)^2}{\lambda \nu_K(k)} \leq \sum_{k=\lfloor \lambda \rfloor + 1}^{n-1}\frac{\nu_K(n)}{\lambda \nu_K(k)} \leq \frac{2n^\alpha}{\lambda}.
\end{align}
Now, we want to bound $S_2$. Unlike the bound on $S_1$, we need to take into account the $F_K(k)$ term to obtain a competitive bound. To that end, we have the following claim.

\begin{claim}
\label{claim: upper bound F_K}
    Under the assumptions of Lemma \ref{lemma: sandbox-poisson-endpoint-congestion},
    for every $0\leq k\leq \lfloor \lambda\rfloor$, we have
    \begin{align}
        F_K(k)
        \leq
        5 \sqrt{\lambda}\nu_K(k).
    \end{align}
\end{claim}

\proof{\textit{Proof of Claim \ref{claim: upper bound F_K}}:}
Motivated by the proof of the Super-Halfin-Whitt regime, we will consider two cases: $0 \leq k \leq \lambda - \sqrt{\lambda}$ and $\lambda - \sqrt{\lambda} < k \leq \lambda$.

\textbf{Case 1}: If $0\leq k\leq \lambda-\sqrt{\lambda}$, then for every $0\leq i\leq k$,
\begin{align}
    \frac{\nu_K(i)}{\nu_K(k)} = \prod_{r=i+1}^{k}\frac{r}{\lambda}
    \leq
    \br{\frac{k}{\lambda}}^{k-i}.
\end{align}
Therefore,
\begin{align}
    F_K(k) = \sum_{i=0}^{k} \nu_K(i)
    \leq  \nu_K(k) \sum_{m=0}^{\infty}\br{\frac{k}{\lambda}}^m = \frac{\nu_K(k)}{1 - \frac{k}{\lambda}}
    \leq
    \sqrt{\lambda} \nu_K(k).
\end{align}

\textbf{Case 2}: Consider $\lambda-\sqrt{\lambda}<k\leq \lambda$. Set $s=1-k/\lambda$, we have $0\leq s<1/\sqrt{\lambda}$. And since $\lambda\geq 3$, we have $k \geq 2$. Using Stirling's approximation for $k \leq \lambda$ and $k \geq 2$, we have
\begin{align}
    e^{-\lambda} \frac{\lambda^k}{k!} &\geq \frac{\exp\br{-\lambda+k+k\log\br{\frac{\lambda}{k}}}}{\sqrt{2\pi\lambda}e^{1/24}} = \frac{\exp\sqbr{-\lambda\br{s+(1-s)\log\br{1-s}}}}{\sqrt{2\pi\lambda}e^{1/24}}.
\end{align}
Since
\begin{align}
    \log(1-s) = -\sum_{r=1}^{\infty}\frac{s^r}{r},
    \qquad |s|<1,
\end{align}
we obtain
\begin{align*}
    -\lambda s-\lambda(1-s)\log(1-s)
    &= \lambda\br{-s+(1-s)\sum_{r=1}^{\infty}\frac{s^r}{r}} \\
    &=
    \lambda\br{
    \sum_{r=2}^{\infty}
    \br{\frac{1}{r}-\frac{1}{r-1}}s^r
    } \\
    &=
    -\lambda
    \sum_{r=2}^{\infty}\frac{s^r}{r(r-1)}.
\end{align*}
Consequently, we have
\begin{align}
    e^{-\lambda} \frac{\lambda^k}{k!} \geq \frac{1}{\sqrt{2\pi\lambda}e^{1/24}} \exp\br{-\lambda \sum_{r=2}^{\infty}\frac{s^r}{r(r-1)}}.
\end{align}
Since $s<1/\sqrt{\lambda}$ and $\lambda\geq 3$, this gives
\begin{align}
    \lambda
    \sum_{r=2}^{\infty}\frac{s^r}{r(r-1)}
    \leq
    \sum_{r=2}^{\infty}\frac{3^{1-r/2}}{r(r-1)}
    <0.643.
\end{align}
And thus,
\begin{align}
    e^{-\lambda} \frac{\lambda^k}{k!}
    \geq
    \frac{e^{-0.643}}{\sqrt{2\pi}e^{1/24}}\frac{1}{\sqrt{\lambda}}
    \geq
    \frac{1}{5\sqrt{\lambda}}.
\end{align}
This implies $F_K(k)\leq 1\leq 5\sqrt{\lambda} \nu_K(k)$, which proves the desired bound in both cases.
Finally,
\begin{align}
    \frac{F_K(k)^2}{\nu_K(k)}
    \leq
    25\lambda\nu_K(k),
    \qquad
    0\leq k\leq \lfloor \lambda\rfloor.
\end{align}
Summing over $k$ gives
\begin{align}
    \sum_{k=0}^{\lfloor \lambda\rfloor}
    \frac{F_K(k)^2}{\nu_K(k)}
    \leq
    25\lambda
    \sum_{k=0}^{\lfloor \lambda\rfloor}\nu_K(k)
    \leq
    25\lambda.
\end{align}
$\square$ \endproof
Back to the proof of Lemma \ref{lemma: sandbox-poisson-endpoint-congestion}, we have from Claim \ref{claim: upper bound F_K} that
\begin{align}
    \frac{F_K(k)^2}{\nu_K(k)}
    \leq
    25\lambda\nu_K(k),
    \qquad
    0\leq k\leq \lfloor \lambda\rfloor.
\end{align}
Summing over $k$ gives
\begin{align}
    \sum_{k=0}^{\lfloor \lambda\rfloor} \frac{F_K(k)^2}{\nu_K(k)}
    \leq
    25\lambda
    \sum_{k=0}^{\lfloor \lambda\rfloor}\nu_K(k)
    \leq
    25\lambda.
\end{align}
Hence, $S_2$ can be bounded as follows:
\begin{align}
    \label{eqn: hn-s2-bound}
    S_2 \leq \frac{\nu_K(n)}{\lambda} \sum_{k=0}^{\lfloor \lambda\rfloor} \frac{F_K(k)^2}{\nu_K(k)} \leq \frac{\nu_K(n)}{\lambda} \times 25\lambda = 25 \nu_K(n).
\end{align}
From Lemma \ref{lemma: small-nu-K-n}, \eqref{eqn: hn-s1-bound}, \eqref{eqn: hn-s2-bound}, we have
\begin{align}
    \label{eqn: hn-final-bound}
    H_n = S_1+S_2 \leq \frac{2n^\alpha}{\lambda} + 25\sqrt{\frac{2}{\pi n}} \exp \br{-\frac{n^{1-2\alpha}}{2}} \leq \frac{2 \frac{n}{\lambda}}{n^{1-\alpha}} +
    25\sqrt{\frac{2}{\pi e}}\frac{1}{n^{1-\alpha}} = \frac{C_H(\rho)}{n^{1-\alpha}}
\end{align}
where $C_H(\rho) = \frac{2}{\rho} + 25\sqrt{\frac{2}{\pi e}}$, as desired.
$\square$ \endproof

As we can see in \eqref{eqn: hn-final-bound}, the term $C_H(\rho)$ can scale inversely with $\rho$ and can blow up to infinity if $\rho \rightarrow 0$. But this is in fact not a problem for us, since we can choose an appropriate $z$ parameter in the drift lemma (Lemma \ref{lemma: drift-lemma-sub-hw-mean-field}) so that the weighted local Poincar\'e  constant is nicely controlled. Indeed, we obtain the weighted local Poincar\'e bound as follows.

\begin{lemma}[Weighted local Poincar\'e inequality for the Sub-Halfin-Whitt regime]
\label{lemma: sandbox-sub-hw-weighted-local-poincare}
Fix $\alpha \in \br{0, \frac{1}{2}}$ and assume $n \geq n_0 = 110, \lambda=n-n^{1-\alpha} \geq 3$. Set $\rho=\lambda/n$ and choose $z=\frac{2}{1+\rho}$. Define $\gamma,b$ as in Lemma \ref{lemma: drift-lemma-sub-hw-mean-field} and let
\begin{align}
    \vartheta_n(q) = \frac{b(q)\nu_n(q)} {\sum_{x\in K}b(x)\nu_n(x)}, \qquad q\in K.
\end{align}
Then Assumption \ref{assumption: weighted-poincare} holds with
\begin{align}
    C_b = 1 + 28.2 n^{2\alpha-1}.
\end{align}
\end{lemma}

\proof{\textit{Proof of Lemma \ref{lemma: sandbox-sub-hw-weighted-local-poincare}}:}
Since $b(q)=\gamma$ for $q<n$ and $b(n)=\gamma+\lambda(z-1)$, we have
\begin{align}
    \vartheta_n
    =
    \frac{\nu_{K}+\frac{\lambda(z-1)}{\gamma} \nu_K(n)\delta_n}{1+\frac{\lambda(z-1)}{\gamma} \nu_K(n)}.
\end{align}
Thus, we obtain the following mixture variance identity
\begin{align}
    \Var_{\vartheta_n}(f) = \frac{1}{1+\frac{\lambda(z-1)}{\gamma} \nu_K(n)}\Var_{\nu_{K}}(f) + \frac{\frac{\lambda(z-1)}{\gamma} \nu_K(n)}{(1+\frac{\lambda(z-1)}{\gamma} \nu_K(n))^2} \br{f(n)-\E_{\nu_{K}}[f]}^2.
\end{align}
Using Lemma \ref{lemma: sandbox-endpoint-hardy} and Lemma \ref{lemma: sandbox-poisson-endpoint-congestion}, we get
\begin{align}
    \Var_{\vartheta_n}(f) &\leq \br{1+\frac{\lambda(z-1)}{\gamma} H_n} \sum_{q=0}^{n-1} \lambda\nu_{K}(q)\br{f(q+1)-f(q)}^2.
\end{align}
Now, since
\begin{align}
    \frac{\lambda(z-1)}{\gamma}
    = \frac{1}{\frac{1}{\rho z}  - 1}.
\end{align}
From Lemma \ref{lemma: sandbox-poisson-endpoint-congestion}, we have
\begin{align}
    \frac{\lambda(z-1)}{\gamma} H_n
    \leq \frac{C_H(\rho) n^{\alpha-1}}{\frac{1}{\rho z} - 1} = \frac{\frac{2}{\rho} + 25\sqrt{\frac{2}{\pi e}}}{\frac{1}{\rho z}-1} \times n^{\alpha-1} = \frac{2 + 25\rho \sqrt{\frac{2}{\pi e}}}{\frac{1}{z}-\rho} \times n^{\alpha-1}.
\end{align}
Here, note that $\rho \in (0,1)$, and so $2 + 25\rho \sqrt{\frac{2}{\pi e}} \in (2,14.1)$. Since $z = \frac{2}{1+\rho}$ and $\rho \in (0,1)$, this choice gives us $z \in \br{1,\frac{1}{\rho}}$, which is an admissible choice of $z$ for Lemma \ref{lemma: drift-lemma-sub-hw-mean-field}. We have
\begin{align}
    \frac{\lambda(z-1)}{\gamma} H_n
    \leq \frac{\br{2 + 25\rho \sqrt{\frac{2}{\pi e}}}}{\frac{1}{z}-\rho} \times n^{\alpha-1} \leq \frac{14.1 n^{\alpha-1}}{\frac{1-\rho}{2}} = 28.2 n^{2\alpha-1}.
\end{align}

Finally,
\begin{align}
    \sum_{q=0}^{n-1}
    \lambda\nu_{K}(q)\br{f(q+1)-f(q)}^2
    &=
    \frac{1}{\nu_n(K)}
    \sum_{q=0}^{n-1}
    \lambda\nu_n(q)\br{f(q+1)-f(q)}^2 \\
    &\leq
    \frac{1}{\nu_n(K)}
    \langle f,-\cL f\rangle_{\nu_n}.
\end{align}
Thus, we have
\begin{align}
    \Var_{\vartheta_n}(f)
    \leq
    \frac{1 + 28.2 n^{2\alpha-1}}{\nu_n(K)}
    \langle f,-\cL f\rangle_{\nu_n},
\end{align}
which proves Assumption \ref{assumption: weighted-poincare} with
$C_b = 1 + 28.2 n^{2\alpha-1}$.
$\square$ \endproof

With the establishment of the weighted local Poincar\'e constant, our proof is almost done. Before going to the final step, we need to bound the term $\sum_{q \in K} b(q) \nu_K(q)$ as follows.

\begin{lemma}
\label{lemma: sub-hw-b-average}
Fix $\alpha\in(0,1/2)$ and let $\lambda=n-n^{1-\alpha}$. Suppose
$n \geq 110$ and $\lambda \geq 3$. Let
$K=\{0,1,\hdots,n\}$ and let $\nu_K$ be the Poisson$(\lambda)$ distribution conditioned on $K$. Choose
$z = \frac{2}{1+\frac{\lambda}{n}}$, and define $\gamma,b$ as in Lemma
\ref{lemma: drift-lemma-sub-hw-mean-field}, we have
\begin{align}
    \gamma \leq \sum_{q\in K}b(q)\nu_K(q) =
    \gamma+\lambda(z-1)\nu_K(n) \leq
    \gamma + \frac{n \exp\br{-\frac{1}{2}n^{1-2\alpha}}}{\sqrt{8\pi e}}.
\end{align}
\end{lemma}

\proof{\textit{Proof of Lemma \ref{lemma: sub-hw-b-average}}:}
From Lemma \ref{lemma: drift-lemma-sub-hw-mean-field}, we have
\begin{align*}
    b(q) &= \gamma, \qquad &0\leq q < n, \\
    b(n) &= \gamma+\lambda(z-1), \qquad &q = n.
\end{align*}
Since $\nu_K$ is supported on $K=\{0,1,\hdots,n\}$, this gives
\begin{align}
    \sum_{q\in K}b(q)\nu_K(q) &= \gamma+\lambda(z-1)\nu_K(n).
\end{align}
The lower bound follows immediately because $\lambda(z-1)\nu_K(n)\geq 0$.
For the upper bound, from $z = \frac{2}{1+\rho}$ where $\rho=\lambda/n$, we have
\begin{align}
    \lambda(z-1) = \frac{n \rho (1-\rho)}{1+\rho} \leq \frac{n}{4}.
\end{align}
Therefore,
\begin{align}
    \lambda(z-1)\nu_K(n)
    \leq
    \frac{n}{4} \sqrt{\frac{2}{\pi e}}\exp\br{-\frac{1}{2}n^{1-2\alpha}},
\end{align}
and so
\begin{align}
    \sum_{q\in K}b(q)\nu_K(q)
    \leq
    \gamma + \frac{n \exp\br{-\frac{1}{2}n^{1-2\alpha}}}{\sqrt{8\pi e}}.
\end{align}
$\square$ \endproof

\subsubsection{Final step}
\label{sssec: sub-hw-mixing-final-proof}
Now, we shall prove Theorem \ref{thm: unifying-theorem} for the Sub-Halfin-Whitt regime as follows.

\proof{\textit{Proof of Theorem \ref{thm: unifying-theorem} for the Sub-Halfin-Whitt regime}:}
Fix $\alpha\in(0,1/2)$ and let $\lambda=n-n^{1-\alpha}$,
$\rho=\lambda/n$, and $z=\frac{2}{1+\rho}$. Since $\rho\in(0,1)$, we have $z \in \br{1, \frac{1}{\rho}}$, which is an admissible choice for Lemma \ref{lemma: drift-lemma-sub-hw-mean-field}. From Lemma \ref{lemma: drift-lemma-sub-hw-mean-field}, the $M/M/n$ system satisfies Assumption \ref{assumption: foster-lyapunov} with the finite set $K=\{0,1,\hdots,n\}$, the function $b$ defined in
\eqref{eqn: light-traffic-b-function}, and
\begin{align}
    \gamma &= \frac{n(1-\rho)^2}{2(1+\rho)} = \frac{n^{1-2\alpha}}{2(1+\rho)} > 0.
\end{align}
Now, by Lemma \ref{lemma: sandbox-sub-hw-weighted-local-poincare}, we have Assumption
\ref{assumption: weighted-poincare} holds with
\begin{align}
    C_b=1+28.2 n^{2\alpha-1}.
\end{align}
Moreover, Lemma \ref{lemma: sub-hw-b-average} gives
\begin{align}
    \gamma \leq \sum_{q\in K}b(q)\nu_K(q) \leq \gamma + \frac{n \exp\br{-\frac{1}{2}n^{1-2\alpha}}}{\sqrt{8\pi e}}.
\end{align}
Applying the Stitching Theorem \ref{theorem: lyapunov-poincare}, the system
admits the Poincar\'e inequality
\begin{align}
    \Var_{\nu_n}(f)
    \leq
    C_P(n)
    \langle f,-\cL f\rangle_{\nu_n},
    \qquad
    \forall f\in \ell_{2,\nu_n},
\end{align}
where
\begin{align}
    C_P(n) &= \frac{1+C_b\sum_{q\in K}b(q)\nu_K(q)}{\gamma} \nonumber \\
    &\leq \frac{1+\br{1+28.2 n^{2\alpha-1}}\br{\gamma+\frac{n \exp\br{-\frac{1}{2}n^{1-2\alpha}}}{\sqrt{8\pi e}}}}{\gamma} \nonumber \\
    \label{eqn: sub-hw-poincare}
    &= 1 + \frac{1}{\gamma} + 28.2 n^{2\alpha-1} + \br{1+28.2 n^{2\alpha-1}}
    \frac{n \exp\br{-\frac{1}{2}n^{1-2\alpha}}}{\sqrt{8\pi e}\gamma}.
\end{align}
Now, we will show that $C_P(n)$ can be upper bounded by some constant. Since $\alpha \in (0,1/2)$ and $\rho \in (0,1)$, we have
\begin{align}
    \frac{1}{\gamma} = 2(1+\rho)n^{2\alpha-1} \leq 4 n^{2\alpha-1} \leq 4.
\end{align}
This gives
\begin{align}
    1 + \frac{1}{\gamma} + 28.2 n^{2\alpha-1} + \br{1+28.2 n^{2\alpha-1}}
    \frac{n \exp\br{-\frac{1}{2}n^{1-2\alpha}}}{\sqrt{8\pi e}\gamma} \leq 33.2 + \frac{58.4 n^{2\alpha} \exp\br{-\frac{1}{2}n^{1-2\alpha}}}{\sqrt{2\pi e}}.
\end{align}
Now, note that the function $f(x) = x^{\frac{2\alpha}{1-2\alpha}} e^{-\frac{x}{2}}$ has a unique maximizer at $x^\star = \frac{4\alpha}{1-2\alpha}$ and the maximum value of $f$ is $f(x^\star) = \br{\frac{4\alpha}{e(1-2\alpha)}}^{\frac{2\alpha}{1-2\alpha}}$. Hence, for $\alpha \in (0,1/2)$, we have
\begin{align}
    C_P(n) \leq 33.2 + \frac{58.4}{\sqrt{2\pi e}} \br{\frac{4\alpha}{e(1-2\alpha)}}^{\frac{2\alpha}{1-2\alpha}}.
\end{align}
By Proposition \ref{prop: poincare-equals-mixing}, we have
\begin{align}
    \rate \geq \frac{1}{33.2 + \frac{58.4}{\sqrt{2\pi e}} \br{\frac{4\alpha}{e(1-2\alpha)}}^{\frac{2\alpha}{1-2\alpha}}}
\end{align}
as desired. $\square$ \endproof}

\section{Proof of Corollary \ref{corollary: mixing-time-bound}}
\label{sec: proof-of-tv-corollary}
\proof{\textit{Proof of Corollary \ref{corollary: mixing-time-bound}}:}
Fix $0<\delta\leq \chi(\pi_{n,0},\nu_n)$, we will show that $\tau_{mix}(\delta) \leq \frac{\log\br{\frac{\chi(\pi_{n,0}, \nu_n)}{\delta}}}{\rate}$. Indeed, we have
\begin{align}
    \label{eqn: t-delta-mixing-time-bound}
    t \geq \frac{\log\br{\frac{\chi(\pi_{n,0}, \nu_n)}{\delta}}}{\rate} \Leftrightarrow -\rate \, \times \, t + \log \chi(\pi_{n,0}, \nu_n) \leq \log \delta \Leftrightarrow e^{-\rate \, \times \, t} \chi(\pi_{n,0}, \nu_n) \leq \delta.
\end{align}
From Theorem \ref{thm: unifying-theorem} and \eqref{eqn: t-delta-mixing-time-bound}, we have
\begin{align}
    \chi(\pi_{n,t},\nu_n) \leq e^{-\rate t} \chi(\pi_{n,0}, \nu_n) \leq \delta
\end{align}
which implies $\tau_{mix}(\delta) \leq \frac{\log\br{\frac{\chi(\pi_{n,0}, \nu_n)}{\delta}}}{\rate}$.  $\square$

\endproof

\section{Proof of Theorem \ref{thm: limiting-mixing-rates}}
\label{sec: limiting-mixing-rates-proof}
{

To show the asymptotic tightness of the mixing rate, our proof consists of two steps. First, we shall establish $\liminf_{n \rightarrow \infty} \rate$ via the non-asymptotic Poincar\'e constant bound in Theorem \ref{thm: unifying-theorem}. Then, by exploiting the variational representation
\begin{align*}
    \rate = \inf_{f:f \in \ell_{2,\nu_n},\Var_{\nu_n}(f) > 0} \frac{\cE(f,f)}{\Var_{\nu_n}(f)}.
\end{align*}
of the spectral gap \citep{Mu-Fa-Chen1991-l2-convergence}, we shall choose an appropriate test function $f$ to obtain a matching $\limsup$ bound, thereby showing the limiting rate $\lim_{n \rightarrow \infty} \rate$ for $\alpha \in (0,1/2)$ and $\alpha \in (1/2,\infty)$. For $\alpha = 1/2$, since the precise characterization of the limiting spectral gap has been previously established in \cite{Halfin-Whitt-Gamarnik_2013}, we simply provide a bound on $\liminf_{n \rightarrow \infty} \rate$.

\subsection{\texorpdfstring{$\liminf_{n \rightarrow \infty} \rate$}{liminf Rate as n tends to infinity} bounds}
First, we have the $\liminf$ results for $\alpha \in (0,1/2)$ and $\alpha \in (1/2,\infty)$, which are attainable from the proof of Theorem \ref{thm: unifying-theorem} if we let $n \to \infty$.
\begin{prop}
    \label{prop: mixing-rate-liminf}
    Consider the sequence of continuous-time $M/M/n$ queues with arrival
    rates $\lambda_n=n-n^{1-\alpha}$ and unit service rate. The following lower bounds hold:
    \begin{enumerate}
        \item If $\alpha\in(1/2,\infty)$, then
        \begin{align}
            \liminf_{n\rightarrow\infty}
            \frac{\rate}{(\sqrt{n}-\sqrt{\lambda_n})^2} \geq 1.
        \end{align}
        \item If $\alpha\in(0,1/2)$, then
        \begin{align}
            \liminf_{n\rightarrow\infty}\rate \geq 1.
        \end{align}
    \end{enumerate}
\end{prop}

\proof{\textit{Proof of Proposition \ref{prop: mixing-rate-liminf}}:}
We first consider the Super-Halfin-Whitt regime. If $\alpha\in[1,\infty)$, then we have previously established in the proof of Theorem \ref{thm: unifying-theorem} that the system admits the Poincar\'e constant $\frac{1}{(\sqrt{n}-\sqrt{\lambda_n})^2}$, which implies
\begin{align}
    \rate \geq (\sqrt{n}-\sqrt{\lambda_n})^2.
\end{align}
If $\alpha\in(1/2,1)$, we have
\begin{align}
    \rate &\geq \frac{(\sqrt{n}-\sqrt{\lambda_n})^2}{1 + 1.04 \times 10^5 n^{2-4\alpha} + 1.6 \times 10^5 n^{3-6\alpha}} \nonumber\\
    \Leftrightarrow \frac{\rate}{(\sqrt{n}-\sqrt{\lambda_n})^2}
    &\geq \frac{1}{1 + 1.04 \times 10^5 n^{2-4\alpha} + 1.6 \times 10^5 n^{3-6\alpha}}.
\end{align}
Taking $n \rightarrow \infty$ on both sides, we have
\begin{align}
    \liminf_{n \rightarrow \infty} \frac{\rate}{(\sqrt{n}-\sqrt{\lambda_n})^2} \geq 1.
\end{align}
For $\alpha \in (0,1/2)$, denote $\rho_n = \frac{\lambda_n}{n}$ and $\gamma_n = \frac{n^{1-2\alpha}}{2(1+\rho_n)}$, we have from \eqref{eqn: sub-hw-poincare} and Proposition \ref{prop: poincare-equals-mixing} that
\begin{align*}
    \rate \geq \frac{1}{1 + \frac{1}{\gamma_n} + 28.2 n^{2\alpha-1} + \br{1+28.2 n^{2\alpha-1}}
    \frac{n \exp\br{-\frac{1}{2}n^{1-2\alpha}}}{\sqrt{8\pi e}\gamma_n}}.
\end{align*}
Here, observe that $\gamma_n \geq \frac{n^{1-2\alpha}}{4} \rightarrow \infty$ and $n \exp\br{-\frac{1}{2}n^{1-2\alpha}} \rightarrow 0$ for $n \rightarrow \infty$ and $\alpha \in (0,1/2)$. Hence, this implies
\begin{align*}
    \lim_{n \rightarrow \infty} 1 + \underbrace{\frac{1}{\gamma_n}}_{\to 0} + \underbrace{28.2 n^{2\alpha-1}}_{\to 0} + \underbrace{\br{1+28.2 n^{2\alpha-1}}
    \frac{n \exp\br{-\frac{1}{2}n^{1-2\alpha}}}{\sqrt{8\pi e}\gamma_n}}_{\to 0} = 1.
\end{align*}
And therefore, we have
\begin{align}
    \liminf_{n \rightarrow \infty} \rate \geq 1.
\end{align}
$\square$ \endproof

\subsection{\texorpdfstring{$\limsup_{n \rightarrow \infty} \rate$}{limsup Rate as n tends to infinity} bounds}
Now that we have established $\liminf$ of the mixing rate, we shall proceed to show $\limsup$ of the mixing rate via the following proposition.
\begin{prop}
    \label{prop: mixing-rate-limsup}
    Consider the sequence of continuous-time $M/M/n$ queues with arrival rates $\lambda_n=n-n^{1-\alpha}$ and unit service rate. The following upper bounds hold.
    \begin{enumerate}
        \item If $\alpha\in(0,1/2)$, then
        \begin{align}
            \limsup_{n\rightarrow\infty}\rate \leq 1.
        \end{align}
        \item If $\alpha\in(1/2,\infty)$, then
        \begin{align}
            \limsup_{n\rightarrow\infty}
            \frac{\rate}{(\sqrt{n}-\sqrt{\lambda_n})^2}
            \leq 1.
        \end{align}
    \end{enumerate}
\end{prop}

\proof{\textit{Proof of Proposition \ref{prop: mixing-rate-limsup}}:}
We consider two separate cases.

\textbf{Case 1 ($\alpha \in (0,1/2)$):} Choose the test function $g(q)=q$ in the variational formula
\eqref{eqn: spectral-gap-variational-form}. Since $g(q+1)-g(q)=1$ for every
$q$, the Dirichlet form is exactly
\begin{align}
    \cE(g,g) = \sum_{q=0}^{\infty}\nu_n(q)\lambda_n \br{g(q+1)-g(q)}^2 = \lambda_n.
\end{align}
On the other hand, from Corollary \ref{corollary: mmn-variance-bound} and Proposition \ref{prop: spectral-gap-poincare-variational-form}, we have
\begin{align}
    \rate \leq \frac{\lambda_n}{\Var_{\nu_n}(X)} \leq \frac{\lambda_n}{\lambda_n} = 1.
\end{align}
Consequently, this proves
\begin{align}
    \limsup_{n\rightarrow\infty}\rate \leq 1
    \qquad
    \text{for } \alpha\in(0,1/2).
\end{align}

\textbf{Case 2 ($\alpha \in (1/2,\infty)$):} Let $p_n=\nu_n(q \geq n)$. Conditioned on $\{q\geq n\}$, we have the stationary distribution is geometric:
\begin{align}
    \nu_n(n+j\mid q\geq n)=(1-\rho_n)\rho_n^j,
    \qquad j\geq 0.
\end{align}
This is the stationary distribution of an $M/M/1$ queue with arrival rate
$\lambda_n$ and service rate $n$. Let $1<a<\rho_n^{-1/2}=\sqrt{n/\lambda_n}$, and define
\begin{align}
    f_a(q)
    =
    \begin{cases}
        1, & 0\leq q\leq n,\\
        a^{q-n}, & q>n.
    \end{cases}
\end{align}
Since $\rho_n a^2<1$, the function
$g_a(j) = a^j \, \forall j \in \Z_{\geq 0}$ is square-integrable with respect to the geometric law
$\pi_{\rho_n}(j)=(1-\rho_n)\rho_n^j$. Note that the function $f_a$ is constant on $\{0,\ldots,n\}$, so all edges below $n$ contribute zero to the Dirichlet form. Hence,
\begin{align}
    \cE(f_a,f_a) &= p_n \sum_{j=0}^{\infty} (1-\rho_n)\rho_n^j \lambda_n a^{2j}(a-1)^2 = p_n \frac{\lambda_n(1-\rho_n)(a-1)^2}{1-\rho_n a^2}.
\end{align}
By the variance decomposition for the mixture between $\{q<n\}$ and
$\{q\geq n\}$, we have
\begin{align}
    \Var_{\nu_n}(f_a) &= p_n\Var_{\pi_{\rho_n}}(g_a) + p_n(1-p_n)\br{\E_{\pi_{\rho_n}}[g_a]-1}^2 \geq p_n\Var_{\pi_{\rho_n}}(g_a).
\end{align}
Moreover, we have
\begin{align}
    \E_{\pi_{\rho_n}}[g_a]
    =
    \frac{1-\rho_n}{1-\rho_n a},
    \qquad
    \E_{\pi_{\rho_n}}\sqbr{g_a^2}
    =
    \frac{1-\rho_n}{1-\rho_n a^2}.
\end{align}
Therefore,
\begin{align}
    \Var_{\pi_{\rho_n}}(g_a)
    &=
    \frac{1-\rho_n}{1-\rho_n a^2}
    -
    \br{\frac{1-\rho_n}{1-\rho_n a}}^2 \\
    &=
    \frac{(1-\rho_n)\rho_n(a-1)^2}
    {(1-\rho_n a^2)(1-\rho_n a)^2}.
\end{align}
Combining these identities gives
\begin{align}
    \rate \leq \frac{\cE(f_a,f_a)}{\Var_{\nu_n}(f_a)} \leq\frac{\frac{\lambda_n(1-\rho_n)(a-1)^2}{1-\rho_n a^2}}{\frac{(1-\rho_n)\rho_n(a-1)^2}{(1-\rho_n a^2)(1-\rho_n a)^2}} = \frac{\lambda_n}{\rho_n}(1-\rho_n a)^2 = n(1-\rho_n a)^2.
\end{align}
Letting $a\uparrow \rho_n^{-1/2}$ yields
\begin{align}
    \rate
    \leq
    n(1-\sqrt{\rho_n})^2
    =
    (\sqrt{n}-\sqrt{\lambda_n})^2.
\end{align}
Hence,
\begin{align}
    \limsup_{n\rightarrow\infty}
    \frac{\rate}{(\sqrt{n}-\sqrt{\lambda_n})^2}
    \leq 1
    \qquad
    \text{for } \alpha\in(1/2,\infty).
\end{align}
$\square$ \endproof

{\textbf{Remark}: While choosing the test function for $\alpha \in (0,1/2)$ is quite straightforward since $g(q) = q$ is also the first non-trivial eigenfunction for the OU process and is similar to the first Charlier polynomial \citep{dominguez-book-orthogonal-polynomial}, the test function for the $\alpha \in (1/2,\infty)$ case is quite interesting as it happens to be exactly the Lyapunov function in the proof of Theorem \ref{thm: unifying-theorem} for $\alpha \in (0,1/2)$. Nevertheless, the idea is to capture the geometric tail of the $M/M/n$ system, which we find inspiration from the proof of Geometric Poincar\'e \citep{Joulin2004-geometric-Poincare}. For the behavior below $n$, we simply cancel it out by choosing a constant function.}

\subsection{Final steps for the proof of Theorem \ref{thm: limiting-mixing-rates}}
Now, by combining the $\liminf$ and $\limsup$ results, we complete the proof of Theorem \ref{thm: limiting-mixing-rates}. First, we obtain the limiting results for $\alpha \in (0,1/2) \cup (1/2, \infty)$. For $\alpha = 1/2$, we only have a $\liminf$ proof, so we will treat it separately later in this subsection.
\proof{\textit{Proof of Theorem \ref{thm: limiting-mixing-rates}} for $\alpha \in (0,1/2) \cup (1/2,\infty)$:}
For $\alpha\in(1/2,\infty)$, Proposition \ref{prop: mixing-rate-liminf} gives
\begin{align}
    \liminf_{n\rightarrow\infty}
    \frac{\rate}{(\sqrt{n}-\sqrt{\lambda_n})^2}
    \geq 1,
\end{align}
while Proposition \ref{prop: mixing-rate-limsup} gives
\begin{align}
    \limsup_{n\rightarrow\infty}
    \frac{\rate}{(\sqrt{n}-\sqrt{\lambda_n})^2}
    \leq 1.
\end{align}
Therefore,
\begin{align}
    \lim_{n\rightarrow\infty}
    \frac{\rate}{(\sqrt{n}-\sqrt{\lambda_n})^2}
    =
    1.
\end{align}
Similarly, for $\alpha\in(0,1/2)$, Proposition \ref{prop: mixing-rate-liminf} gives
\begin{align}
    \liminf_{n\rightarrow\infty}\rate\geq 1,
\end{align}
while Proposition \ref{prop: mixing-rate-limsup} gives
\begin{align}
    \limsup_{n\rightarrow\infty}\rate\leq 1.
\end{align}
Therefore, for $\alpha \in (0,1/2)$, we have
\begin{align}
    \lim_{n\rightarrow\infty}\rate=1.
\end{align}
$\square$ \endproof

\textbf{Remark on the limiting results}: Previously, \cite{Halfin-Whitt-Gamarnik_2013} suggests that there exists a threshold $\rho_n^\star$ such that if the system load $\rho_n = \frac{\lambda_n}{n}$ is at least $\rho_n^\star$ then the spectral gap is exactly $(\sqrt{n}-\sqrt{\lambda_n})^2$. While there is no precise characterization of $\rho_n^\star$, \cite{Halfin-Whitt-Gamarnik_2013} shows that there exists $B^\star > 0$ such that
\begin{align*}
    \lim_{n \rightarrow \infty} \sqrt{n}(1-\rho_n^\star) = B^\star.
\end{align*}
Hence, when $\alpha \in (1/2,\infty)$, we will have $\rho_n \geq \rho_n^\star$ for all sufficiently large $n$. Thus, our results confirm the findings in \cite{Halfin-Whitt-Gamarnik_2013}. For $\rho_n < \rho_n^\star$, it is unclear what the precise rate is since Theorem 3 in \cite{Halfin-Whitt-Gamarnik_2013} only says that $\operatorname{Gap}(\cL) < (\sqrt{n}-\sqrt{\lambda_n})^2$ and shows that $\operatorname{Gap}(\cL) \geq \frac{1}{2} \sqrt{\frac{\lambda_n}{n}}$ for sufficiently large $n$ (please refer to Appendix \ref{ssec: other-discussions-with-prev-works} for a more detailed discussion).

\subsection{Proof of Theorem \ref{thm: unifying-theorem} and Theorem \ref{thm: limiting-mixing-rates} in the \texorpdfstring{$\alpha = 1/2$}{alpha = 1/2} regime}
\label{ssec: unified-halfin-whitt-proof}
For the $\alpha = 1/2$ regime, we are not able to obtain a tight finite-time bound and characterize a phase transition as in \citep{Halfin-Whitt-Gamarnik_2013}. Yet, we can still show that the mixing is bounded below by some universal constant by performing a similar analysis as in the $\alpha \in (1/2,1)$ case, with some additional refinements on the constants as follows.

\proof{\textit{Proof of Theorem \ref{thm: unifying-theorem} and Theorem \ref{thm: limiting-mixing-rates} for $\alpha = 1/2$}:}
From Claim \ref{claim: t1} and Claim \ref{claim: t2}, we have
\begin{align*}
    g_1(n)+g_2(n) \leq 1277+1975=3252.
\end{align*}
Moreover, from \eqref{eqn: b-term-super-hw-drift}, for
$\alpha=1/2$ and $n\geq n_0=110$, we have
\begin{align*}
    b(n-1) &= \sqrt{n}\br{1+\frac{3}{4(\sqrt{n}-1)}}
    \geq e,\\
    b(q) &= \frac{e^{\frac{1}{1-\frac{1}{\sqrt{n}}}}}{1-\frac{1}{\sqrt{n}}} \geq e, \qquad q\in K\setminus\{n-1\}.
\end{align*}
From the fact the $\nu_K$ is a probability distribution on $K$, this gives
\begin{align*}
    \tau(K)=\sum_{q\in K}b(q)\nu_K(q) \geq e.
\end{align*}
From Lemma \ref{lemma: weighted-poincare-super-HW}, we obtain
\begin{align*}
    \tau(K)C_b
    &\leq \frac{g_1(n)+g_2(n)}{\tau(K)}
    \leq \frac{3252}{e}.
\end{align*}
Applying the Stitching Theorem \ref{theorem: lyapunov-poincare} and using $(\sqrt{n}-\sqrt{\lambda_n})^2\geq 1/4$, we have
\begin{align*}
    C_P(n) &\leq \frac{1+\tau(K)C_b}{(\sqrt{n}-\sqrt{\lambda_n})^2} \leq 4\br{1+\frac{g_1(n)+g_2(n)}{\tau(K)}} \leq 4\br{1+\frac{3252}{e}} \leq 4789.38.
\end{align*}
Thus, when $\alpha = 1/2$, Proposition \ref{prop: poincare-equals-mixing} gives $\rate\geq 1/4789.38$ for all $n\geq n_0$. Next, we establish the limiting mixing rate bound. From \eqref{eqn: nu-K-endpoint-upper-bound-super-hw}, we have
\begin{align*}
    0\leq \nu_K(n-1) \leq \frac{1}{\lfloor\sqrt{n}\rfloor}.
\end{align*}
Using the lower bound $\sqrt{n}\nu_K(n-1)\geq L_K(n)$ from Lemma \ref{lemma: roughly-uniform}, we obtain
\begin{align*}
    \tau(K) &= \sqrt{n}\br{1+\frac{3}{4(\sqrt{n}-1)}} \nu_K(n-1) + \frac{e^{\frac{1}{1-\frac{1}{\sqrt{n}}}}}{1-\frac{1}{\sqrt{n}}} \br{1-\nu_K(n-1)}\\
    &\geq \br{1+\frac{3}{4(\sqrt{n}-1)}} L_K(n) + \frac{e^{\frac{1}{1-\frac{1}{\sqrt{n}}}}}{1-\frac{1}{\sqrt{n}}} \br{1-\nu_K(n-1)}.
\end{align*}
Since $L_K(n)\to 1/(2e)$, it follows that
\begin{align*}
    \liminf_{n\to\infty}\tau(K)
    \geq \frac{1}{2e}+e.
\end{align*}
Combining this with
$g_1(n)\to 2e^4$ and $g_2(n)\to e^5$, we obtain
\begin{align*}
    \limsup_{n\to\infty}C_P(n) &\leq 4\br{1+\frac{2e^4+e^5}{e+\frac{1}{2e}}} \leq 359.06.
\end{align*}
Therefore, Proposition \ref{prop: poincare-equals-mixing} gives
\begin{align*}
    \liminf_{n\to\infty}\rate \geq \frac{1}{359.06}.
\end{align*}
Hence, we are done.
$\square$ \endproof

\textbf{Remark on the constant}: It is noteworthy that our goal here is to show that the mixing rate is bounded below by some universal constant, which is $\frac{1}{4789.38}$ in this case. Since we are not able to obtain a convergence rate that matches the limiting behavior, we have not made any effort to optimize this constant. Based on the spectral gap characterization at Halfin-Whitt in \citep{Halfin-Whitt-Gamarnik_2013}, it seems that characterizing the phase transition in this regime is rather non-trivial.

}

\section{Proof of finite-time statistics results}
\subsection{Proof of Corollary \ref{corollary: mean-queue-length-mmn}}
\label{ssec: mean-queue-length-mmn-proof}
From Proposition \ref{prop: variational-representation-chi-square} and the $M/M/n$ variance bound in Corollary \ref{corollary: mmn-variance-bound}, we can establish the finite-time mean queue length as follows.

\proof{\textit{Proof of Corollary \ref{corollary: mean-queue-length-mmn}}:}
From Theorem \ref{thm: unifying-theorem}, we have
\begin{align}
    \label{eqn: chi-square-convergence-1}
    \chi^2(\pi_{n,t},\nu_n) = \sum_{x = 0}^\infty \frac{\br{\pi_{n,t}(x) - \nu_n(x)}^2}{\nu_n(x)} \leq e^{-2\rate \, \times \, t} \chi^2(\pi_{n,0}, \nu_n).
\end{align}
From Corollary \ref{corollary: moment-bound}, we have that
\begin{align}
    \label{eqn: donsker-varadhan-first-moment}
    \left| \E_{\pi_{n,t}} \sqbr{X} - \E_{\nu_n} \sqbr{X} \right| \leq \chi(\pi_{n,t},\nu_n) \sqrt{\E_{\nu_n}\sqbr{X^2}-\E_{\nu_n}\sqbr{X}^2}
\end{align}
From \eqref{eqn: chi-square-convergence-1} and \eqref{eqn: donsker-varadhan-first-moment}, we have
\begin{align}
    \label{eqn: mean-queue-length-bound-with-general-mixing-rate}
    \big| \E_{\pi_{n,t}}\sqbr{X}-\E_{\nu_n}\sqbr{X} \big| \leq e^{-\rate \, \times \, t} \sqrt{\Var_{\nu_n}\sqbr{X}} \chi(\pi_{n,0},\nu_n).
\end{align}
{From Corollary \ref{corollary: mmn-variance-bound}, we have
\begin{align}
    \sqrt{\Var_{\nu_n}[X]} \leq \sqrt{2n + \frac{5}{4} n^{2\alpha}} \leq \sqrt{2}\br{\sqrt{n}+n^\alpha}.
\end{align}
This implies}
\begin{align}
    \big| \E_{\pi_{n,t}}\sqbr{X}-\E_{\nu_n}\sqbr{X} \big| \leq e^{-\rate \, \times \, t} \sqrt{2}\br{\sqrt{n}+n^\alpha} \chi(\pi_{n,0},\nu_n).
\end{align}
$\square$ \endproof

\textbf{Remark on the dependence of $n$}: While it is known that with a proper choice of the rescaling factor, we can show the existence of the MGF and consequently, the boundedness of the second moment, it would give a weak characterization on the dependence of $n$. By establishing the second moment bound of $\nu$, we have shown that the $n$ dependency is only polynomial in $n$, and thus, it would be negligible given the exponential convergence of the system.

\subsection{Proof of Corollary \ref{corollary: tail-bounds}}
\label{sssec: tail-bounds-proof}
To establish tail bound results, we first need to obtain bounds on the MGF, which involves careful analysis and choice of parameters. Our MGF bounds are stated as follows.

\begin{lemma}
\label{lemma: mgf-bounds}
Let $n \geq n_0 = \max\left\lbrace 110, \left\lceil 2^{\frac{1}{\alpha}} \right\rceil\right\rbrace$, and $\nu_n$ be the stationary distribution of the $M/M/n$ system with arrival rate $\lambda = n-n^{1-\alpha}$ and unit service rate. Denote
$\varepsilon = 1-\frac{\lambda}{n}$ and let $\delta > 0$, we have
\begin{align}
    \label{eqn: mgf-pre-limit-bound}
    \E_{\nu_n}\sqbr{e^{\frac{\varepsilon (q-n)}{1+\delta}}} \leq 1+\frac{1}{\delta}.
\end{align}
Moreover, if $\alpha \in (0,1/2)$ then
\begin{align*}
    \E_{\nu_n}\sqbr{e^{\frac{\varepsilon (q-n)}{1+\delta}}} \leq e^{-\frac{n^{1-2\alpha}}{2(1+\delta)}} + \br{10+\frac{2}{\delta}}n^{2\alpha-1}.
\end{align*}
\end{lemma}

{
\proof{\textit{Proof of Lemma \ref{lemma: mgf-bounds}}:}
Let $A_n=\sum_{q=0}^{n-1}\frac{\lambda^q}{q!}, B_n=\frac{\lambda^n}{n!}$. Since $1-\lambda/n = \varepsilon =n^{-\alpha}$, the normalizing constant of the Erlang-C stationary distribution is
\begin{align*}
    Z_n = A_n+B_n \sum_{q=n}^\infty \br{\frac{\lambda}{n}}^{q-n} = A_n+B_n \frac{1}{1-\frac{\lambda}{n}} = A_n+B_n n^\alpha.
\end{align*}
By reversibility of the $M/M/n$, we have
\begin{align*}
    \nu_n(q)
    &=
    \frac{1}{Z_n}\frac{\lambda^q}{q!},
    \qquad 0\le q\le n-1, \\
    \nu_n(q)
    &=
    \frac{1}{Z_n}B_n\br{\frac{\lambda}{n}}^{q-n},
    \qquad q\ge n.
\end{align*}
Hence, we can compute the MGF as follows
\begin{align}
    \label{eqn: exact-mgf-epsilon}
    \E_{\nu_n}\sqbr{e^{\frac{\varepsilon(q-n)}{1+\delta}}} &= \frac{\sum_{q=0}^{n-1} \frac{\lambda^q}{q!} e^{\frac{\varepsilon(q-n)}{1+\delta}} + B_n \sum_{k=0}^{\infty} \br{e^{\frac{\varepsilon}{1+\delta}}\frac{\lambda}{n}}^k}{A_n+B_n n^\alpha} = \frac{\sum_{q=0}^{n-1} \frac{\lambda^q}{q!} e^{\frac{\varepsilon(q-n)}{1+\delta}} + \frac{B_n}{1-e^{\frac{\varepsilon}{1+\delta}}\frac{\lambda}{n}}}{A_n+B_n n^\alpha}.
\end{align}
Here, the geometric series is finite because
\begin{align*}
    0 < e^{\frac{\varepsilon}{1+\delta}}\frac{\lambda}{n} = e^{\frac{\varepsilon}{1+\delta}}(1-\varepsilon) < 1.
\end{align*}
We next bound the geometric factor. Since $n\ge n_0$, we have  $\varepsilon = 1-\frac{\lambda}{n} \leq 1/2$. From Lemma \ref{lemma: log-lower-bound}, we have $e^x\le 1+x+x^2$ for $x \in [0,1/2]$. Since we have $\frac{\varepsilon}{1+\delta} \in (0,1/2)$, this implies that
\begin{align*}
    e^{\frac{\varepsilon}{1+\delta}}(1-\varepsilon) &\leq
    \br{1+\frac{\varepsilon}{1+\delta}+\frac{\varepsilon^2}{(1+\delta)^2}}(1-\varepsilon) \leq 1-\frac{\delta}{1+\delta}\varepsilon
\end{align*}
by direct expansion. Therefore,
\begin{align}
    \label{eqn: geometric-factor-bound}
    \frac{1}{1-e^{\frac{\varepsilon}{1+\delta}}\frac{\lambda}{n}}
    \leq
    \frac{1+\delta}{\delta}\frac{1}{\varepsilon}
    =
    \br{1+\frac{1}{\delta}}n^\alpha.
\end{align}
Since $q<n$ implies $e^{\frac{\varepsilon(q-n)}{1+\delta}}\leq 1$, we have from \eqref{eqn: exact-mgf-epsilon} and \eqref{eqn: geometric-factor-bound} that
\begin{align*}
    \E_{\nu_n}\sqbr{e^{\frac{\varepsilon(q-n)}{1+\delta}}} &\leq \frac{A_n+\br{1+\frac{1}{\delta}}B_n n^\alpha}{A_n+B_n n^\alpha} \leq 1+\frac{1}{\delta}.
\end{align*}
This proves \eqref{eqn: mgf-pre-limit-bound}.

\noindent
For $\alpha\in(0,1/2)$, denote
\begin{align*}
    m=\left\lfloor n-\frac{n^{1-\alpha}}{2}\right\rfloor + 1.
\end{align*}
For $q < m$, we have $n-q\ge n^{1-\alpha}/2$, and hence
\begin{align*}
    e^{\frac{\varepsilon(q-n)}{1+\delta}} \leq e^{-\frac{n^{1-2\alpha}}{2(1+\delta)}}.
\end{align*}
For $m \leq q<n$, we simply use
\begin{align*}
    e^{\frac{\varepsilon(q-n)}{1+\delta}}\leq 1.
\end{align*}
Also, since $m > \lambda$ and the sequence $\lambda^q/q!$ is decreasing for $q > \lambda$, we have
\begin{align}
    \label{eqn: middle-poisson-sum-bound}
    \sum_{m \leq q < n}\frac{\lambda^q}{q!} &\leq \frac{\lambda^m}{m!} \sum_{m \leq q < n}\br{\frac{\lambda}{m}}^{q-m} \leq \frac{\lambda^m}{m!}\frac{1}{1-\lambda/m} \leq 2n^\alpha\frac{\lambda^m}{m!}.
\end{align}
Moreover, for all $n \geq n_0$, the interval $\{\lceil \lambda\rceil,\ldots,m\}$ contains at least $\lceil n^{1-\alpha}/4 \rceil$ integers, and each corresponding term is at least $\lambda^m/m!$. Hence
\begin{align}
    \label{eqn: denominator-lower-m}
    Z_n\ge A_n
    \geq
    \frac{n^{1-\alpha}}{4}\frac{\lambda^m}{m!}.
\end{align}
Similarly, the interval $\{\lceil \lambda\rceil,\ldots,n-1\}$ contains at least $n^{1-\alpha}/2$ integers, and each corresponding term is at least $B_n=\lambda^n/n!$. Therefore,
\begin{align}
    \label{eqn: denominator-lower-n}
    Z_n\ge A_n
    \geq
    \frac{n^{1-\alpha}}{2}B_n.
\end{align}
Combining \eqref{eqn: exact-mgf-epsilon}, \eqref{eqn: geometric-factor-bound}, \eqref{eqn: middle-poisson-sum-bound}, \eqref{eqn: denominator-lower-m}, and \eqref{eqn: denominator-lower-n}, we get
\begin{align*}
    \E_{\nu_n}\sqbr{e^{\frac{\varepsilon(q-n)}{1+\delta}}} &\leq e^{-\frac{n^{1-2\alpha}}{2(1+\delta)}} + \frac{2n^\alpha \frac{\lambda^m}{m!}}{\frac{n^{1-\alpha}}{4}\frac{\lambda^m}{m!}} + \frac{\br{1+\frac{1}{\delta}}B_n n^\alpha}{\frac{n^{1-\alpha}}{2}B_n} \\
    &= e^{-\frac{n^{1-2\alpha}}{2(1+\delta)}} + 8n^{2\alpha-1} + 2\br{1+\frac{1}{\delta}}n^{2\alpha-1} \\
    &= e^{-\frac{n^{1-2\alpha}}{2(1+\delta)}} + \br{10+\frac{2}{\delta}}n^{2\alpha-1}.
\end{align*}
$\square$ \endproof

From the MGF bounds in Lemma \ref{lemma: mgf-bounds}, we can now formally prove Corollary \ref{corollary: tail-bounds} as follows.

\proof{\textit{Proof of Corollary \ref{corollary: tail-bounds}}:}
First, let's consider the regime $\alpha \in (0,\infty)$ and recall that $\varepsilon=1-\frac{\lambda}{n}$. Let $n_0 = \max\left\lbrace 110, \left\lceil 2^{\frac{1}{\alpha}} \right\rceil\right\rbrace$. From Lemma \ref{lemma: mgf-bounds} and Corollary \ref{corollary: mgf-bound}, we have for some $\delta > 0$ to be chosen that
\begin{align*}
    \left| \E_{\pi_{n,t}} \sqbr{e^{\frac{\varepsilon (q-n)}{2(1+\delta)}}} - \E_{\nu_n} \sqbr{e^{\frac{\varepsilon (q-n)}{2(1+\delta)}}} \right| \leq \chi\br{\pi_{n,t}, \nu_n} \sqrt{\E_{\nu_n}\sqbr{e^{\frac{\varepsilon (q-n)}{(1+\delta)}}}}.
\end{align*}
Rearrange the inequalities and apply the Cauchy-Schwarz inequality, we have
\begin{align}
    \label{eqn: finite-time-MGF-bound}
    \E_{\pi_{n,t}} \sqbr{e^{\frac{\varepsilon (q-n)}{2(1+\delta)}}} \leq \E_{\nu_n} \sqbr{e^{\frac{\varepsilon (q-n)}{2(1+\delta)}}} + \chi\br{\pi_{n,t}, \nu_n}   \sqrt{\E_{\nu_n}\sqbr{e^{\frac{\varepsilon (q-n)}{(1+\delta)}}}} \leq \br{1 +  \chi\br{\pi_{n,t}, \nu_n}} \sqrt{\E_{\nu_n}\sqbr{e^{\frac{\varepsilon (q-n)}{(1+\delta)}}}}
\end{align}
Now, we shall obtain the tail bound for $q$. Note that by Markov's inequality and Equation \eqref{eqn: finite-time-MGF-bound}, we have
\begin{align*}
    \P_{\pi_{n,t}}\sqbr{\varepsilon (q-n) > x} \leq \E_{\pi_{n,t}} \sqbr{e^{\frac{\varepsilon (q-n)}{2(1+\delta)}}} e^{-\frac{x}{2(1+\delta)}} \leq \br{1 +  \chi\br{\pi_{n,t}, \nu_n}} \sqrt{\E_{\nu_n}\sqbr{e^{\frac{\varepsilon (q-n)}{(1+\delta)}}}} e^{-\frac{x}{2(1+\delta)}}.
\end{align*}
And so, we have from Theorem \ref{thm: unifying-theorem} that
\begin{align*}
    \P_{\pi_{n,t}}\sqbr{\varepsilon \br{q-n} > x} &\leq \operatorname{prefactor} \times \br{1 +  \chi\br{\pi_{n,t}, \nu_n}} e^{-\frac{x}{2}} \leq \operatorname{prefactor} \times \br{1 +  e^{-\rate \, \times \, t} \chi\br{\pi_{n,0}, \nu_n}} e^{-\frac{x}{2}}
\end{align*}
where the $\operatorname{prefactor}$ term is to be determined by the value of $\alpha$ and the term $\rate$ is determined as in Theorem \ref{thm: unifying-theorem}. In particular, for $\alpha \in (1/2,\infty)$, we have $\E_{\nu_n}\sqbr{e^{\frac{\varepsilon (q-n)}{(1+\delta)}}} \leq 1+\frac{1}{\delta}$ from Lemma \ref{lemma: mgf-bounds}. Therefore, choosing $\delta = \frac{1}{x}$ gives us
\begin{align*}
    \sqrt{1+\frac{1}{\delta}} e^{-\frac{x}{2(1+\delta)}} = \sqrt{x+1} e^{-\frac{x^2}{2(x+1)}} \leq \sqrt{e(x+1)} e^{-\frac{x}{2}}.
\end{align*}
And so, we have
\begin{align*}
    \operatorname{prefactor} = \sqrt{e(x+1)}.
\end{align*}
On the other hand, for $\alpha \in (0,1/2)$, we have from Lemma \ref{lemma: mgf-bounds} and $\delta = \frac{1}{x}$ gives
\begin{align*}
    \operatorname{prefactor} =  \sqrt{e^{1-\frac{x n^{1-2\alpha}}{2(x+1)}} + 10e(x+1)n^{2\alpha-1}},
\end{align*}
as desired. $\square$ \endproof

}

\subsection{Proof of Corollary \ref{corollary: finite-time-idle-server-probability}}
\label{ssec: idle-server-finite-time}
In this subsection, we provide a finite-time tail bound for the number of idle servers, which is denoted as $r_n = \sqbr{n-q}^+$.

\proof{\textit{Proof of Corollary \ref{corollary: finite-time-idle-server-probability}}:}
    Similarly to the finite-time mean queue length proof, we also multiply both the numerator and denominator with a function to obtain the desired quantity. To do this, we use an indicator function to capture the quantity $\P_{\pi_{n,t}}\sqbr{r_n > 0}$. Let $1_{0 \leq x \leq n-1}$ be the indicator function when the queue length $x$ is less than $n$, we have $\E_{\pi_{n,t}}\sqbr{1_{0 \leq x \leq n-1}} = \sum_{x=0}^{n-1} \pi_{n,t}(x) = \P_{\pi_{n,t}}[r_n > 0]$ and $\E_{\nu_n}\sqbr{1_{0 \leq x \leq n-1}} = \sum_{x=0}^{n-1} \nu_n(x) = \P_{\nu_n}[r_n > 0]$. We have:
    \begin{align*}
        \chi^2\br{\pi_{n,t}, \nu_n} = \sum_{x = 0}^\infty \frac{\br{\pi_{n,t}(x)-\nu_n(x)}^2}{\nu_n(x)} &\geq \sum_{x = 0}^{n-1} \frac{\br{ \pi_{n,t}(x)-\nu_n(x)}^2}{ \nu_n(x)} \\
        &\overset{(a)}{\geq} \frac{\br{\sum_{x = 0}^{n-1} |\pi_{n,t}(x) - \nu_n(x)|}^2}{\sum_{x = 0}^{n-1} \nu_n(x)} \\
        &\overset{(b)}{\geq} \frac{\br{\sum_{x = 0}^{n-1} \pi_{n,t}(x) - \sum_{x = 0}^{n-1} \nu_n(x)}^2}{\sum_{x = 0}^{n-1} \nu_n(x)} \\
        &= \frac{\br{\P_{\pi_{n,t}}\sqbr{r_n > 0} - \P_{\nu_n}\sqbr{r_n > 0}}^2}{\P_{\nu_n}\sqbr{r_n > 0}}
    \end{align*}
    where (a) is the Cauchy-Schwarz inequality and (b) follows from the inequality $|x| + |y| \geq |x+y| \, \forall x,y \in \R$. This implies the bound
    \begin{align}
        \left| \P_{\pi_{n,t}}\sqbr{r_n > 0} - \P_{\nu_n}\sqbr{r_n > 0} \right| &\leq  \chi\br{\pi_{n,t}, \nu_n} \sqrt{\P_{\nu_n}\sqbr{r_n > 0}}
    \end{align}
    which is equivalent to
    \begin{align}
        \P_{\nu_n}\sqbr{r_n > 0} - \chi\br{\pi_{n,t}, \nu_n} \sqrt{\P_{\nu_n}\sqbr{r_n > 0}} \leq \P_{\pi_{n,t}}\sqbr{r_n > 0} &\leq \P_{\nu_n}\sqbr{r_n > 0} +  \chi\br{\pi_{n,t}, \nu_n} \sqrt{\P_{\nu_n}\sqbr{r_n > 0}}.
    \end{align}
    Now, we will analyze $\P_{\pi_{n,t}}\sqbr{r_n > 0}$ depending on the regime. When we have $\alpha \in (1/2,\infty)$, from Theorem \ref{thm: unifying-theorem} and Theorem 8 in \citep{prakirt-confluence-large-deviation}, we have
    \begin{align*}
        \P_{\pi_{n,t}}\sqbr{r_n > 0} &\leq \P_{\nu_n}\sqbr{r_n > 0} + \sqrt{\P_{\nu_n}\sqbr{r_n > 0}} \chi\br{\pi_{n,t}, \nu_n} \\
        &\leq
        4 e \pi n^{\frac{1}{2}-\alpha} + 2\sqrt{e \pi} n^{\frac{1}{4}-\frac{\alpha}{2}} e^{-\rate \, \times \, t} \chi(\pi_{n,0}, \nu_n).
    \end{align*}
    In the special case $\alpha = 1/2$, we have
    \begin{align}
        \lim_{n \rightarrow \infty} \P_{\nu_n}[r_n > 0] = \frac{\sqrt{2\pi} e^{\frac{1}{2}} \Phi(1)}{1 + \sqrt{2\pi} e^{\frac{1}{2}} \Phi(1)}.
    \end{align}
    This means that there exists a sequence of positive parameters $R_n$ such that $R_n \in [0,1]$ and $\lim_{n \rightarrow \infty} R_n = \frac{\sqrt{2\pi} e^{\frac{1}{2}} \Phi(1)}{1 + \sqrt{2\pi} e^{\frac{1}{2}} \Phi(1)}$ where
    \begin{align}
        \P_{\pi_{n,t}}\sqbr{r_n > 0} &\leq R_n + 2\sqrt{e \pi} n^{\frac{1}{4}-\frac{\alpha}{2}} e^{-\rate \, \times \, t} \chi(\pi_{n,0}, \nu_n).
    \end{align}
    On the other hand, when we have $\alpha \in (0,1/2)$, we have
    \begin{align*}
        \P_{\pi_{n,t}}\sqbr{r_n > 0} &\geq \P_{\nu_n}\sqbr{r_n > 0} -  \sqrt{\P_{\nu_n}\sqbr{r_n > 0}} \chi\br{\pi_{n,t}, \nu_n} \\
        &\geq \P\sqbr{r_n > 0} - \chi\br{\pi_{n,0}, \nu_n} e^{-\rate \, \times \, t} \sqrt{\P_{\nu_n}\sqbr{r_n > 0}} \\
        &\geq 1-\kappa n^{\alpha - \frac{1}{2}} e^{-n^{\frac{1}{2}-\alpha}} - e^{-\rate \, \times \, t} \chi\br{\pi_{n,0}, \nu_n}.
    \end{align*}
    The second inequality follows from Theorem \ref{thm: unifying-theorem} and the last inequality follows from $\P\sqbr{r_n > 0} \leq 1$ and Theorem 8 in \citep{prakirt-confluence-large-deviation}. {Here, the term $\rate$ is described as in Theorem \ref{thm: unifying-theorem} and $\kappa$ is the corresponding constant in Theorem 8 in \cite{prakirt-confluence-large-deviation}.}
$\square$ \endproof

\section{Comparison with other analyses in the literature}
\label{ssec: other-discussions-with-prev-works}

{In this section, we compare our work with other analyses of the Erlang-C system in the literature, with the most notable being \citep{Van_Doorn1985-decay-bounds, zeifman_lognorm, Halfin-Whitt-Gamarnik_2013}. Compared to these previous works, our Stitching Theorem (Theorem \ref{theorem: lyapunov-poincare}) provides a general framework to analyze the mixing time of positive recurrent reversible Markov chains, whereas the spectral representation approach in \citep{Van_Doorn1985-decay-bounds, Halfin-Whitt-Gamarnik_2013} and the log-norm approach in \citep{zeifman_lognorm} are mostly applicable for birth-and-death processes. In addition, we will take a deeper look into the results of each work in the following subsections.}

\subsection{Obtaining \texorpdfstring{$\ell_2$}{L2} convergence from spectral representation}

{For an irreducible Markov chain on the state space $\cS$, it is known that there exists a parameter $\beta_{ij}$ such that for all $i,j \in \cS$, we have
\begin{align}
    \label{eqn: beta-ij-exponential-decay}
    |P_{ij}(t) - \pi(j)| = O\br{e^{-\beta_{ij} t}}
\end{align}
where $P_{ij}(t) = \P\sqbr{q(t) = j | q(0) = i}$, $\pi$ is the stationary distribution \citep{Mu-Fa-Chen1991-l2-convergence}. From here, we shall formally state the notion of exponential ergodicity of the Markov chain from \citep{Mu-Fa-Chen1991-l2-convergence} and the definition of the exponential decay parameter as follows.}
{\begin{defi}
    \label{def: exponential-decay-and-exponential-ergodicity}
    Denote {$\hat{\beta} = \inf_{i,j \in \cS} \beta_{ij}$ where $\beta_{ij}$ are the largest  parameters satisfying \eqref{eqn: beta-ij-exponential-decay} for all $t \geq 0$}, we call $\hat{\beta}$ as the exponential decay parameter of the Markov chain. If $\hat{\beta} > 0$ then we call the process exponentially ergodic.
\end{defi}
Under this definition, $\hat{\beta}$ is the best parameter such that
\begin{align}
    |P_{ij}(t) - \pi(j)| = O\br{e^{-\hat{\beta} t}}
\end{align}
for all $i,j \in \cS$. Typically, using the spectral representation method by \citep{Karlin1957-KM-representation}, one will obtain convergence bounds in the form of \eqref{eqn: beta-ij-exponential-decay} and now our goal is to establish bounds on this exponential decay parameter $\hat{\beta}$. Using the results from \citep{Van_Doorn1985-decay-bounds}, we obtain the following bound on the exponential decay parameter of the $M/M/n$ queue.}
{
\begin{lemma}
    \label{lemma: van-doorn-beta-hat-bound}
    For an irreducible positive recurrent $M/M/n$ queue with arrival rate $\lambda_n$ and for any non-negative integers $i,j$,
    we have
    \begin{align}
        \hat{\beta} \geq \inf_{k \geq 1} \lambda_n + \min \{k, n\} - \sqrt{\lambda_n \min \{k-1, n\}} - \sqrt{\lambda_n \min \{k, n\}}.
    \end{align}
\end{lemma}}
\textit{Proof of Lemma \ref{lemma: van-doorn-beta-hat-bound}:}
    Since $M/M/n$ is a birth-and-death process with birth rate $\lambda_n$ and death rate $\mu_k = \min \{k, n\}$, we can apply Theorem 4.1 in \citep{Van_Doorn1985-decay-bounds} to obtain the desired result. $\square$
\endproof
{Now, our quest is to provide a lower bound on $\hat{\beta}$ and its asymptotic properties using this lemma. Inspired by \citep{Halfin-Whitt-Gamarnik_2013}, we have the following estimation on the exponential decay parameter of the $M/M/n$ queue.
\begin{lemma}
    \label{lemma: mmn-beta-hat-bound-from-goldberg}
    Let $\hat{\beta}$ be the exponential decay parameter of the continuous-time $M/M/n$ system with arrival rate $\lambda_n$ and unit service rate, we have $\hat{\beta} \geq \min\left\{\frac{1}{2}\sqrt{\frac{\lambda_n}{n}}, (\sqrt{n}-\sqrt{\lambda_n})^2\right\}$. Moreover, when $\lambda_n = n-n^{1-\alpha}$ for $\alpha \in (0, 1/2)$ or when $\lim_{n \rightarrow \infty} \lambda_n = \infty$ and $\lim_{n \rightarrow \infty} \frac{\lambda_n}{n} = c$ for $c \in [0,1)$, we have $\liminf_{n \rightarrow \infty} \hat{\beta} \geq \frac{1}{2}$.
\end{lemma}
\proof{\textit{Proof of Lemma \ref{lemma: mmn-beta-hat-bound-from-goldberg}}:}
    From Lemma \ref{lemma: van-doorn-beta-hat-bound}, $\hat{\beta}$ can be lower bounded as
    \begin{align}
        \nonumber
        \hat{\beta} &\geq \inf_{k \geq 1} \left\{\lambda_n + \min \{k, n\} - \sqrt{\lambda_n \min \{k-1, n\}} - \sqrt{\lambda_n \min \{k, n\}}\right\} \\
        &= \min\left\{\min_{1 \leq k \leq n} (\sqrt{\lambda_n}-\sqrt{k})^2 + \frac{\sqrt{\lambda_n}}{\sqrt{k} + \sqrt{k-1}}, (\sqrt{n}-\sqrt{\lambda_n})^2\right\}.
    \end{align}
    Note that for $1 \leq k \leq n$, we have $\sqrt{k}+\sqrt{k-1} \leq 2 \sqrt{n}$. And so,
    \begin{align}
        \underbrace{(\sqrt{\lambda_n}-\sqrt{k})^2}_{\geq 0} + \frac{\sqrt{\lambda_n}}{\sqrt{k} + \sqrt{k-1}} \geq \frac{1}{2} \sqrt{\frac{\lambda_n}{n}}.
    \end{align}
    This implies
    \begin{align}
        \label{eqn: goldberg-beta-lower-bound}
        \hat{\beta} \geq \min\left\{\frac{1}{2}\sqrt{\frac{\lambda_n}{n}}, (\sqrt{n}-\sqrt{\lambda_n})^2\right\}.
    \end{align}
    Now, we will further show that the lower bound on $\hat{\beta}$ can be further improved to $\frac{1}{2}$ when $n \rightarrow \infty$. Indeed, denote $f(k) =  (\sqrt{\lambda_n}-\sqrt{k})^2 + \frac{\sqrt{\lambda_n}}{\sqrt{k} + \sqrt{k-1}}$ and $f_n^* = \min_{1 \leq k \leq n} (\sqrt{\lambda_n}-\sqrt{k})^2 + \frac{\sqrt{\lambda_n}}{\sqrt{k}+\sqrt{k-1}}$, we have
    \begin{align*}
        f'(x) &= 1 - \frac{\sqrt{\lambda_n}}{2\sqrt{x}} - \frac{\sqrt{\lambda_n}}{2\sqrt{x-1}}, \\
        f''(x) &= \frac{\sqrt{\lambda_n}}{4\sqrt{x^3}} + \frac{\sqrt{\lambda_n}}{4\sqrt{(x-1)^3}} > 0.
    \end{align*}
    Hence, the optimal point $x^* \in (1,n)$ is also the unique minimizer in this range. Note that $f'(\lambda_n) < 0 < f'(\lambda_n+1)$, and so $\lambda_n < x^* < \lambda_n+1$. {This implies
    \begin{align}
         \max\{f(\lambda_n), f(\lambda_n+1)\}\geq f_n^* &\geq \min_{\lambda_n \leq k \leq \lambda_n+1} \frac{\sqrt{\lambda_n}}{\sqrt{k}+\sqrt{k-1}} \geq \frac{1}{2} \frac{\sqrt{\lambda_n}}{\sqrt{\lambda_n+1}}.
    \end{align}}
    Therefore,
    \begin{align}
        \lim_{n \rightarrow \infty} f_n^* &\leq \lim_{n \rightarrow \infty}  \max\left\{\frac{\sqrt{\lambda_n}}{\sqrt{\lambda_n} + \sqrt{\lambda_n-1}}, (\sqrt{\lambda_n+1}-\sqrt{\lambda_n})^2 + \frac{\sqrt{\lambda_n}}{\sqrt{\lambda_n+1} + \sqrt{\lambda_n}} \right\} &= \frac{1}{2}, \\
        \lim_{n \rightarrow \infty} f_n^* &\geq \lim_{n \rightarrow \infty}  \frac{1}{2} \frac{\sqrt{\lambda_n}}{\sqrt{\lambda_n+1}} &= \frac{1}{2},
    \end{align}
    since $\lim_{n \rightarrow \infty} \lambda_n = \infty$. By the Sandwich Theorem, we have $\lim_{n \rightarrow \infty} f_n^* = \frac{1}{2}$. Now, since $\lim_{n \rightarrow \infty} (\sqrt{n}-\sqrt{\lambda_n})^2 = \infty$ in the Sub-Halfin-Whitt regime, whereas $f_n^* = \min_{1 \leq k \leq n} (\sqrt{\lambda_n}-\sqrt{k})^2 + \frac{\sqrt{\lambda_n}}{\sqrt{k} + \sqrt{k-1}} = O(1)$. Thus, we have $\hat{\beta} \geq f_n^*$ for a sufficiently large $n$. This means that
    \begin{align}
        \label{eqn: beta-hat-limit}
        \liminf_{n \rightarrow \infty} \hat{\beta} \geq \frac{1}{2}
    \end{align}
    when $\alpha \in (0,1/2)$ or $\lim_{n \rightarrow \infty} \frac{\lambda_n}{n} = c$. $\square$
\endproof
While the first two lemmas established exponential ergodicity and bounds on the decay parameter, these are not sufficient to get finite-time results because they are asymptotic in nature. And so, to obtain Chi-square convergence bounds, {we have the following result by \citep{Mu-Fa-Chen1991-l2-convergence}}.
\begin{lemma}[Theorem 5.3 of \citep{Mu-Fa-Chen1991-l2-convergence}]
    \label{lemma: l2-equivalent-expo-ergodicity-birth-and-death-process}
    Denote $\operatorname{Gap}(\cL) = \inf_{\substack{f\in\ell_{2,\pi}\\Var_\pi(f)>0}} \frac{\langle f, -\cL f \rangle}{\Var_\pi(f)}$ as the spectral gap of the continuous-time Markov chain with the infinitesimal generator $\cL$ and steady state distribution $\pi$. For every positive recurrent birth-and-death process, we have $\hat{\beta} = \operatorname{Gap}(\cL)$, where $\hat{\beta}$ is the exponential decay parameter as defined in Definition \ref{def: exponential-decay-and-exponential-ergodicity}.
\end{lemma}
From here, we obtain the following Chi-square convergence result that is inspired by the results in \citep{Van_Doorn1985-decay-bounds, Halfin-Whitt-Gamarnik_2013}.
\begin{prop}
    \label{prop: l2-convergence-from-van-doorn}
    Let $\pi_{n,t}, \nu_n$ be the queue length distribution at time $t$ and the steady state distribution of the continuous-time $M/M/n$ system with arrival rate $\lambda_n \in (0,n)$ and unit service rate, respectively. We have
    \begin{align*}
        \chi(\pi_{n,t},\nu_n) \leq e^{-\zeta_n t} \chi(\pi_{n,0},\nu_n) \, \forall t \geq 0
    \end{align*}
    for some positive parameter $\zeta_n$ such that $\zeta_n \geq \min \left\{(\sqrt{n}-\sqrt{\lambda_n})^2, \frac{1}{2} \sqrt{\frac{\lambda_n}{n}}\right\}$. Moreover, when $\lambda_n = n-n^{1-\alpha}$ for $\alpha \in (0, 1/2)$ or when $\lim_{n \rightarrow \infty} \lambda_n = \infty$ and $\lim_{n \rightarrow \infty} \frac{\lambda_n}{n} = c$ for $c \in [0,1)$, we have $\liminf_{n \rightarrow \infty} \zeta_n \geq \frac{1}{2}$.
\end{prop}}
\proof{\textit{Proof of Proposition \ref{prop: l2-convergence-from-van-doorn}}:}
    Since we have $M/M/n$ with arrival rate $\lambda_n \in (0,n)$ and unit service rate is irreducible positive recurrent, we have that the system is exponentially ergodic, and there exists an exponential decay parameter $\hat{\beta}$ such that
    \begin{align}
        |P_{ij}(t)-\nu_n(j)| = O\br{e^{-\hat{\beta} t}} \, \forall i,j \in \cS.
    \end{align}
    From Lemma \ref{lemma: mmn-beta-hat-bound-from-goldberg}, we have $\hat{\beta} \geq \min\left\{\frac{1}{2}\sqrt{\frac{\lambda_n}{n}}, (\sqrt{n}-\sqrt{\lambda_n})^2\right\}$ and $\liminf_{n \rightarrow \infty} \hat{\beta} \geq \frac{1}{2}$ when $\lambda_n = n-n^{1-\alpha}$ for $\alpha \in (0,1/2)$. From Lemma \ref{lemma: l2-equivalent-expo-ergodicity-birth-and-death-process} and Proposition \ref{prop: poincare-equals-mixing}, this implies that the system admits the Chi-square convergence bound
    \begin{align}
        \chi(\pi_{n,t},\nu_n) \leq e^{-\zeta_n t} \chi(\pi_{n,0},\nu_n) \, \forall t \geq 0,
    \end{align}
    for some $\zeta_n \geq \min \left\{(\sqrt{n}-\sqrt{\lambda_n})^2, \frac{1}{2} \sqrt{\frac{\lambda_n}{n}}\right\}$ and $\liminf_{n \rightarrow \infty} \zeta_n \geq \frac{1}{2}$ for when $\alpha \in (0,1/2)$. $\square$
\endproof
{In comparison with our results, Proposition \ref{prop: l2-convergence-from-van-doorn} provides a weaker asymptotic result in the Sub-Halfin-Whitt regime as the mixing rate converges to $1/2$ whereas our generalized Lyapunov-Poincar\'e method gives a mixing rate that converges to $1$ at the limit and matches the mixing rate of the limiting system $M/M/\infty$. For the Super-Halfin-Whitt regime, since $(\sqrt{n}-\sqrt{\lambda_n})^2 = \Theta(n^{1-2\alpha}) \ll \frac{1}{2} \sqrt{\frac{\lambda_n}{n}}$ for a sufficiently large $n$, both methods obtain the $(\sqrt{n}-\sqrt{\lambda_n})^2$ mixing rate at the limit, which is tight according to \cite{Van_Doorn1985-decay-bounds}.}

\subsection{Comparison with \texorpdfstring{\citep{zeifman_lognorm}}{Zeifman et al.}}

{Another approach is to use log-norms \citep{zeifman_lognorm}, which provides a convergence to stationarity bound in terms of TV distance. In this work, the author obtains the mixing rate $(\sqrt{n}-\sqrt{\lambda_n})^2$ for $\rho_n = \frac{\lambda_n}{n} \geq \br{1-\frac{1}{n}}^2$, which corresponds to the regime $\alpha \geq 1$ in our work. For $\alpha < 1$, the author provides the estimate $1-\frac{\lambda_n}{n-1} = \Theta\br{n^{-\alpha}}$ for the mixing rate. For comparison, we obtain Chi-square convergence bounds that can be used to obtain the TV distance bounds. In particular, we obtain the mixing rate $C_n (\sqrt{n}-\sqrt{\lambda_n})^2 = \Theta\br{n^{1-2\alpha}}$ for $\alpha \in (1/2,1)$, which is better than the $\Theta(n^{-\alpha})$ for $\alpha < 1$. Additionally, we also obtain the mixing rate $\Omega(1)$ in the Sub-Halfin-Whitt regime (i.e. $\alpha \in (0,1/2)$), which matches the order of the $M/M/\infty$ by a constant factor, whereas the $\Theta(n^{-\alpha})$ mixing rate by \citep{zeifman_lognorm} vanishes to $0$ as $n \rightarrow \infty$.}

\section{Miscellaneous results}
\label{sec: misc-results}
In this Section, we provide the proof of some miscellaneous technical results.

\begin{lemma}
    \label{lemma: taylor-expansion-bound}
    Given $x \in \R$ such that $|x| \leq 1$, we have
    \begin{align}
        1+x \leq e^x \leq 1+x+x^2e^x.
    \end{align}
\end{lemma}

\proof{\textit{Proof of Lemma \ref{lemma: taylor-expansion-bound}}:}
    The LHS is a well known inequality. For the RHS, consider $f(x) = e^x - (1+x+x^2 e^x)$, we want to show that $f(x) \leq 0 \, \forall |x| \leq 1$. Factorize $f$, we have
    \begin{align*}
        f(x) = e^x - (1+x+x^2 e^x) = (1+x)\sqbr{e^x(1-x)-1}.
    \end{align*}
    Since $|x| \leq 1$, we have $1+x \geq 0$ so it is sufficient to show that $g(x) = e^x(1-x)-1 \leq 0$. Note that $g'(x) =  -e^x x$ and $g''(x) = -e^x(x+1)$, we have $g'(0) = 0$ and $g''(x) \leq 0 \, \forall x \geq -1$. This means that $g$ attains maximum at $x = 0$, which implies $g(x) \leq g(0) = 0$. Hence proved.
$\square$ \endproof

\begin{lemma}
    \label{lemma: log-lower-bound}
    Given $x \in \R$ such that $0 \leq x \leq \frac{1}{2}$, we have
    \begin{align}
        e^x \leq 1+x+x^2.
    \end{align}
    Moreover, for $0 \leq x \leq \frac{1}{2}$, we also have
    \begin{align}
        \frac{x}{2} &\leq \log \br{x+1}.
    \end{align}
    For $0 \leq x \leq 1$, we have
    \begin{align}
        x-\frac{x^2}{2} \leq \log \br{x+1}.
    \end{align}
    For $-1 \leq x \leq 0$, we have
    \begin{align}
        \log \br{x+1} \leq x-\frac{x^2}{2}.
    \end{align}
\end{lemma}

\proof{\textit{Proof of Lemma \ref{lemma: log-lower-bound}}:}
    Recall that for $0 \leq x \leq 1/2$, we have
    \begin{align*}
        e^x &= \sum_{k=0}^\infty \frac{x^k}{k!} = 1+x+x^2 - \frac{x^2}{2} + \sum_{k=3}^\infty \frac{x^k}{k!} \\
        &\leq 1+x+x^2 - \frac{x^2}{2} + \frac{x^3}{6} \br{\sum_{k=0}^\infty x^k} \\
        &\leq 1+x+x^2 - \frac{x^2}{2} + \frac{x^3}{3} \leq 1+x+x^2.
    \end{align*}
    Furthermore, note that $x \geq 0$, and so
    \begin{align*}
        e^x &\leq 1+x+x^2 \leq (x+1)^2 \\
        \Leftrightarrow x &\leq 2 \log (x+1) \\
        \Leftrightarrow \frac{x}{2} &\leq \log (x+1).
    \end{align*}
    On the other hand, let $f(x) = \log (x+1) - \br{x-\frac{x^2}{2}}$, we have $f'(x) = \frac{1}{x+1} - 1 + x \geq 0 \, \forall x \in [-1,1]$. Thus, we have $f(x) \geq f(0) = 0$ for $x \in [0,1]$, which implies
    \begin{align*}
        x - \frac{x^2}{2} \leq \log (x+1) \, \forall x \in [0,1].
    \end{align*}
    On the other hand, we have $f(x) \leq f(0) = 0$ for $x \in [-1,0]$, which implies
    \begin{align*}
        \log (1+x) \leq x - \frac{x^2}{2} \, \forall x \in [-1,0].
    \end{align*}
    Hence proved.
$\square$ \endproof

\begin{lemma}
    \label{lemma: mgf-above-n-asymptotic}
    Let $\delta$ be a given positive real number, $n \geq 2$ be a positive integer and denote $\lambda = n-n^{1-\alpha}$ where $\alpha \in (0,+\infty)$. We have
    \begin{align}
        \frac{1}{n^\alpha \br{1-\br{\frac{\lambda}{n}}^{\frac{\delta}{\delta+1}}}} \leq 1 + \frac{1}{\delta}.
    \end{align}
    Moreover, we also have
    \begin{align}
        \lim_{n \rightarrow \infty} \frac{1}{n^\alpha \br{1-\br{\frac{\lambda}{n}}^{\frac{\delta}{\delta+1}}}} = 1 + \frac{1}{\delta}.
    \end{align}
\end{lemma}

\proof{\textit{Proof of Lemma \ref{lemma: mgf-above-n-asymptotic}}:}
    To prove the inequality, observe that {
    \begin{align*}
        \frac{1-\frac{\lambda}{n}}{\br{1-\br{\frac{\lambda}{n}}^{\frac{\delta}{\delta+1}}}} = \frac{1}{n^\alpha \br{1-\br{\frac{\lambda}{n}}^{\frac{\delta}{\delta+1}}}} &= \frac{1}{n^\alpha \br{1-\br{1-n^{-\alpha}}^{\frac{\delta}{\delta+1}}}}.
    \end{align*}
    Denote $x = n^{-\alpha}, r = \frac{\delta}{\delta+1} \in (0,1)$, we have the inequality is reduced to
    \begin{align*}
        \frac{x}{1-(1-x)^r} \leq \frac{1}{r} \Leftrightarrow (1-x)^r \leq 1-rx.
    \end{align*}
    This is true by Bernoulli's inequality, and thus the inequality is established.}

    \noindent
    For the limit, we have
    \begin{align*}
        \lim_{n \rightarrow \infty} \frac{1}{n^\alpha \br{1-\br{\frac{\lambda}{n}}^{\frac{\delta}{\delta+1}}}} &= \lim_{n \rightarrow \infty} \frac{1}{n^\alpha \br{1-\br{1-n^{-\alpha}}^{\frac{\delta}{\delta+1}}}} \\
        &\overset{(a)}{=} \lim_{n \rightarrow \infty} \frac{1}{n^\alpha \br{1-e^{-\frac{\delta}{\delta+1} n^{-\alpha}}}} \\
        &\overset{(b)}{=} \frac{\delta+1}{\delta} \lim_{n \rightarrow \infty} \frac{\frac{\delta}{\delta+1} n^{-\alpha}}{\br{1-e^{-\frac{\delta}{\delta+1} n^{-\alpha}}}} \\
        &= \frac{\delta+1}{\delta}
    \end{align*}
    where $(a)$ follows from $\lim_{x \rightarrow \infty} \br{1-\frac{1}{x}}^x = \frac{1}{e}$ and $(b)$ follows from the L'Hopital rule.
$\square$ \endproof

\end{document}